\documentclass[11pt, reqno]{amsart}
\RequirePackage[OT1]{fontenc}
\usepackage[colorlinks=true, urlcolor=blue, linkcolor=blue, citecolor=blue, pdftex]{hyperref}
\usepackage{csquotes}
\topmargin - 1 cm
\usepackage[utf8]{inputenc}
\usepackage[matrix,arrow,curve]{xy}
\usepackage{amsmath,amsfonts,verbatim,afterpage,euscript,mathrsfs,amssymb}
\usepackage{array}
\usepackage{amsthm}
\usepackage{dsfont}
\usepackage{hyperref}
\usepackage{enumerate}
\usepackage[english]{babel}
\usepackage{tikz}
\usetikzlibrary{shapes}
\usepackage{titlesec}
\usepackage{algorithmicx}
\usepackage{algpseudocode}
\usepackage{algorithm}

\titleformat{\section}[block]{\normalfont\bfseries\filcenter}{\itshape\thesection}{1em}{}
\titleformat{\subsection}[block]{\normalfont\bfseries}{\itshape\thesubsection}{0.9em}{}
\titleformat{\subsubsection}[block]{\normalfont\bfseries}{\itshape\thesubsubsection}{0.8em}{}
\titleformat{\caption}[block]{\normalfont}{\itshape}{0.8em}{}

\DeclareRobustCommand{\EX}[2][{\mathbb{E}}]{\ensuremath {#1} \left[{#2}\right]}

\DeclareRobustCommand{\cEX}[3][{\mathbb{E}}]{\ensuremath {#1}\left[ {#2} \big| {#3} \right]}

\DeclareRobustCommand{\one}[1]{\ensuremath \mathbf{1}_{\{#1\}}}

\DeclareRobustCommand{\F}{\ensuremath \mathcal{F}}
\DeclareRobustCommand{\RR}{\ensuremath \mathbb{R}}

\DeclareRobustCommand{\NN}{\ensuremath \mathbb{N}}
\DeclareRobustCommand{\EE}{\ensuremath \mathbb{E}}

\newcommand{\ApproxEx}[2][T]{\mathbb{E}_{\mathbb{Q}_{#1}} \left[{#2}\right]}
\newcommand{\ApproxVarX}[1]{\hat{\ensuremath{#1}}}
\newcommand{\ApproxVarT}[1]{\hat{\ensuremath{#1}}^{\pi,1}}
\newcommand{\ApproxTwoVarT}[1]{\hat{\ensuremath{#1}}^{\pi,2}}

\newcommand{\PredictorVarT}[1]{\tilde{\ensuremath{#1}}}

\newcommand{\mc}[1]{\ensuremath{\mathcal{#1}}}

\newcommand{\LB}{\textbf{(LB) }}
\newcommand{\SB}{\textbf{(SB) }}

\newcommand{\bV}{\bar{V}}

\newcommand{\bX}{\bar{X}}
\newcommand{\bY}{\bar{Y}}
\newcommand{\bZ}{\bar{Z}}
\newcommand{\bu}{\bar{u}}
\newcommand{\bTheta}{\bar{\Theta}}

\newcommand{\tu}{\tilde{u}}

\newcommand{\tX}{\tilde{X}}
\newcommand{\hX}{\hat{X}}
\newcommand{\hmu}{\hat{\mu}}
\newcommand{\hu}{\hat{u}}
\newcommand{\hv}{\hat{v}}
\newcommand{\hzeta}{\hat{\zeta}}
\newcommand{\hTheta}{\hat{\Theta}}

\newcommand{\E}{\mathbb{E}}

\newcommand{\R}{\mathbb{R}}
\newcommand{\Q}{\mathbb{Q}}

\newcommand{\mL}{\mathcal{L}}

\newcommand{\mE}{\mathcal{E}}
\newcommand{\mF}{\mathcal{F}}
\newcommand{\p}{\partial}

\renewcommand{\mE}{\mathcal{E}}
\renewcommand{\p}{\partial}
\renewcommand{\t}{\underline{t}}
\newcommand{\s}{\underline{s}}

\newcommand{\la}{\langle}
\newcommand{\ra}{\rangle}

\numberwithin{equation}{section}

\newtheorem{theoreme}{Theorem}[section]

\newtheorem{corollaire}[theoreme]{Corollary}
\newtheorem{lemme}[theoreme]{Lemma}
\newtheorem{definition}[theoreme]{Definition}
\newtheorem*{remarque}{Remark}

\newtheorem{claim}[theoreme]{Claim}

\begin{document}

\title[A cubature based algorithm to solve MKV-FBSDE]{A cubature based algorithm to solve decoupled McKean-Vlasov Forward Backward Stochastic Differential Equations}

\author{P.E. Chaudru de Raynal}
\address{UNIVERSITE NICE SOPHIA ANTIPOLIS, LABORATOIRE JEAN-ALEXANDRE DIEUDONNE\\ PARC VALROSE, 06108 CEDEX 02, NICE, FRANCE}
\email[P.E. Chaudru de Raynal]{deraynal@unice.fr}

\author{C.A. Garcia Trillos}
\email[C.A. Garcia Trillos]{camilo@unice.fr}
\begin{abstract}
We propose a new algorithm to approach weakly the solution of a McKean-Vlasov SDE. Based on the cubature method of Lyons and Victoir \cite{lyons_cubature_2004}, the algorithm is deterministic differing from the the usual methods based on interacting particles. It can be parametrized in order to obtain a given order of convergence. 

Then, we construct implementable algorithms to solve decoupled Forward Backward Stochastic Differential equations (FBSDE) of McKean-Vlasov type, which appear in some stochastic control problems in a mean field environment. We give two algorithms and show that they have convergence of order one and two under appropriate regularity conditions.\\

\footnotesize{This is a reprint version of an article published in \href{http://www.sciencedirect.com/science/article/pii/S0304414914003044?via\%3Dihub}{Stochastic Processes and their Applications 125 (2015) 2206–2255, doi:10.1016/j.spa.2014.11.018}}
\end{abstract}

\maketitle

\section{Introduction}
We call decoupled McKean-Vlasov forward backward stochastic differential equation (MKV-FBSDE) the following FBSDE system:
\begin{equation}\label{NLFBSDE}
\left\lbrace\begin{array}{ll}
dX^x_t=\sum_{i=0}^dV_i(t,X^x_t,\E \varphi_i(X^x_t))  dB^i_t\\
dY^x_{t}  = -f(t,X^x_t,Y^x_t,Z^x_t,\E \varphi_f(X^x_t,Y^x_t))dt + Z^x_t dB_t^{1:d}\\
X^x_0 = x,\quad Y^x_{T}  = \phi(X^x_T)
\end{array}
\right.
\end{equation}
for any $t$ in $[0,T]$, $T>0$ be given. We place ourselves in a filtered probability space $(\Omega, \mathcal{F}, \mathbb{P}, (\mathcal{F}_t)_{ t \geq 0})$, with $B_t^{1:d}$ a $d$-dimensional adapted Brownian motion and $B_t^0=t$. We take $ V_i: (t,y,w)\in  [0,T]\times\R^d \times \R \mapsto V_i(t,y,w)$; functions $ \varphi_i: y\in \R^d \mapsto \varphi_i(y)\in \R,\ i=0,\cdots,d$ and $\varphi_f: (y, y') \in \R^d \times \R \mapsto \varphi_f(y,y')$ and the mapping $f: t,y,y',z,w \in [0,T]\times \R^d \times \R \times \R^d \times \R \mapsto f(t,y,y',z,w) \in \R$ to be infinitely differentiable with bounded derivatives. The mapping $\phi$ is an at least Lipschitz function from $\R^d$ to $\R$ whose precise regularity is given below.\\

McKean Vlasov processes may be regarded as a limit approximation for interacting systems with large number of particles. They appeared initially in statistical mechanics, but are now used in many fields because of the wide range of applications requiring large populations interactions. For example, they are used in finance, as factor stochastic volatility models \cite{bergomi_smile_2009} or uncertain volatility models \cite{guyon_smile_2011}; in economics, in the theory of ``mean field games'' recently developed by J.M. Lasry and P.L. Lions in a series of papers \cite{lasry_jeux_2006,lasry_jeux_2006-1,lasry_large_2007,lasry_mean_2007} (see also \cite{carmona_control_2012, carmona_probabilistic_2012, carmona_mean_2012} for the probabilistic counterpart) and also in physics, neuroscience, biology, etc. In section \ref{Sec:ControlMeanField}, we present a class of control problems in which equation \eqref{NLFBSDE} explicitly appears. \\

The note of Sznitman \cite{sznitman_topics_1991} gives a complete overview on the topic of systems with a large number of particles. A proof of the existence and uniqueness of the solution of a MKV-FBSDE system related but different to the one of our setup is found in \cite{buckdahn_mean-field_2009}. These existence and uniqueness results are easily extended to \eqref{NLFBSDE}.\\

\textbf{A cubature algorithm for MKV-FBSDE processes.} Cubature on Wiener space was introduced in 2004 by T.Lyons and N.Victoir \cite{lyons_cubature_2004}, following the earlier work of S. Kusuoka \cite{kusuoka_approximation_2001}.  Ever since, the cubature method has been used to solve the problem of calculating Greeks in finance \cite{teichmann_calculating_2006}, non-linear filtering problems \cite{crisan_convergence_2007}, stochastic partial differential equations \cite{bayer_cubature_2008}, \cite{doersek_cubature_2012} and backward stochastic differential equations \cite{crisan_second_2010,crisan_solving_2010}.\\ 
The main idea of the cubature method consists in replacing the Brownian motion by choosing randomly a path among an a priori (finite) set\footnote{Explicit examples of such functions are given in \cite{lyons_cubature_2004}. We put back this discussion to a next subsection.} of continuous functions from $[0,T]$ to $\R^d$ with bounded variations such that the expectation of the iterated integrals against both the Brownian and such paths are the same, up to a given order $m$. Hence, the SDE is replaced by a system of weighted ODEs.\\ 

 We give the main idea to construct a cubature based approximation scheme for \eqref{NLFBSDE}. The main issue in the case of a MKV-FBSDE is the McKean-Vlasov dependence that appears in the coefficients. This dependence breaks the Markov property (considered only on $\R^d$) of the process so that it is not possible to apply, a priori, many classical analysis tools. In order to handle this problem, the idea consists in taking benefit on the following observation: \emph{given the law of the solution of the system, \eqref{NLFBSDE} is a classical time inhomogeneous FBSDE} (the law just acts as a time dependent parameter).\\

Let $(\eta_t)_{0\leq t \leq T}$ be a family  of probability measures on $\R^d$, and let us fix the law in the McKean-Vlasov terms of \eqref{NLFBSDE} to be $(\eta_t)_{0\leq t \leq T}$ . For this modified system, we may apply a classical cubature FBSDE scheme for the forward component (the time dependence of the coefficients being handled as an additional dimension). The trick consists in taking advantage of the decoupled setting:  we first build a cubature tree (depending on the order of the cubature) and then go back along the nodes of the tree by computing the current value of the backward process as a conditional expectation at each node. We refer to \cite{crisan_solving_2010} or \cite{crisan_second_2010} for a detailed description of such algorithm.\\
 
Obviously, at each step of the scheme, we pay the price of using an arbitrary probability measure as parameter for the coefficients instead of the law of the process. Therefore this law has to be chosen carefully in order to keep a good control on the error and achieve convergence. An example of a ``good choice'' is to take at each step of the cubature tree the discrete marginal law given by the solution of the ODEs along the cubature paths and corresponding weights. We show that for a cubature of order $m$ and a number $N$ of discretization steps, this choice of approximation law leads to a $N^{-(m-1)/2}$ order approximation of the expectation of any $m+2$ times continuously differentiable functional of the forward component, when all the derivatives are bounded\footnote{this is the special case when $\phi$ is $m+2$ times continuously differentiable with bounded derivatives and $f=0$ in \eqref{NLFBSDE}.}, and to a first order approximation scheme of the backward component, where the given orders  stand for 
the supremum norm error. Higher orders of approximation are also obtained by correcting some terms in the algorithm.\\

As it is pointed out in \cite{lyons_cubature_2004} and \cite{crisan_second_2010}, the regularity of the terminal condition $\phi$ in \eqref{NLFBSDE} may be relaxed to Lipschitz and the approximation convergence rate preserved, provided that the vector fields are uniformly non-degenerate (in fact, the condition in the given references is weaker, since the vector fields are supposed to satisfy an UFG condition, see \cite{kusuoka_applications_1987}). This relies on the regularization properties of parabolic and semi-linear parabolic PDEs (see \cite{friedman_partial_2008}  for an overview in the elliptic case and respectively \cite{kusuoka_applications_1987} and \cite{crisan_sharp_2012} for the UFG case). We show that this remains true in the McKean-Vlasov case and that the convergence rate still holds when the function $\phi$ is Lipschitz only and when the vector fields are uniformly elliptic. \\

Usually, forward MKV-SDEs are solved by using particle algorithms (see for example \cite{antonelli_rate_2002,Talay_stochastic_2003} or \cite{bossy_stochastic_2005} and references therein) in which the McKean term is approached with the empirical measure of a large number of interacting particles with independent noise. Adapting such algorithms to the forward-backward problem is not obvious as the high dimension of the involved Brownian motion (given by the number of particles) induces, a priori, a high dimension backward problem with the obvious consequences for the numerical implementation. In comparison, our proposed algorithm gives a deterministic approximation of the McKean term, and since it does not induce any additional noise, it does not increase the dimension of the backward problem.

Although our algorithm works for decoupled MKV-FBSDEs, we believe this solver may also be considered as a building block if one is interested in approaching the fully coupled case (when the forward coefficients depend on the backward variable), for example via fixed point procedures. Nevertheless, a lot of work is required to define the precise conditions and setup in which such algorithm would converge to the (or at least a) solution of the fully coupled problem.\\

\textbf{The conditional system.} Let us shortly develop what we mean with the sentence ``\emph{given the law of the solution of the system, \eqref{NLFBSDE} is a classical time inhomogeneous FBSDE}".\\
Working with a non-linear problem, such as MKV-FBSDE, could be tricky. In our case, the main object to work with is the \emph{conditional system}. This is the formulation that allows to get rid of the dependence on the law and to replace it by a time dependent parameter.\\

Following the same line of arguments presented in Buckdahn et al. \cite{buckdahn_mean-field_2009}, it is well seen that there exists a unique solution $\{X_t^x,Y_t^x\}_{t\geq 0}$ to the system \eqref{NLFBSDE}. In a Markovian setting, the law of this couple is entirely determined by the law $\mu=(\mu_t)_{0\le t\leq T}$ of the forward process $(X_t)_{0\leq t\leq T}$ and a given deterministic function $u:[0,T]\times \R^d \to \R$. In our case, one can show that this remains true (see Section \ref{Sec:MathTools} below for a proof) so that there exists a deterministic $u(t,y)$ such that for all $t$,
\begin{equation} \label{Eq:DefU}Y_t = u(t,X_t).\end{equation}
We prove that under appropriate assumptions $u$ is regular and satisfies the parametrized non-local semi linear PDE:
\begin{equation}\label{Eq:NonLinearPDE}
\left\lbrace
\begin{array}{ll}
\p_t u(t,y)  + \mL^{\mu} u(t,y) =f\left(t,y,u(t,y), (\mc{V}^{\mu} u(t,y))^T,\la\mu_t,\varphi_f(\cdot,u(t,\cdot))\ra\right)\\
 u(T,y)=\phi(y)
\end{array}
\right.,
\end{equation}
where $\mc{V}^{\mu} u$ stands for the row vector $(\nabla u \cdot V_1 ,\ldots,\nabla u \cdot V_d )$,  $(\mc{V}^\mu u)^T$ is the transpose of $\mc{V}^\mu u$ and $\mL^{\mu}$ is the generator of the forward component in \eqref{cFBSDE} below and given by:
\begin{equation}\label{defmL}
\mL^{\mu} := V_0(\cdot,\cdot,\la\mu_\cdot,\varphi_0\ra)\cdot D_y +\frac{1}{2}{\rm Tr}[VV^T(\cdot,\cdot,\la\mu_\cdot,\varphi_i\ra)D_y^2].
\end{equation}
Here, we used the duality notation $\la\mu,\varphi_i\ra$ for $\int \varphi_i d\mu$. Likewise, the superscript $\mu$ means that the vector fields are taken at the point $(\cdot,\cdot,\la\mu_{\cdot},\varphi_{i}\ra)$ (where the $i\in\{0,\cdots,d\}$ is taken with respect to the corresponding vector field), $V$ is the matrix $[V_1,\cdots,V_d]$,  $``\cdot$'' stands for the euclidean scalar product on $\R^d$ and ``${\rm Tr}$'' for the trace.\\

The conditional MKV-FBSDE system is then defined as
\begin{equation}\label{cFBSDE}
\left\lbrace\begin{array}{ll}
dX_s^{t,y,\mu}=\sum_{i=0}^dV_i(s,X_s^{t,y,\mu}, \la\mu_s, \varphi_i\ra) dB^i_s\\
dY_{s}^{t,y,\mu}  = -f(s,X_s^{t,y,\mu},Y_s^{t,y,\mu},Z_s^{t,y,\mu},\la\mu_s, \varphi_f(\cdot,u(s,\cdot))\ra)ds + Z_s^{t,y} dB_s^{1:d}\\
X_t^{t,y,\mu} = y,\quad Y_{T}^{t,y,\mu}  = \phi(X^{t,y,\mu}_T).
\end{array}
\right.
\end{equation}

Let us remark that in this setting, we can also define the deterministic mapping $v:[0,T]\times \R^d \to \R$ such that
\begin{equation} \label{Eq:DefV}Z_t = v(t,X_t).\end{equation}
As in the classical BSDE theory, under appropriate regularity conditions,
\[v(t,x) = (\mc{V}^{\mu}u(t,x))^T.\]

\textbf{Assumptions}. As the reader might guess from the previous discussion, the error analysis of the proposed algorithm uses extensively the regularity of $(t,x)\in [0,T]\times \R^d \mapsto \E \phi(X^{t,x,\mu}_T) $ for the forward part and of the solution $u$ of \eqref{Eq:NonLinearPDE} for the backward part. Therefore, we present two types of hypotheses that guarantee that the required regularity is attained. 

The first option we present is to require smoothness on the boundary condition and all the coefficient functions, from where we will deduce the necessary regularity. However, it is also interesting to consider boundary conditions with less regularity. In this case,  we need to compensate the regularity loss by imposing stronger diffusion conditions on the forward variable, namely asking for uniform ellipticity of the diffusion matrix $V$.

\begin{description}
\item[(SB)] We say that assumption \textbf{(SB)} holds if the mapping $\phi$  in \eqref{NLFBSDE} is $C_b^\infty$.
\item [(LB)] We say that assumption \textbf{(LB)} holds if the mapping $\phi$  in \eqref{NLFBSDE} is uniformly Lipschitz continuous and if the matrix $VV^T$ is uniformly elliptic i.e., there exists $c>0$ such that
\[\forall (t,y,w)\in [0,T]\times \R^d \times \R, \forall \varsigma \in \R^d ,  c^{-1} |\varsigma |^2\leq VV^T(t,y,w) \varsigma \cdot \varsigma \leq  c |\varsigma |^2.\]
\end{description}

\begin{remarque}
A reader familiarized with the cubature method might wonder why we assume uniform ellipticity instead of the weaker UFG condition usually needed for applying the method with Lipschitz boundary conditions. The reason is that the smoothing results of Kusuoka and Stroock \cite{kusuoka_applications_1987} hold for space dependent vector fields only, and therefore do not apply directly to our framework with a time dependence coming from the McKean term. There is some extension in the time inhomogeneous case that do not include derivatives in the $V_0$ direction (see for example \cite{cattiaux_hypoelliptic_2002} and references therein), but, to the best of our knowledge, there is no result that could be applied  to our framework.
\end{remarque}

\textbf{Towards a more general class of coefficients.} We have chosen to work with the assumed explicit dependence of the coefficients with respect to the law, as it is very natural in practice. In fact as a consequence of our analysis, our algorithm works for a more general class of MKV-FBSDE, i.e. for a system written as
\begin{equation}\label{NLFBSDEG2}
\left\lbrace\begin{array}{ll}
dX^x_t=\sum_{i=0}^dV_i(t,X^x_t,\mu_t)  dB^i_t\\
dY^x_{t}  = -f(t,X^x_t,Y^x_t,Z^x_t,\mu^{X,Y}_t)dt + Z^x_t dB_t^{1:d}\\
X^x_0 = x,\quad Y^x_{T}  = \phi(X^x_T),
\end{array}
\right.
\end{equation}
where $\mu^{X,Y}=(\mu^{X,Y}_t)_{0\le t\leq T}$ denotes the joint law of $(X_t,Y_t)_{0\leq t\leq T}$, the coefficients $V_i$, $0\leq i \leq d$ (and $f$) are Lipschitz continuous with respect to an appropriately defined distance in the space of probability measures on $\R^d$ (respectively $\R^d\times \R$ ) which integrates the usual Euclidean norm. 
The distance we consider is defined by duality: let $\mF$ be a (sufficiently rich\footnote{It should contain the space of $1$-Lipschitz functions.}) class of functions (that will be detailed in the following). Then we define the distance $d_\mF$ between two probability measures on $\R^n$ by
\begin{equation}
d_\mF(\mu,\nu) = \sup_{\varphi \in \mF} |\la\varphi,\mu-\nu\ra|.\label{Eq:distance}
\end{equation}

In this decoupled case the Lipschitz property of the coefficients with respect to $d_\mF$ ensures the existence of a unique solution of \eqref{NLFBSDEG2} \footnote{by definition of $\mF$, the distance $d_\mF$ is less than or equal to the Wasserstein 2 distance. Then, one uses the same kind of arguments as in \cite{sznitman_topics_1991}.}. Then, we are able to analyze the convergence of our procedure in two different cases. When $\mF$ is the class of 1-Lipschitz functions, i.e. when the distance is the so-called Wasserstein-1 distance, and when the vector fields are uniformly elliptic our algorithm leads to an $N^{-1/2}$ order approximation\footnote{Recall that $N$ denotes the number of discretization steps.}. When $\mF$ is the class of $C^{\infty}_b$ functions we obtain an $N^{-1}$ order approximation without any ellipticity assumption on the diffusion matrix.

\textbf{Objectives and organization of this paper.} As a corollary of the discussion on the conditional system we can resume our objective as the approximation of $\E \phi(X^x_T)$, where $(X^x_t)_{0\leq t \leq T}$ is the solution of \eqref{NLFBSDE} and of $u$ satisfying \eqref{Eq:DefU}.\\

This paper is organized as follows: Section \ref{Sec:algo} states the algorithm, while the convergence rate of the forward and backward approximations is stated in Section \ref{Sec:MainResult}. Then, we give a numerical example for each set of hypotheses \SB and \LB in Section \ref{Sec:NumEx}. A class of control problems  is introduced in Section \ref{Sec:ControlMeanField}. The remainder of the paper is dedicated to the proof of the convergence. For the sake of simplicity, we first recall some definitions, basic facts and notations in Section \ref{Sec:Preliminaire}. The forward and backward convergence rates for regular boundary conditions are successively proved in Section \ref{Sec:ProofMR} and the common mathematical tools are given in Section \ref{Sec:MathTools}. Section \ref{Sec:Lipschitz} presents the extension to the Lipschitz boundary condition case and the announced generalization of the law dependence of the McKean terms.\\

\textbf{Notations.} As we are treating with objects exhibiting different dependences, the notation can become a bit heavy. For the sake of simplicity, we adopt the following conventions. We denote by $\varphi$ the function $\varphi=[\varphi_1,\cdots,\varphi_d]$. For two positive integers $i<j$, the notation ``$i:j$" means ``from index $i$ to $j$". For all $\varsigma \in \R^n,\ n \in \mathbb{N}$ the partial derivative $[\p/\p \varsigma]$ is denoted by $\p_\varsigma$. Let $g:\ y\in \R^d \mapsto g(y) \in \R$ be a $p$-continuously differentiable function. We set $||g||_{\infty,p}:=\max_{j\leq p}||\p^j_y g||_{\infty}$. We say that a function $g$ from $[0,T] \times \R^d \times \R^d$ is $C^{p}_b$, $p \in \NN^*$ if it is bounded and $p$-times continuously differentiable with bounded derivatives. We usually denote by $\eta$ (eventually with an exponent) a family of probability measures $(\eta_t)_{0\leq t\leq T}$ on $\R^d$. For such a family, we set $\mL^\eta$ to be the second order operator of the form \eqref{defmL} 
with $\eta$ instead of $\mu$. In general, we will work with the vector fields taken at the point $(\cdot,\cdot,\la\mu_{\cdot},\varphi_{i}\ra)$ (where the $i\in\{0,\cdots,d\}$ signals the corresponding vector field) and in general we will omit the explicit dependence on $\mu$ in the notation. In any case, we will mark the law dependence explicitly when needed, in particular when a dependence with respect to a different law appears. 

\section{Algorithm} \label{Sec:algo}

\textbf{Multi-index.} Multi-indices allow to easily manage differentiation and integration in several dimensions. Let 
\begin{equation}\label{Eq:defM} 
\mc{M} =  \{ \emptyset \} \cup \bigcup_{l\in\NN^*} \ \{0,1,\ldots,d\}^l,
\end{equation}
denotes the set of multi-indices where $\emptyset$ refers, for the sake of completeness, to the zero-length multi-index.  We define ``$*$'' to be the concatenation operator such that if $\beta^1=(\beta^1_1,\ldots,\beta^1_l)$ and $\beta^2=(\beta^2_1,\ldots,\beta^2_n)$ then $\beta^1*\beta^2 = (\beta^1_1,\ldots,\beta^1_l,\beta^2_1,\ldots,\beta^2_n)$.\\

\textbf{Cubature on Wiener Space.} In the introduction, we mentioned that the cubature method consists in replacing the Brownian path by choosing randomly a path $\omega$ among a finite subset $\{\omega_1,\cdots,\omega_\kappa\}$, $\kappa\in \mathbb{N}^*$, of $C^0_{{\rm bv}}([0,T],\R^d)$ (the set of continuous functions from $[0,T]$ to $\R^d$ with bounded variations) with probability $\lambda$ in $\{\lambda_1,\cdots,\lambda_\kappa\}  \subset \R^+$. We precise this notion with the definition given by Lyons and Victoir \cite{lyons_cubature_2004}:

\begin{definition}{}
Let $m$ be a natural number and $t\in \R^{+}$. A $m$-cubature formula on the Wiener space $C^{0}([0,t],\R^d)$ is a discrete probability measure $\Q_t$ with finite support on $C_{\mathrm{bv}}^{0}([0,t],\R^d)$ such that the expectation of the iterated Stratonovitch integrals of degree $m$ under the Wiener measure and under the cubature measure $\Q_t$ are the same, i.e., for all multi-index $(i_1,\cdots, i_l) \in \{1,\cdots,d\}^l$, $l\leq m$
\begin{equation*}
\E \int_{0<t_{1}<\cdots<t_{l}<t}  \circ dB_{t_{1}}^{i_1} \cdots  \circ dB_{t_{l}}^{i_l}= \E_{\Q_t} \int_{0<t_{1}<\cdots<t_{l}<t}   \circ dB_{t_{1}}^{i_1}\cdots  \circ dB_{t_{l}}^{i_l}=\sum_{j=1}^l \lambda_{j}\int_{0<t_{1}<\cdots<t_{l}<t}   d\omega^{i_1}_j (t_1)\cdots d\omega_{j}^{i_l}(t_l),\\
\end{equation*}
where ``$\circ$'' stands for the Stratonovich operator and $\omega_j^i$ for the $i^{{\rm th}}$ coordinate of the $j^{{\rm th}}$ path.
\end{definition}
As a direct consequence of the Taylor-Stratonovitch expansion, a cubature formula of degree $m$ is such that:
\begin{equation}
|(\E  - \E_{\Q_t})F(B^{1:d}_t)| \leq C t^{(m+1)/2} ||F||_{m+2,\infty},
\end{equation}
for all bounded and $m+2$ times continuously differentiable function $F$ with bounded derivatives. 

Of course this error control is not in general small, but the Markovian and scaling properties of the Brownian motion can be used to apply the cubature method iteratively in small subdivisions of the interval $[0,t]$ for which we have a good error control. 

Indeed, consider a cubature formula $\Q_1$ of order $m\in \mathbb{N}^*$ with support $\{\omega_1,\ldots,\omega_\kappa\}$ and corresponding weights $\{\lambda_1,\ldots,\lambda_\kappa\}$. For all $h>0$ and any $t\in [0,T-h]$, one can deduce a cubature measure $\Q_{t,t+h}$ of order $m$ with finite support on $C_{\mathrm{bv}}^{0}([t,t+h],\R^d)$ equal to $\{\tilde{\omega}_1,\cdots,\tilde{\omega}_\kappa\}$ with the same weights $\{\lambda_1,\ldots,\lambda_\kappa\}$ and where the paths are defined as $\tilde{\omega}_j: s\in [t,t+h] \mapsto \tilde{\omega}(s) = \sqrt{h}\omega_j((s-t)/h)$ for all $1\leq j \leq \kappa$.

Then, by virtue of the Markovian property, this subdivision leads to the construction of a tree which has $\kappa^k$ nodes (corresponding to the number of paths) at the $k^{{\rm th}}$ subdivision. Each path $\tilde{\omega}_{(i_1,\ldots,i_k)}$, where $(i_1,\ldots,i_k)$ stands for the trajectory of the path, has then a cumulate weight of the form $\Lambda_{(i_1,\ldots,i_k)} =\prod_{j=1}^{k} \lambda_{i_j}$, see the example and figure \ref{treecub} below.\\

\subsection{Main idea}
Take a subdivision of the time interval $[0,T]$ into $N\in \NN^*$ steps $0=T_0<\cdots<T_N=T$. The procedure can be decomposed in two parts: 
\begin{enumerate}
\item \textbf{Building the tree $\mathcal{T}$.} This part of our algorithm can be resumed as a combined Euler-cubature approach and can be divided in four steps.
\begin{enumerate}
\item First, freeze in the space of probability measures the law that appears in the coefficients of \eqref{NLFBSDE} (the choice of this measure is explained below). At step 0, this measure is the Dirac mass in the starting point.
\item Freeze, in time, the given deterministic measure: this is an Euler step. 
\item Apply the cubature method. This will produce a cloud of deterministic particles given by the solution of the resulting ODEs
\item At each step, construct a discrete measure coming from the obtained cloud of particles and their associated cumulative weights. This is the law to be used to approximate the law in the coefficients of \eqref{NLFBSDE}.
\end{enumerate}
The reader might guess that the order of approximation of such an algorithm is one, due to the Euler step. Hence, in order to obtain higher order, we expand the function that appears in the McKean-Vlasov part up to a certain order, which is denoted by ``$q$" in the sequel.
\item \textbf{The backward component}. The backward component of the algorithm runs by assigning the value of the function on the boundary (known from the definition of the equation), and then back-propagating its value thanks to
\begin{enumerate}
\item A discretization scheme for the backward approximation
\item The cubature measure for finding conditional expectations.
\end{enumerate}
As mentioned before, we present two versions of the algorithm, with convergence of order one and two. As will be clear in the definition, the change in convergence order requires the use of a different backward scheme and cubature order.
\end{enumerate}

\textbf{Example: one dimensional cubature of order $m=3$.} In this case, we may use a cubature formula with $\kappa=2$ paths given by $\{+t,-t\}$, and associated weights: $\{\lambda_1=1/2, \lambda_2=1/2\}$. Let us explain the idea behind the algorithm with an example for $2$ steps,  as shown in Figure \ref{treecub}. We initialize the tree at a given point $x$, and the law at $T_0$ as $\delta_x$.  Then, we find two descendants given as the solution of an ODE that uses the position $X$, the two cubature paths, and an approximated law using the information at time $0$. Each descendant will have a weight equal to the product of the weight of its parent times the weight given to the corresponding cubature path. Once all nodes at time $T_{1}$ are calculated, we obtain the discrete measure $\hmu_{T_{1}}$, the law approximation at time $T_{1}$. The process is then repeated for each node at time $T_1$ to reach the final time $T=T_2$.

Figure \ref{treecub} right illustrates the idea behind the backward approximation: the approximated function $\hu$ is defined first at the leaves of the constructed tree, and then back-propagates using the approximated law to obtain $\hu$ at previous times. The back-propagation is made by conditional expectation: average with respect to the weight of each cubature path.

\begin{figure}[h!]
\begin{center}
\begin{tikzpicture}[scale=0.8, every node/.style={transform shape}]
\tikzstyle{level 0}=[]
\tikzstyle{level 1}=[level distance=2.5cm, sibling distance=4cm]
\tikzstyle{level 2}=[level distance=3.5cm, sibling distance=2cm]
\tikzstyle{edge from parent}=[->,draw]
\node[level 0] {$x$} [grow'=right,sibling distance=3cm]
  child {node[level 1] {$\hX_{T_1}^{(1)},\lambda_1$}
         child {node[level 2] {$\hX_{T_2}^{(1,1)},\Lambda_{(1,1)} = \lambda_1\lambda_1$} 
                edge from parent node[above left] {}
               }
         child {node[level 2] {$\hX_{T_2}^{(1,2)},\Lambda_{(1,2)} = \lambda_1\lambda_2$}
                edge from parent node[below left] {}
               }
         edge from parent node[above left] {}
        }
  child {node[level 1] {$\hX_{T_1}^{(2)},\lambda_2$}
         child {node[level 2] {$\hX_{T_2}^{(2,1)},\Lambda_{(2,1)} =\lambda_2\lambda_1$}           
                edge from parent node[above left] {}
               }
         child {node[level 2] {$\hX_{T_2}^{(2,2)},\Lambda_{(2,2)} = \lambda_2\lambda_2$}
                edge from parent node[below left] {}
               }
         edge from parent node[below left] {}
        };
\end{tikzpicture}
\begin{tikzpicture}[scale=0.8, every node/.style={transform shape},level distance=2.5cm]
\tikzstyle{level 0}=[]
\tikzstyle{level 1}=[level distance=2.5cm, sibling distance=4cm]
\tikzstyle{level 2}=[level distance=3.5cm, sibling distance=2cm]
\tikzstyle{edge from parent}=[<-,draw]
\node[level 0] {$\hu(0,x)$} [grow'=right,sibling distance=3cm]
  child {node[level 1] {$\hu(T_{1},\hX_{T_1}^{(1)})$}
         child {node[level 2] {$\hu(T_{2},\hX_{T_2}^{(1,1)}) = \phi(\hX_{T_2}^{(1,1)})$ }
                edge from parent node[above left] {}
               }
         child {node[level 2] {$\hu(T_{2},\hX_{T_2}^{(2,1)}) = \phi(\hX_{T_2}^{(2,1)})$}
                edge from parent node[below left] {}
               }
         edge from parent node[above left] {}
        }
  child {node[level 1] {$\hu(T_{2},\hX_{T_1}^{(2)})$}
         child {node[level 2] {$\hu(T_{2},\hX_{T_2}^{(1,2)}) = \phi(\hX_{T_2}^{(1,2)})$}           
                edge from parent node[above left] {}
               }
         child {node[level 2] {$\hu(T_{2},\hX_{T_2}^{(2,2)}) = \phi(\hX_{T_2}^{(2,2)})$}
                edge from parent node[below left] {}
               }
         edge from parent node[below left] {}
        };
\end{tikzpicture}\\
\begin{tikzpicture}[scale=0.8, every node/.style={transform shape}, level distance=2.5cm]
\tikzstyle{level 0}=[level distance=2.5cm]
\tikzstyle{level 1}=[level distance=2.5cm]
\tikzstyle{level 2}=[level distance=3.5cm]
\tikzstyle{edge from parent}=[->, draw]
\node[level 0] {$\delta_x$} [grow'=right,sibling distance=3cm]
  child {node[level 1] {$\hmu_{T_1}$}
         child {node[level 2] {$\hmu_{T_2}$ }} };
\end{tikzpicture}
\hspace{2 em}
\begin{tikzpicture}[scale=0.8, every node/.style={transform shape}, level distance=2.5cm]
\tikzstyle{level 0}=[level distance=2.5cm]
\tikzstyle{level 1}=[level distance=2.5cm]
\tikzstyle{level 2}=[level distance=3.5cm]
\tikzstyle{edge from parent}=[-, draw]
\node[level 0] {$\delta_x$} [grow'=right,sibling distance=3cm]
  child {node[level 1] {$\hmu_{T_1}$}
         child {node[level 2] {$\hmu_{T_2}$ }} };
\end{tikzpicture}
\caption{Left: Cubature tree.  Right: Backward scheme. }
\label{treecub}
\end{center}
\end{figure}
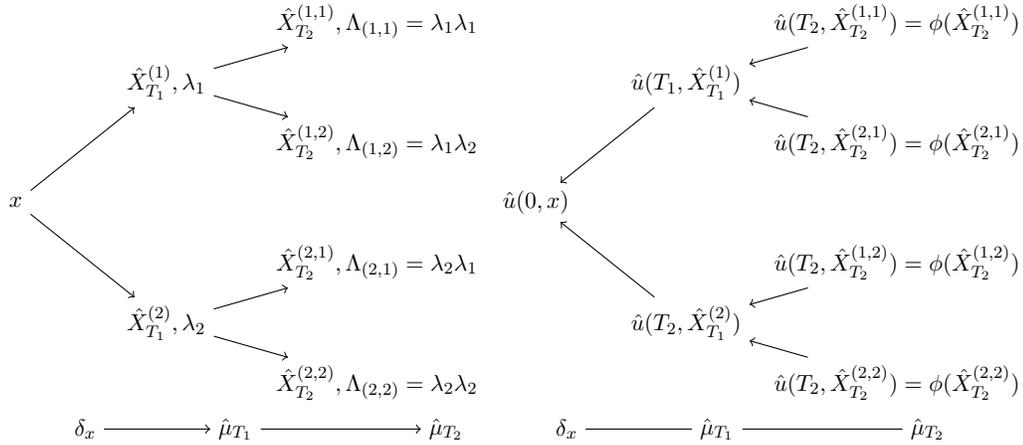


\subsection{Algorithms}
Having the general idea in mind, we can give a precise description of each of the two main parts of our proposed algorithm. Since the cubature involves Stratanovitch integrals, we set:
\begin{equation}\label{Eq:V0Bar}
\bV_0^k = V_0^k - \frac{1}{2}\sum_{i,j=1}^d V_j^i \frac{\p}{\p x_j} V_k^j,
\end{equation}
for all $k \in \{1,\cdots,d\}$ and rewrite the system \eqref{NLFBSDE} as:

\begin{equation}\label{NLFBSDEStrata}
\left\lbrace\begin{array}{ll}
dX^x_t=\bV_0(t,X^x_t,\E \varphi_0(X^x_t))dt + \sum_{i=1}^dV_i(t,X^x_t,\E \varphi_i(X^x_t)) \circ dB^i_t\\
dY^x_{t}  = -f(t,X^x_t,Y^x_t,Z^x_t,\E \varphi_f(X^x_t,Y^x_t))dt + Z^x_t dB_t^{1:d}\\
X^x_0 = x,\quad Y^x_{T}  = \phi(X^x_T).
\end{array}
\right.
\end{equation}

In order to make the description of the algorithm as clear as possible, for any $k,\kappa$ in $\NN$, we set $\mc{S}_\kappa(k) = \{\text{multi-index }(j_1,\cdots,j_k) \in \{1,\cdots,\kappa\}^k\}$, i.e., $\mc{S}_\kappa(k)$ is the set of multi-indices with entries between $1,\ldots,\kappa$ of length (exactly) $k$.

\subsubsection{Building the tree $\mathcal{T}(\gamma,q,m)$}\label{Sec:AlgoCubTree}

\textbf{The subdivision.} Let  $\gamma> 0$, $N\in \mathbb{N}^*$, let $0=T_0<\ldots<T_N=T$ be a discretization of the time interval $[0,T]$ given as
\begin{equation}\label{Eq:TimeDiscretization}
T_k = T\left[1-\left(1-\frac{k}{N}\right)^\gamma\right]
\end{equation}
and let $\Delta_{T_k} = T_k-T_{k-1}$.

\begin{remarque}
When the boundary condition is not smooth, we take a non-uniform subdivision in order to refine the discretization step close to the boundary as proposed by Kusuoka in \cite{kusuoka_approximation_2001}.  If, on the contrary, the boundary condition is smooth, we may use a classical uniform discretization. For this reason, in the following we assume that $\gamma=1$ if \textbf{(SB)} holds, and that  $\gamma> m-1$ under \textbf{(LB)}.
\end{remarque}

Let $\gamma$ be given as explained above, $q$ and $m$ be two given integers, and $\big\{\{\omega_1,\cdots,\omega_\kappa\},\{\lambda_1,\cdots,\lambda_\kappa\}\big\}$ be a $m$ order cubature (the number $\kappa$ of paths and weights depends on $m$). Recall that $\omega_j: t \in [0,1] \mapsto (\omega^{1}_j(t),\ldots,\omega^{d}_j(t)) \in \R^d$ is some continuous function with bounded variation and for all $t$ in $[0,T]$, we set $\omega_0(t)=t$. Examples of cubature formulas of order $3,5,7,9,11$ can be found in \cite{lyons_cubature_2004} or \cite{gyurko_efficient_2011}. 
\small
\begin{algorithm}[h!]
\caption{Cubature Tree $\mathcal{T}(\gamma,q,m)$}\label{algo:KF}
\begin{algorithmic}[1]
\State Set $(X^{\emptyset},\hat{\mu}_{T_0},\Lambda_0) =(x, \delta_x,1)$
\For{$0 \leq i \leq d$}
\State Set $\displaystyle F_i(t,\hat{\mu}_{T_{0}})=\sum_{p=0}^{q-1} \dfrac{1}{p!}(t-T_{0})^p \la\delta_x, (\mathcal{L}^{\delta_x})^p \varphi_i\ra$
\EndFor
\For {$1 \leq k \leq N-1$}
\For {$\pi \in \mc{S}_\kappa(k)$}
\For {$1 \leq j \leq \kappa$}
\State Define $\hat{X}_{T_{k+1}}^{\pi * j}$ as the solution of the ODE:
\begin{equation*}
\left.\begin{array}{ll}
d\hX^{\pi * j}_t =  \sum_{i=0}^{d}V_i(t,\hX_t^{\pi * j},F_i(t,\hat{\mu}_{T_k}))  \sqrt{\Delta_{T_{k+1}}}d\omega_{j}^{i}((t-T_k)/\Delta_{T_{k+1}}),\\
\hX^{\pi *j}_{T_k}=\hX^{\pi}_{T_k}
\end{array}
\right.
\end{equation*}
\State Set the associated weight: $\Lambda_{\pi * j} = \Lambda_{\pi} \lambda_j$
\EndFor
\EndFor
\State Set $\displaystyle \hat{\mu}_{T_{k+1}} =\sum\limits_{\pi \in \mc{S}_\kappa(k+1)}^{} \Lambda_{\pi} \delta_{\hX_{T_{k+1}}^{\pi}}$
\For{$0 \leq i \leq d$}
\State Set $\displaystyle F_i(t,\hat{\mu}_{T_{k+1}})=\sum_{p=0}^{q-1} \dfrac{1}{p!}(t-T_k)^p \la\hmu_{T_{k+1}} , (\mathcal{L}^{\hmu})^p \varphi_i\ra$
\EndFor
\EndFor
\end{algorithmic}
\end{algorithm} 
\newpage
\subsubsection{Backward scheme}
\begin{algorithm}[h!]
\caption{First order backward scheme}\label{alg:FirstOrder}
\begin{algorithmic}[1]
\For {$\pi \in \mc{S}_\kappa(N)$}
\State Set $ \displaystyle \hu^1(T_N,\hX_{T_N}^{\pi})  =\phi(\hX^{\pi}_{T_N})$
\State Set $\displaystyle \hv^1(T_N,\hX_{T_N}^{\pi})  = 0$ \label{algfo:step1}
\EndFor
\For {$N-1 \geq k \geq 1$}
\For {$\pi \in \mc{S}_\kappa(k)$}
\State $\displaystyle \hv^1(T_k, \hX^{\pi}_{T_k})  =  \frac{1}{\Delta_{T_{k+1}}}\sum_{j=1}^{\kappa} \lambda_j \hu^1(T_{k+1},X^{\pi * j}_{T_{k+1}}) \sqrt{\Delta_{T_{k+1}}}\omega_{j}(1)$
\For{$1 \leq j \leq \kappa$}
\State $\displaystyle\hat{\Theta}^{\pi,1}_{k+1,k}(j)=\left(T_{k+1},\hX^{\pi * j}_{T_{k+1}},\hu^1(T_{k+1},\hX^{\pi * j}_{T_{k+1}}) ,\hv ^1(T_{k},\hX^{\pi}_{T_k}),  F^1(T_{k+1},\hmu_{T_{k+1}})\right)$ \label{Eq:DefHThetaK1}
\EndFor
\State $\displaystyle \hu^1(T_k, \hX^{\pi}_{T_k}) = \sum_{j=1}^{\kappa} \lambda_j \left(\hu^1(T_{k+1},X^{\pi * j}_{T_{k+1}})+\Delta_{T_{k+1}} f(\hat{\Theta}_{k+1,k}(j)) \right)$
\EndFor
\State Set $\displaystyle F^1(T_{k+1},\hat{\mu}_{T_{k+1}})= \la\hmu_{T_{k+1}}, \varphi_f(\cdot,\hu^1(T_{k+1},\cdot))\ra$
\EndFor
\end{algorithmic}
\end{algorithm} 

\begin{algorithm}[h!]
\caption{Second order backward scheme}\label{alg:SecondOrder}
\begin{algorithmic}[1]
\For {$\pi \in \mc{S}_\kappa(N)$}
\State Set 
$\hu^2(T_N,\hX_{T_N}^{\pi}) =\phi(\hX^{\pi}_{T_N})$
\State Set $\hv^2(T_N,\hX_{T_N}^{\pi}) = 0$ \label{algso:step1}
\EndFor
\For {$\pi \in \mc{S}_\kappa(N-1)$}
\State Set 
$\hu^2(T_{N-1},\hX_{T_{N-1}}^{\pi}) =\hu^1(T_{N-1},\hX_{T_{N-1}}^{\pi})$ and  $\hv^2(T_{N-1},\hX_{T_{N-1}}^{\pi}) = \hv^1(T_{N-1},\hX_{T_{N-1}}^{\pi})$ \label{algso:step2}
 \State Set $\displaystyle F^2(T_{N-1},\hat{\mu}_{T_{N-1}})= \la\hmu_{T_{N-1}},\varphi_f(\cdot,\hu^2(T_{N-1},\cdot))\ra$
\EndFor

\For {$N-2 \geq k \geq 1$}
\For {$\pi \in \mc{S}_\kappa(k)$}
\For{$1 \leq j \leq \kappa$}
\State $\displaystyle  \hTheta^{\pi,2}_{k+1}(j)=\left(T_{k+1},\hX^{\pi * j}_{T_{k+1}},\hu^2(T_{k+1},\hX^{\pi * j}_{T_{k+1}}) ,\hv^2(T_{k+1},\hX^{\pi * j}_{T_{k+1}}),  F^2(T_{k+1},\hat{\mu}_{T_{k+1}})\right)$ \label{Eq:DefHThetaK2}
\State $\displaystyle \hzeta^{\pi*j}_{k+1}:=4\frac{1}{\Delta_{T_{k+1}}}\sqrt{\Delta_{T_{k+1}}}\omega_j(1)-6 \frac{1}{\Delta^2_{T_{k+1}}}\int_{T_k}^{T_{k+1}} (s-T_k) \sqrt{\Delta_{T_{k+1}}}d\omega_j((s-T_k)/\Delta_{T_{k+1}})$\label{Eq:DefHZeta}
\EndFor
\State Set $\displaystyle  \hv^2(T_k, \hX^{\pi}_{T_k})  = \sum_{j=1}^{\kappa} \lambda_j \left(\hu^2(T_{k+1},X^{\pi*j}_{T_{k+1}}) + \Delta_{T_{k+1}}f({\hTheta}^{\pi*j,2}_{k+1}) \right) \hzeta_{k+1}^{\pi*j}$
\State [\emph{Predictor}]
 $\displaystyle
\tu(T_k, \hX^{\pi}_{T_k})  =  \sum_{j=1}^{\kappa} \lambda_j \left(\hu^2(T_{k+1},X^{\pi * j}_{T_{k+1}})+\Delta_{T_{k+1}} f({\hTheta}^{\pi*j,2}_{k+1}) \right)$
\State Set $\tilde{F}(T_{k},\hat{\mu}_{T_{k}})=\E_{\hmu_{T_{k}}} \varphi_f(\cdot,\tu(T_{k},\cdot))$
\State [\emph{Corrector}] $\displaystyle \tilde{\Theta}^\pi_{k}=\left(T_{k},\hX^{\pi }_{T_k},\tu(T_{k},\hX^{\pi }_{T_k}) ,\hv^2(T_{k},\hX^{\pi }_{T_k}),\tilde{F}(T_{k},\hat{\mu}_{T_{k}})\right)$ \label{Eq:DefHThetaK}
\State $ \displaystyle \hu^2(T_k, \hX^{\pi}_{T_k}) =  \sum_{j=1}^{\kappa} \lambda_j \left(\hu^2(T_{k+1},X^{\pi * j}_{T_{k+1}})+\frac{1}{2}\Delta_{T_{k+1}} \left(f({\hTheta}_{k+1}^{\pi*j,2}) +f(\tilde{\Theta}_{k}^{\pi}) \right) \right)$
\State Set $\displaystyle F^2(T_{k},\hat{\mu}_{T_{k}})= \la\hmu_{T_{k}}, \varphi_f(\cdot,\hu^2(T_{k},\cdot))\ra$
\EndFor
\EndFor
\end{algorithmic}
\end{algorithm} 
\newpage
\normalsize

\begin{remarque}
The initialization value for $v$ at the boundary, that we have fixed in $0$, is arbitrary, given that the first steps in both algorithms does not use this value. 

However, if the algorithm is used under \SB and the values of $D_x u$ can be easily calculated on the boundary, we have a natural initialization value for $v$. In this case, we may initialize the backward algorithms of order one and two to reflect this additional information, by setting $v(t,x) = (\mc{V}^{\mu}u(t,x))^T$.

This modification will have no effect at all for the first order scheme, and is interesting only from the point of view of consistence. On the other hand, on the second order algorithm, the natural initialization of $v$ allows to skip the first order step. This change does not affect the overall rate of convergence of the algorithm but will induce a reduction in the error constant whence of the total approximation error.
\end{remarque}

It is worth noticing that the given algorithm is particularly effective for treating the McKean dependence of the backward component. Indeed, note that the expectation of any regular enough function of $\hu$ is readily available given that the support of the approximating measure $\hmu$ coincides with the points where $\hu$ is available. Of course the situation is quite different when a different approach, like a particle method, is used. 

\section{Main Results}\label{Sec:MainResult}

In this section, we first give the rate of convergence of our algorithms \ref{algo:KF}, \ref{alg:FirstOrder} and \ref{alg:SecondOrder} when both the coefficients and terminal condition in \eqref{NLFBSDE} are smooth. This is given in Theorem \ref{MR} below. Then, we give the rate when the boundary condition is Lipschitz and when the diffusion part of \eqref{NLFBSDE} is uniformly non-degenerate. This does not really affect the convergence order, provided the subdivision is taken appropriately. The result is summarized in Corollary \ref{Lipcase} below. Finally,  we give the convergence of a version of our algorithm applied to equation \eqref{NLFBSDEG2}: when the dependence of the coefficients with respect to the law is general. This is given in Corollary \ref{wasscase}.

In order to make the exposition of our results clear, let us define, for $i=1,2$:

\begin{equation} \label{Eq:def_epsilon}
\mE^i_u(k):=\max\limits_{\pi\in\mc{S}_\kappa(k)} |u(T_k,\hX^{\pi}_{T_k})-\hu^i(T_k,\hX^{\pi}_{T_k})|; \quad \quad \mE^i_v(k):=\max\limits_{\pi\in\mc{S}_\kappa(k)} |v(T_k,\hX^{\pi}_{T_k})-\hv^i(T_k,\hX^{\pi}_{T_k})|,
\end{equation}
with $\hu^1,\hu^2, \hv^1$ and $\hv^2$ as defined by the algorithms \ref{alg:FirstOrder} and \ref{alg:SecondOrder} and where $u, v$ are defined in \eqref{Eq:DefU}, \eqref{Eq:DefV}. 

\textbf{Main result in a smooth setting.} We have that
\begin{theoreme}\label{MR} Suppose that assumption \textbf{(SB)} holds. Let $m$ be a given cubature order, $q$ a given non-negative integer and $\mc{T}(1,q,m)$ the cubature tree defined by Algorithm \ref{algo:KF}. Then, there exists a positive constant $C$, depending only on $T$, $q$, $d$, $||\varphi_{0:d}||_{2q+m+2,\infty}$, $||\phi||_{m+2,\infty}$, such that:
\begin{equation}\label{Eq:ErrorforSB}
\max_{k \in \{0,\ldots,N\}}\left|\la \mu_{T_k}-\hmu_{T_k} , \phi\ra\right| \leq C \left(\frac{1}{N}\right)^{[(m-1)/2 ]\wedge q},
\end{equation}
with $\hmu$ as defined in Algorithm \ref{algo:KF}.\\
Suppose in addition that $q\geq 1$ and $m\geq3$. Then, there exists a positive constants $C_1$, depending only on $T$, $q$, $d$, $||\varphi_{0:d}||_{2q+m+2,\infty}$, $||\varphi_f||_{m+2,\infty}$,$||\phi||_{m+3,\infty}$, such that for all $k=0, \ldots, N$:
\begin{equation}
\mE^1_u(k) + \Delta_{T_k}^{1/2} \mE^1_v(k) \leq C_1 \left(\frac{1}{N}\right),\label{Eq:rate1}
\end{equation}
Moreover,  suppose in addition that  $q\geq 2$ and $m\geq 7$. Then, there exists a positive constant $C_2$, depending only on $T$, $q$, $d$, $||\varphi_{0:d}||_{2q+m+2,\infty}$, $||\varphi_f||_{m+2,\infty}$,$||\phi||_{m+4,\infty}$, such that for all $k=0, \ldots, N$:
\begin{equation}
\mE^2_u(k) + \Delta_{T_k}^{1/2} \mE^2_v(k) \leq C_2 \left(\frac{1}{N}\right)^2.\label{Eq:rate3}
\end{equation}
\end{theoreme}

\textbf{Convergence order for a Lipschitz boundary condition.}
\begin{corollaire}\label{Lipcase}
Suppose that assumption \textbf{(LB)} holds. Let $m$ be a given cubature order, $q$ a given non-negative integer, $\gamma$ a non negative real and $\mc{T}(\gamma,q,m)$ the cubature tree defined by the algorithm \ref{algo:KF}. Then, there exists a positive constant $C$ depending only on $T$, $||\varphi||_{2q+m+2,\infty}$, $||\phi||_{1,\infty}$, such that:
\begin{eqnarray}\label{Eq:ErrorforLB}
\left|\la \mu_{T} - \hmu_{T}, \phi\ra \right|& \leq & C \left(\left(\frac{1}{N}\right)^{(\gamma /2 )\wedge q}\vee L(\gamma,m)\right). 
\end{eqnarray}
where
\begin{equation} \label{Eq:MFunction}
L(\gamma,m) = 
\left\lbrace\begin{array}{lll}
N^{-\gamma /2}& \quad \text{ if } \quad\gamma \in \left(0,m-1\right)\\ 
N^{-(m-1)/2} \ln(N)  &\quad \text{ if } \quad\gamma =m-1\\ 
N^{-(m-1)/2}  &\quad \text{ if } \quad \gamma \in \left(m-1,+\infty\right)
\end{array}
\right.
\end{equation}
Moreover, if $\gamma > m-1$, the results on the error control of $\hu^1,\hv^1 ;\hu^2, $ and $\hv^2$ respectively given by \eqref{Eq:rate1} and \eqref{Eq:rate3} remain valid (with a constant $C_2'$ depending only on $T$, $q$, $d$, $||\varphi_{0:d}||_{2q+m+2,\infty}$, $||\varphi_f||_{m+2,\infty}$,$||\phi||_{1,\infty}$).
\end{corollaire}

Note that, in \eqref{Eq:ErrorforLB}, the control holds only at time $T$ although it holds at each step in \eqref{Eq:ErrorforSB}: this is because the boundary condition is Lipschitz only so that we have to wait for the smoothing effect to take place. 

We emphasize that the result still applies if we let the boundary condition $\phi: (y,w) \in \R^d\times \R \mapsto \phi(y,w)$ depend also on the law of the process $(X_t^x,\ 0\leq t \leq T)$. For example one can consider
\[\E\left[\phi(X_t^x,\E[\varphi_\phi(X_T^x)]) \right] \] 
for a given $\varphi_\phi\in C^{m+2}_b$ and where $\phi$ is Lipschitz in $w$ uniformly in $y$.

The algorithm can be easily adapted to the case of the particular dependence explored in \cite{buckdahn_mean-field_2009}:
$$V_i(t,y,\mu) = \la\mu_t,V_i(t,y,\cdot)\ra,\ i=0,\cdots,d,$$
and the result of Theorem \ref{MR} and Corollary \ref{Lipcase} remain valid. Note that in that case the uniform ellipticity \textbf{(LB)} has to be understood for the matrix $\left[\la\eta_t,V(t,y,\cdot)\ra\right] \left[\la\eta_t,V(t,y,\cdot)\ra\right]^*$ uniformly in $y$, $t$  in $\R^d\times \R^+$ and $\eta$ family of probability measures on $\R^d$.

\textbf{Results under a more general law dependence.}

As mentioned in the introduction, the algorithm may be modified to solve problems in a naturally extended framework. Let us precise the framework of this extension.

Let $\mF$ and $\mF'$ be two classes of functions, dense in the space of continuous functions that are zero at infinity. Let $d_{\mF}$ $d_{\mF'}$  be two distances as defined in \eqref{Eq:distance}. Recall that we suppose the vector fields $V_i$ $0\leq i \leq d$ that appear in \eqref{NLFBSDEG2} to be Lipschitz continuous with respect to $d_\mF$ and the driver $f$ to be Lipschitz continuous with respect to $d_{\mF'}$. Furthermore, let us suppose that there exists a unique solution $X_t^x,Y_t^x,Z_t^x$ to such a system and, as before, denote by $u$ the decoupling function defined as \eqref{Eq:DefU} given its existence.

Clearly, we need to modify Algorithm \ref{algo:KF} in the natural way to be used in this framework, that is, at each discretization time, we plug directly in the coefficients the cubature based law. 

In order to retrieve higher orders of convergence, we need to expand the McKean term that appears in the coefficients. In this extended case, we have to be careful when considering the forward algorithm with $q>1$: indeed, we must give sense to the expansion proposed at the definition of the functions $F_i$ $0\leq i \leq d$ in Algorithm \ref{algo:KF}. A good notion may be the one proposed in Section 7 of \cite{cardaliaguet_notes_2010}. To avoid further technicalities, we will consider here only the case with $q=1$, i.e., when no expansion is performed. 

With this definitions and observations in mind, we give the main result under a more general law dependence.

\begin{corollaire}\label{wasscase_backward}\label{wasscase}
Let $\mu_T$ be the marginal law of the forward process in \eqref{NLFBSDEG2} at time $T$. Let $m \geq 3$ be a given cubature order and $\hmu_T$ be the discrete measure given by the cubature tree $\mc{T}(1,1,m)$ defined by the algorithm \ref{algo:KF}. Then, there exist two positive constants $C_1$ and $C_2$, depending only on $T$, $d$ such that:
\begin{itemize}
\item If \textbf{(SB)} holds and $\mF$ (resp. $\mF'$) is the class of functions $\varphi$ in $C^\infty_b(\R^d,\R)$ (resp. $C^\infty_b(\R^d\times \R,\R)$) such that $||\varphi||_{\infty,\infty} \leq 1$, then
\begin{equation}\label{wasserror1}
d_\mF(\mu_T,\hmu_T) +\mE^1_u(k) + \Delta_{T_k}^{1/2} \mE^1_v(k)  \leq C_1 N^{-1}.
\end{equation}
\item If \textbf{(LB)} holds and $\mF$ (resp. $\mF'$) is the class of functions $\varphi$ in $C^1_b(\R^d,\R)$ (resp. $C^1_b(\R^d\times \R,\R)$) such that $||\varphi||_{1,\infty} \leq 1$, then
\begin{equation}\label{wasserror2}
d_\mF(\mu_T,\hmu_T)  + \mE^1_u(k) + \Delta_{T_k}^{1/2} \mE^1_v(k)  \leq C_2 N^{-1/2}.
\end{equation}
\end{itemize}
\end{corollaire}

We emphasize that, when $\mF = \{\varphi \in C^1_b(\R^d,\R),\ s.t.\ ||\varphi||_{1,\infty} \leq 1\}$, thanks to the Monge-Kantorovitch duality theorem, the distance $d_\mF$ is the so-called Wasserstein-1 distance.

\section{Numerical examples}\label{Sec:NumEx}

In this section, we illustrate the algorithm behavior by applying it to a toy model for which the exact solution is available. 

Consider the $d-$dimensional MKV-FBSDE on the interval $[0,1]$ with dynamics given by
\begin{align*}
dX_t & = \EX{ \sin(X_t)} dt + dB_t\\
-dY_t & = \left( \frac{\mathbf{1}\cdot \cos(X_t)}{2} + \EX{ (\mathbf{1}\cdot\sin(X_t))\exp(-Y_t^2)} \right)dt - Z_t\cdot dB_t,
\end{align*}
where $(B_t)_{0\leq t\leq 1}$ is a $d-$dimensional Brownian motion, $\mathbf{1}$ is a $d-$dimensional vector having each entry equal to one and the $\sin$ and $\cos$ functions are applied entry-wise. Moreover, suppose that $X_0=\mathbf{0}$. It is easily verified that a solution for the forward variable is $X=B$, and thanks to the uniqueness result this is the unique solution for the forward variable.\\

With respect to the backward part, take two different boundary conditions corresponding to the two considered set of assumptions \textbf{(SB)} and \textbf{(LB)}. 

\begin{description}
\item [(SB)] For $x\in \R^d$, we fix $\phi(x)=\mathbf{1}\cdot\cos(x)$. In this case, the solution to the backward part of the problem is
\[u(t,x) = \mathbf{1}\cdot \cos(x) \text{; and } v(t,x)=-\sin(x),\]
which clearly implies $Y_t = \mathbf{1}\cdot  \cos(X_t)$ and $Z_t = -\sin(X_t)$.
\item [(LB)] We fix the boundary condition to be $\phi(x):= \phi'(d^{-1/2}(\mathbf{1}\cdot x))$ where $\phi'$ is the triangular function defined for all $y\in \R$ as
\[\phi'(y)=\begin{cases} y+K & \text{ if } y \in (-K,0] \\ -y+K & \text{ if } y\in (0,K] \\ 0 & \text{otherwise.}\end{cases}\]
In this case, the solution is given by
$$u(t,x) = \E\left[\phi(X_T^x) + \int_t^T \left(\frac{\mathbf{1}\cdot \cos(X_t)}{2} + \EX{ (\mathbf{1}\cdot\sin(X_t))\exp(-Y_t^2)} \right)ds \right].$$
Basic properties of the Brownian motion imply that
\[u(t,x) = U(t,d^{-1/2}(\mathbf{1}\cdot x)) + (\mathbf{1}\cdot\cos(x)) \left[ \exp\left(\frac{t-1}{2}\right)-1 \right] \]
where
\begin{align*}
U(t,y) = & \sqrt{\frac{1-t}{2\pi}} \left[ \exp\left(\frac{-(K+y)^2}{2(1-t)}\right) +\exp\left(\frac{-(K-y)^2}{2(1-t)}\right) -2 \exp\left(\frac{-y^2}{2(1-t)}\right) \right] \\
& + (K+y) \left[F\left( \frac{-y}{\sqrt{1-t}}\right) - F\left( \frac{-K-y}{\sqrt{1-t}}\right) \right] + (K-y) \left[F\left( \frac{K-y}{\sqrt{1-t}}\right) - F\left( \frac{-y}{\sqrt{1-t}}\right) \right]
\end{align*}
and $F$ is the cumulative distribution function of the standard normal distribution. Evidently $v(t,x) = D_x u(t,x)$, and is defined for $t<1$. 
\end{description}

\subsection{Tests in dimension one}

Given that the law dependence already increases the dimension of the problem, we start by presenting some results when we fix $d=1$.

\subsubsection{Forward component}

To implement the forward variable, we use the cubature formulae of order 3 and 5 presented in \cite{lyons_cubature_2004}, which have paths support of size $\kappa=2$ and $\kappa=3$ respectively. Given the simple structure of the forward variable dynamics and the piecewise linear definition of the cubatures, we are able to solve explicitly the ODEs appearing during the tree construction. Hence there is no need to use any ODE solver.\\

In our first test we evaluate the weak approximation error of $X$ using the function $\phi$ as test function. Indeed we plot as error for the \SB\ case
\[\max\limits_{k=1,\ldots,N} \left|\la\hmu_{T_{k}}-\mu_{T_k}, \cos \ra  \right|,\] 
while for the \LB case, we plot
\[ \left|\la\hmu_{T_{k}}-\mu_{T_k}, \phi \ra  \right|,\] 
where $\phi$ is the defined triangular function with $K=0.6$. As was pointed out before, the difference between the kind of error we are observing for each case is justified as a smoothing effect is needed for the approximation to be valid under \LB.

\begin{figure}[h!]
\begin{center}
\includegraphics[width = 0.45\textwidth]{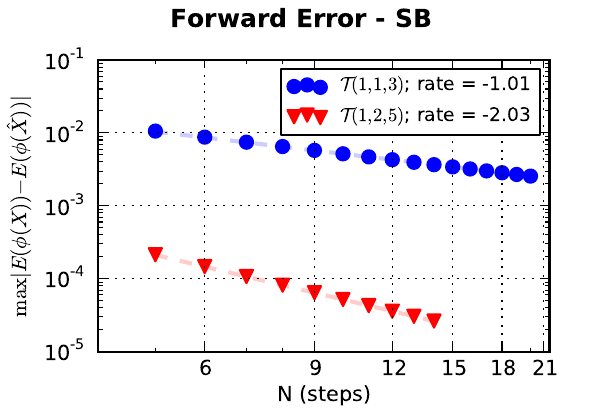}  
\includegraphics[width = 0.45\textwidth]{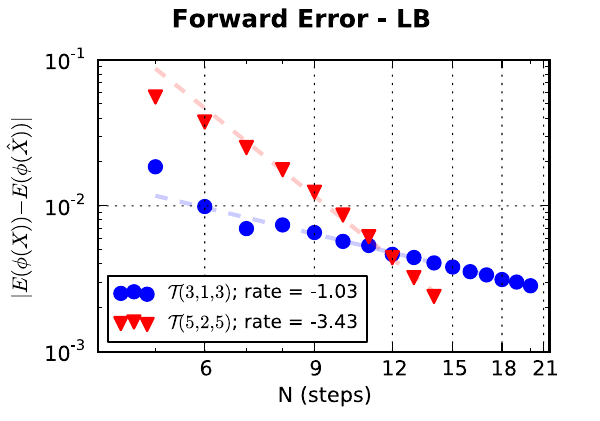}  
\caption{\small{Weak approximation of the forward variable: The calculated rates are the slope of a linear regression on the last 8 points.}}
\label{Fig:ErrorX}
\end{center}
\end{figure}

Figure \ref{Fig:ErrorX} shows the obtained rate of convergence where we have used the uniform discretization in the \SB\ case and the discretization with $\gamma=2$ for the \LB\ case. With the exception of the rate of convergence for the second order algorithm under \LB (which is actually better than the predicted one), the expected rates of convergence are verified in both cases.  \\

Moreover, under the smooth case, the benefit of using the higher order scheme is not only evident from a quickest convergence, but the error constant itself is smaller. This is an effect that depends on the particular example, but we remark it as it is interesting to notice that a higher order of convergence does not imply necessarily a higher initial constant.

\subsubsection{Backward component}
 \begin{figure}[h!]
\begin{center}
\includegraphics[width = 0.45\textwidth]{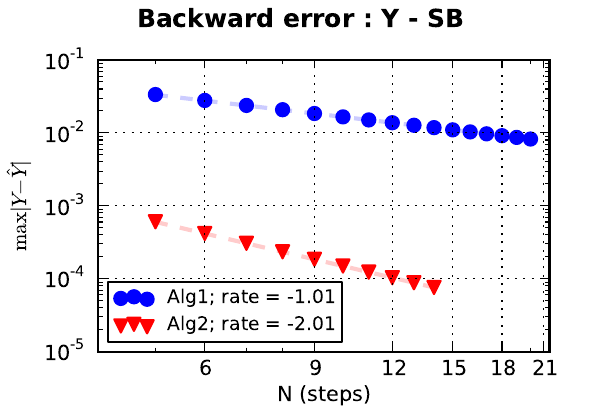}  
\includegraphics[width = 0.45\textwidth]{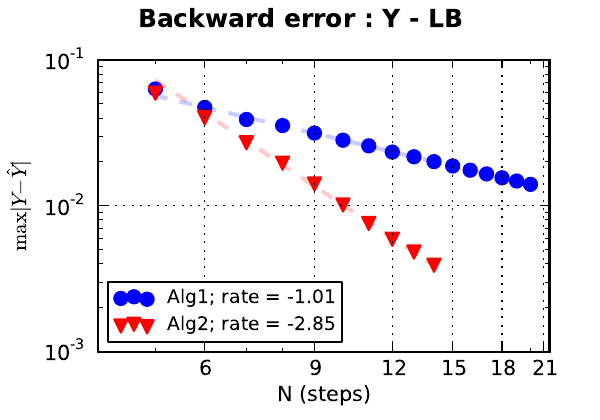}\\
\includegraphics[width = 0.45\textwidth]{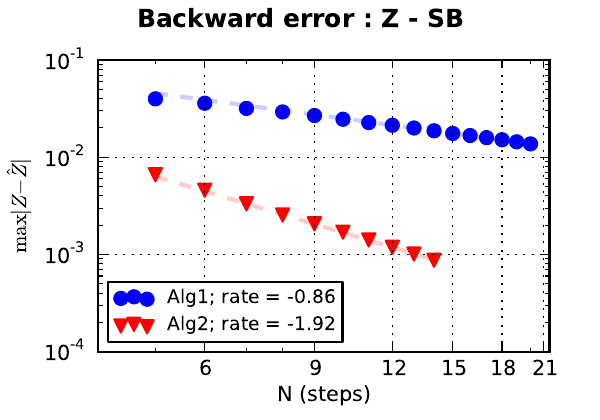}  
\includegraphics[width = 0.45\textwidth]{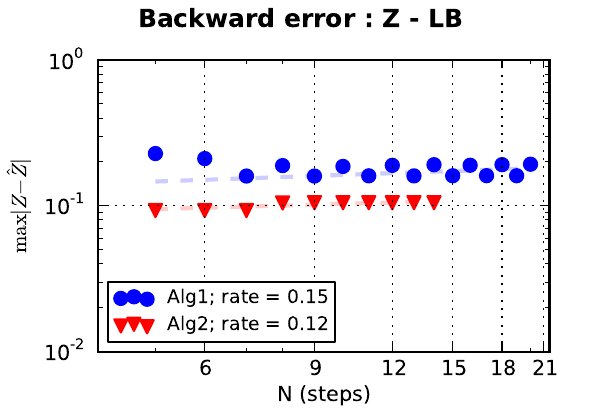} \\
\includegraphics[width = 0.45\textwidth]{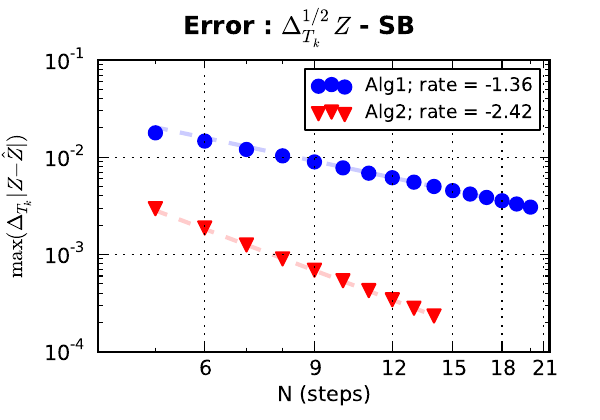}  
\includegraphics[width = 0.45\textwidth]{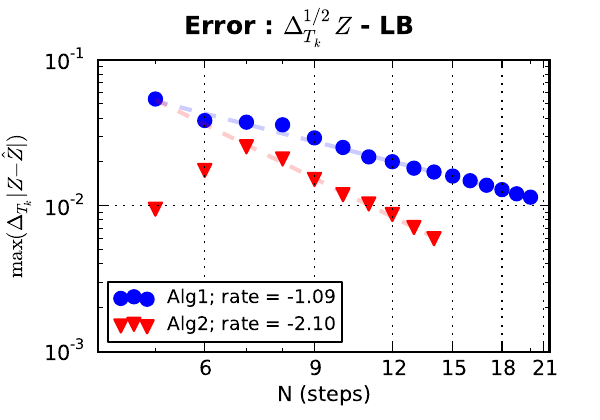}  
\caption{\small{Weak approximation of the backward variable: The calculated rates are the slope of a linear regression on the last 8 points.}}
\label{Fig:ErrorY}
\end{center}
\end{figure}

Let us check now the approximation of the backward variable. We evaluate numerically 
\[  \max\limits_{0\leq k \leq N-2 ;\ \pi \in \mc{S}_\kappa(k)} \left|  \hat{u}^1(T_k,\ApproxVarX{X}_{T_k}^\pi) -  u(T_k\ApproxVarX{X}_{T_k}^\pi) \right| \quad \text{and} \quad  \max\limits_{0\leq k \leq N-2 ;\ \pi \in \mc{S}_\kappa(k)} \left|  \hat{u}^2(T_k,\ApproxVarX{X}_{T_k}^\pi) -  u(T_k\ApproxVarX{X}_{T_k}^\pi) \right|;\] 
for both the (SB) and (LB) cases, where we fix $K=0.6$ for the latter.

%
The specific structure of our examples allows us to obtain a second order convergence scheme with a cubature of order only 5. Indeed, in such a case, the terms in front of the leading rate of convergence on the cubature error estimate (cf Claim \ref{Cl:back2-2}) are identically 0. Given that the order 5 cubature induces a lower complexity, it is simpler to carry out simulations for a larger number of steps.  

As can be appreciated from the two uppermost plots in Figure \ref{Fig:ErrorY}, the expected rates of convergence for both algorithms are verified under the smooth and Lipschitz conditions. Just as we remarked in the forward approximation, solving the backward variable in the smooth case with the higher order scheme has the double benefit of better rate of convergence and smaller constant. As one would expect, due to the use of higher order derivatives, this is no longer true for the Lipschitz case.\\

It is interesting to look at the behavior of the other backward variable, $Z$. We look first at an error of the type
\[ \max\limits_{0\leq k \leq N ;\ \pi \in \mc{S}_\kappa(k)} \left|  \hat{v}^1(T_k,\ApproxVarX{X}_{T_k}^\pi) -  v(T_k\ApproxVarX{X}_{T_k}^\pi) \right|\quad \text{and} \quad  \max\limits_{0\leq k \leq N ;\ \pi \in \mc{S}_\kappa(k)} \left|  \hat{v}^2(T_k,\ApproxVarX{X}_{T_k}^\pi) -  v(T_k\ApproxVarX{X}_{T_k}^\pi) \right|.\] 

 The two plots in the middle of Figure \ref{Fig:ErrorY} are concerned with these errors. Although nice convergence is obtained in the smooth case, this is no longer true for the \LB case, where the error stagnates. As will be clear from the analysis, this is a consequence of the singularity appearing at the boundary on the control of derivatives in this case.  Hence, a more adequate error analysis considers errors given by
\[ \max\limits_{0\leq k \leq N ;\ \pi \in \mc{S}_\kappa(k)} \Delta_{T_k}^{1/2}  \left|  \hat{v}^1(T_k,\ApproxVarX{X}_{T_k}^\pi) -  v(T_k\ApproxVarX{X}_{T_k}^\pi) \right|\quad \text{and} \quad  \max\limits_{0\leq k \leq N ;\ \pi \in \mc{S}_\kappa(k)} \Delta_{T_k}^{1/2} \left|   \hat{v}^2(T_k,\ApproxVarX{X}_{T_k}^\pi) -  v(T_k\ApproxVarX{X}_{T_k}^\pi) \right|.\] 

The expected rate of convergence of this type of error is, respectively for $\hv^1$ and $\hv^2$, of the same order of the order of the error of $\hu^1, \hu^2$ with respect to $u$. As shown in the bottommost plots in Figure \ref{Fig:ErrorY}, the numerical tests for the \SB\ and \LB\ cases reflect the expected rates. \\

\subsection{Tests in higher dimensions}

We evaluate as well the algorithm using our test models $\SB, \LB$ with dimensions $d=2$ and $d=4$. For these tests, we evaluate only the first order schemes and use the   3-cubature formulae presented in \cite{gyurko_efficient_2011} which have supports of size $\kappa=4$ and $\kappa=6$ respectively. \\

Figure \ref{Fig:ErrorMultidim} shows that, just as is in the one-dimensional case, the announced rates of convergence for the forward and backward variables are verified. Note that for the particular chosen examples, the error value changes just slightly with dimension. \\

 \begin{figure}[h!]
\begin{center}
\includegraphics[width = 0.45\textwidth]{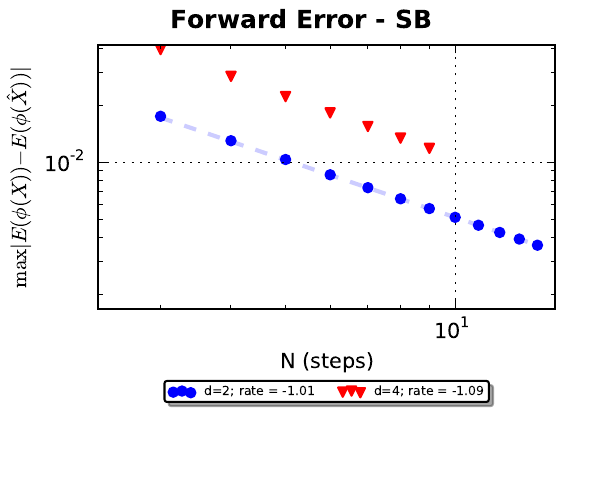}  
\includegraphics[width = 0.45\textwidth]{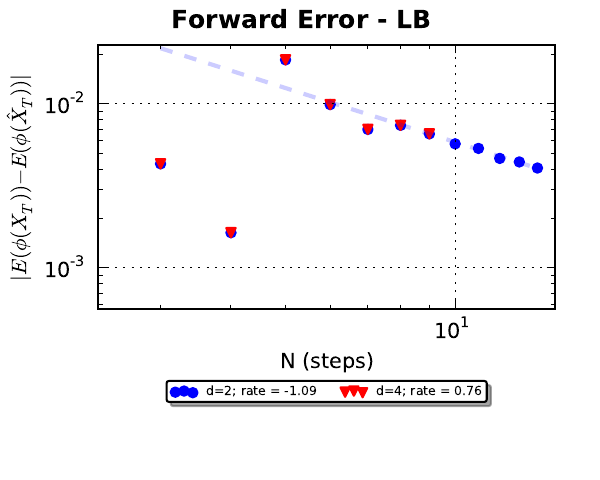}\\
\includegraphics[width = 0.45\textwidth]{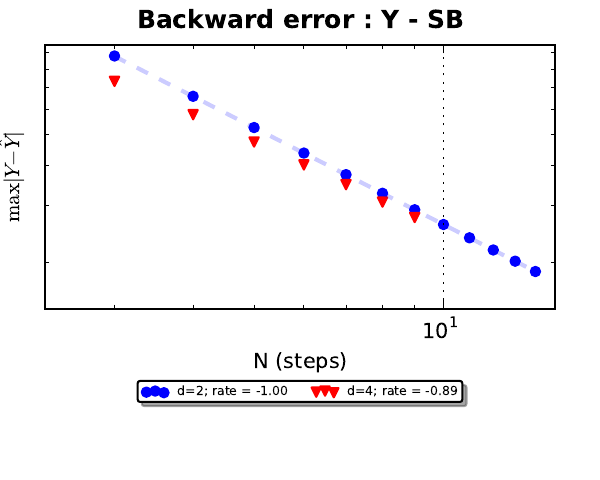}  
\includegraphics[width = 0.45\textwidth]{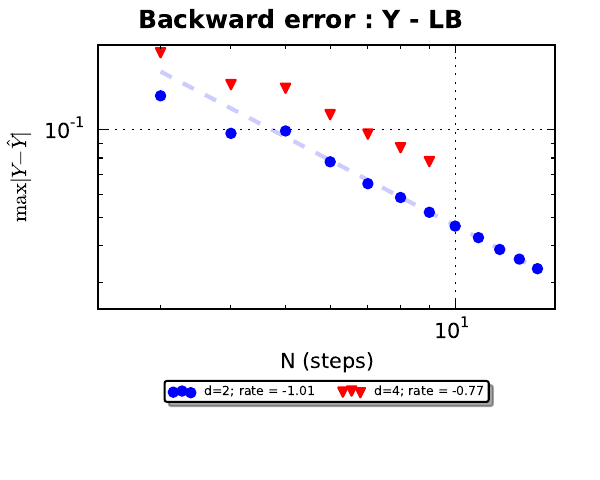}\\
\includegraphics[width = 0.45\textwidth]{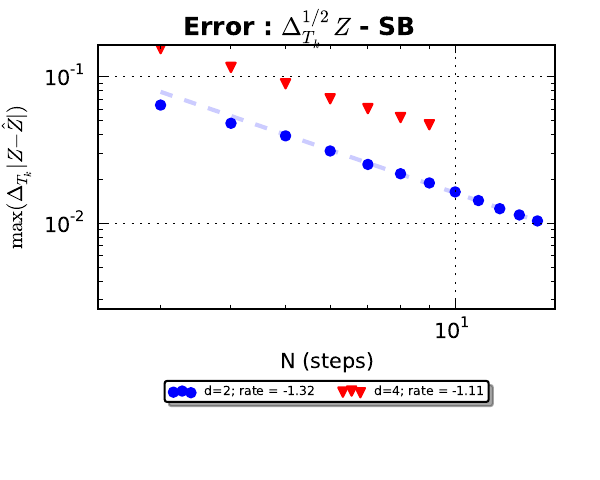}  
\includegraphics[width = 0.45\textwidth]{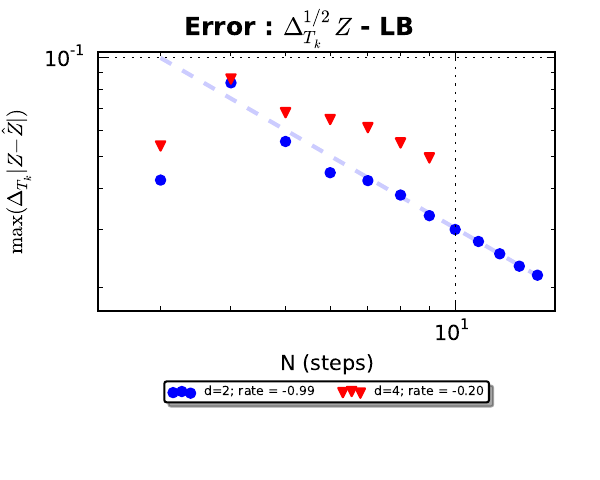}  
\caption{\small{Results in dimension 2 and 4.}}
\label{Fig:ErrorMultidim}
\end{center}
\end{figure}

The case of dimensions 2 and 4  show one of the current limitations of the method: its complexity grows, in general, exponentially both in terms of the number of iterations and the dimension of the problem. Indeed, considering once again the 3-cubature formula, we have that in general the number of nodes in the tree is 
\[ \sharp(nodes) = \frac{(2d)^{n-1}-1}{2d-1},\]
with the obvious effects on memory management and execution time. We remark that for some particular cases, the complexity can be radically lower. For instance, under the case of constant drift and diffusion coefficients and smooth boundary conditions, using symmetric cubature formulae (as we did here) leads to a kind of ``auto-pruning'' of the cubature tree leading to complexity grow of the form
\[ \sharp(nodes) = \sum_{i=1}^{n} i^{d},\]
i.e. polynomial in $n$ with the order of the polynomial depending on the dimension $d$. \\

\section{A class of control problems in a mean field environment} \label{Sec:ControlMeanField}

In this section, we show that equation \eqref{NLFBSDE} appears when solving a class of control problems inspired from the theory of mean field games but designed in such a way that the dynamics of the controlled process have no influence on the mean field environment.

For the sake of illustration, consider for instance the problem of optimization of an issuer having a large portfolio of credit assets inspired in the framework presented in \cite{bush_stochastic_2011}. One of the methods used to model credit asset dynamics is the so called structural model (see \cite{bielecki_credit_2009} for a review on credit risk models). Under this model, we assume that a credit default is triggered when the value of the corresponding credit asset is below a certain threshold. In the original Merton setup, the default may only be triggered at a certain fixed maturity time $T$. In a rather more realistic view, the default is triggered the first time the credit asset is below the threshold. 

We assume that the credit assets in the basket are small and homogeneous (for example, we suppose they belong to the same economic sector) so that their value is modeled by SDEs with the same volatility and drift function terms. To simplify, we will consider the simpler Merton model. Moreover, in order to account for sector-wise contagion effects, we suppose there is a mean field type dependence in the dynamics.  In addition to the credit assets, we suppose the issuer has a market portfolio used by the issuer to backup the credit risk, for example to comply with credit risk regulations, or to provide liquidity to its credit branch. Then, the value of the position of the issuer position is modeled by an SDE with coefficients depending on the contribution of all credit assets. The objective of the control problem is to maximize the value of the issuer position.

We will formalize mathematically a generalized version of the presented example. For this, we introduce a system in which a \emph{marked} particle (the issuer in our example) with a controlled state variable $\Xi(\alpha,M)$ is immersed in an \emph{environment} of $M$ interacting particles (the credit assets in our example) with state variables $\xi^1,\ldots,\xi^M$, and which dynamics are given by 
\begin{equation}\label{sysparticule}
\left\lbrace \begin{array}{l}
d \xi^1_t = b(t,\xi^1_t,\frac{1}{M} \sum_{i=1}^M\delta_{\xi^i_t} )dt + \sigma(t,\xi^1_t,\frac{1}{M} \sum_{i=1}^M\delta_{\xi^i_t} )dB^1_t\\
\quad \vdots \qquad \qquad\qquad\qquad  \vdots \\
d \xi^M_t = b(t,\xi^M_t,\frac{1}{M} \sum_{i=1}^M\delta_{\xi^i_t} )dt + \sigma(t,\xi^M_t,\frac{1}{M} \sum_{i=1}^M\delta_{\xi^i_t} )dB^M_t\\
d\Xi_t(\alpha,M) = b^0(t,\xi^1_t, \Xi_t(\alpha,M), \frac{1}{M} \sum_{i=1}^M\delta_{\xi^i_t},\mu_t,\alpha_t )dt + \sigma^0(t,\xi^1_t,\Xi_t(\alpha,M), \frac{1}{M} \sum_{i=1}^M\delta_{\xi^i_t},\mu_t)dW_t\\
\xi^1_0 = \ldots \xi^M_0  = \xi_0, \quad \Xi_0(\alpha,M) = \Xi_0
\end{array} \right.
\end{equation}
where $\xi_0, \Xi_0$ belong to $\RR^{d}$, $(\alpha_t,\ t \geq 0)$ is a progressively measurable process with image in $A \subset \R$, $B_t^1,\ldots,B_t^M $ are $M$ independent Brownian motions and $(W_t,\ t\geq 0)$ is a $d$-dimensional Brownian motion possibly correlated with $(B_t^1,\ldots,B_t^M )^T$. Note that, in this framework, the marked player does not influence the dynamics of the other players.
For a large number of environment players, the system is approximated by the McKean-Vlasov system
\begin{equation}\label{sysparticulelimitte}
\left\lbrace \begin{array}{l}
d \xi_t = b(t,\xi_t,\mu_t )dt + \sigma(t,\xi_t,\mu_t )dB_t\\
d\bar{\Xi}_t(\alpha) = b^0(t,\xi_t,\bar{\Xi}_t(\alpha),\mu_t,\alpha_t )dt + \sigma^0(t,\xi_t,\bar{\Xi}_t(\alpha),\mu_t )dW_t\\
\xi_0  = \xi, \quad \bar{\Xi}_0(\alpha) = \Xi_0
\end{array} \right.
\end{equation}
where $\mu_t$ is the law of $\xi_t$ that we will assume in the following to be fixed. Assume that the marked player is interested in minimizing the cost functional
\[J(t,\xi_0,\Xi_0,\alpha)= \E\left[g\left(\xi^{t,\xi_0}_T,\bar{\Xi}^{t,\Xi_0}_T(\alpha),\mu_T\right) + \int_0^T f(s,\xi_s^{t,\xi_0},\bar{\Xi}^{t,\Xi_0}_s(\alpha),\mu_s) ds \right],\]
for $\alpha \in \mathcal{A}$, the set of all progressively measurable process $\alpha = (\alpha_t,\ t\geq 0)$ valued in $A$ (the maximization case is available up to a change of sign). Suppose that we want to compute the optimal value function $u(t,\xi_0,\Xi_0) = \inf\{J(t,\xi_0,\Xi_0,\alpha),\  \alpha \in \mathcal{A}\} $. Then, we know that under appropriate assumptions, $u$ can be obtained as the solution of the following Hamilton Jacobi Bellman equation on $[0,T]\times \R \times \R^d$
\begin{align} 
0 = & \p_t u(t,x,\bar{x}) + \frac{1}{2} {\rm Tr}( \bar{a} D^2_{x,\bar{x}} u (t,x,\bar{x})  ) + b(t,x,\mu_t) D_x u (t,x,\bar{x}) +  H(t,x,\bar{x}, D_x u,\mu_t).   \label{Eq:HJB}
\end{align}
Here $H$ is the Hamiltonian
\[ H(t,x,\bar{x}, z,\mu_t) =  \inf\limits_{\alpha \in \mathcal{A}} \left[ b^0(t,x,\bar{x},\mu_t,\alpha) z + f(t,x,\bar{x},\mu_t)  \right]\]
and 
\[\bar{a} = \left[\begin{array}{cc} \sigma \sigma^T & \sigma \rho (\sigma^0)^T \\ \sigma^0 \rho^T \sigma^T & \sigma^0 (\sigma^0)^T \end{array} \right], \qquad \rho = [B,W],\]
where $[\cdot,\cdot]$ stands for the quadratic variation.

We will not discuss the resolvability of the HJB equation (see e.g. \cite{fleming_controlled_2006} or \cite{pham_continuous-time_2009} for a partial review). Given the existence of an optimal control, we can interpret  \eqref{Eq:HJB} from a probabilistic point of view: we have that $u(t,x,\bar{x}) = Y_t^{t,x,\bar{x}}$ where $Y^{t,x,\bar{x}}$ is given by the MKV-FBSDE
\begin{equation}\label{sysparticulelimitteEDP}
\left\lbrace \begin{array}{l}
dX^{t,x,\bar{x}}_s = b(s,X^{t,x,\bar{x}}_s,\mu_s)ds + \sigma(s,X^{t,x,\bar{x}}_s,\mu_s) dB_s\\
d\check{X}^{t,x,\bar{x}}_s = \sigma^0(s,X^{t,x,\bar{x}}_s,\check{X}^{t,x,\bar{x}}_s,\mu_s) dW_s\\ 
-dY^{t,x,\bar{x}}_s = H(s,X^{t,x,\bar{x}}_s,\bX^{t,x,\bar{x}}_s,\bZ^{t,x,\bar{x}}_s,\mu_s) - \bZ^{t,x,\bar{x}}_s dW_s + Z^{t,x,\bar{x}}dB_s\\
X^{t,x,\bar{x}}_t = x, \quad \bX_t^{t,x,\bar{x}} = \bar{x}, \quad Y^{t,x,\bar{x}}_0 = g(X^{t,x,\bar{x}}_T,\bX_T^{t,x,\bar{x}},\mu_T).
\end{array} \right.
\end{equation}

The reader may object that the Hamiltonian $H$ does not satisfy the boundedness condition we have assumed for the analysis of the algorithm (bounded with bounded derivatives w.r.t. the variable $z$). However, some relatively mild assumptions guarantee that the first derivative term $\bZ$ will be bounded. This is almost direct when the boundary condition $g$ is bounded and smooth and proved in \cite{crisan_sharp_2012} when $g$ is Lipschitz and the diffusion matrix uniformly elliptic. Hence, given an estimate on this quantity, one may introduce a modified system in which we replace in the function $(Z,\bZ)$ by $(\psi(Z),\bar{\psi}(\bZ))$, where $\psi, \bar{\psi}$ are truncation functions used to make the value of $Z,\bZ$ satisfy its known estimates, as in \cite{richou_numerical_2011} (if the estimate is not explicitly known, a sequence of functions approximating the identity may be used as in \cite{imkeller_path_2010}, but some additional work would be needed to account for the truncation error). In both cases,
 the truncated problem will then satisfy the needed assumptions and may be solved with the presented Algorithm \ref{algo:KF}, \ref{alg:FirstOrder}.\\
 
 \textbf{\emph{Remark.}} Let us just comment some structural assumptions on our class of exemples. We first emphasize that we do not allow the marked player in \eqref{sysparticule} to influence the dynamic of the other player : this allows to make the forward component in \eqref{sysparticulelimitteEDP} independant on the law of the backward component. Secondly, the diffusion part of the marked player is not allowed to be controlled : this is the reason why the position of the backward variable does not appear in the forward component in \eqref{sysparticulelimitteEDP}. Then, the second component in \eqref{sysparticulelimitteEDP} has no drift : this is because the drift part of the second component in \eqref{sysparticulelimitte} is included in the Hamiltonian in \eqref{sysparticulelimitteEDP}. Finally, we emphasize that here we do not solve the stochastic control problem but we only compute the value of the optimal cost.

\section{Preliminaries}\label{Sec:Preliminaire}

In the following we set a subdivision $T_0=0< \cdots <T_N=T$ of $[0,T]$.\\

\textbf{Artificial dynamics.}
We denote by $\s$ the mapping $s \mapsto \s= T_{k}$ if $s \in [T_k,T_{k+1})$, $k\in\{0,\cdots ,N-1\}$. \\

For any family of probability measures $\eta^1$ and $\eta^2$, one denotes by $P^{\eta^1}$ and  $\tilde{P}^{\eta^2}$ the operators such that, for all $t<s$ in $[0,T]$, for all measurable function $g$ from $\R^d$ to $\R$ and for all $y$ in $\R^d$:
\begin{equation*}
P^{\eta^1}_{t,s}g(y)=\E[g(X_s^{t,y,\eta^1})] \text{ and } \tilde{P}^{\eta^2}_{t,s}g(y)=\E[g(\tX_s^{t,y,\eta^2})] 
\end{equation*}
and $(\mL^{\eta^1}_s)_{t\leq s \leq T }$ and $(\tilde{\mL}^{\eta^2}_{s})_{t\leq s \leq T }$ their infinitesimal generator, where for all $g$ in $C^2(\R^d,\R)$
\begin{equation} \label{Eq:CondimL}
\mL ^{\eta^1}_sg(y):= V_0 (s,y,\la\eta^1_s,\varphi_0\ra) \cdot D_y g(y)+\frac{1}{2}{\rm Tr}[VV^T(s,y,\la\eta^1_s,\varphi\ra)D^2_yg(y)]
\end{equation}
and by definition $\tilde{\mL}^{\eta^2}_{s} = \mL^{\eta^2}_{\s}$. Here $X^{t,y,\eta^1}$ and $\tX^{t,y,\eta^2}$ are the respective solutions of
\begin{eqnarray}
&&dX_s^{t,y,\eta^1}=  \sum_{i=0}^d V_i\left(s,X^{t,y,\eta^1}_s,\la\eta^1_s,\varphi_i\ra\right) dB^i_s,\quad X_t^{t,y,\eta^1} =y,\label{Eq:ArtifDyna1}\\
&&d\tX_s^{t,y,\eta^2}=\sum_{i=0}^d V_i\left(s,\tX^{t,y,\eta^2}_s,\sum_{p=0}^{q-1}[(t-\t)^p/p!]\la \eta^2_{\s},(\tilde{\mL}^{\eta^2})^p\varphi_i\ra\right) dB^i_t,\quad \tX_t^{t,y,\eta^2} =y\label{Eq:ArtifDyna2}.
\end{eqnarray}

Finally, let us define the operator associated to the cubature measure, $Q^{\hmu}$, as 
\begin{equation}\label{Eq:DefOperatorQ}
Q^{\hmu}_{t,s}g(y)=\E_{\Q_{t,s}}[g(\tX_s^{t,y,\hmu})] 
\end{equation}
for all $t<s$ in $[0,T]$ for all $y$ in $\R^d$ and for all measurable function $g$ from $\R^d$ to $\R$. Note that for all $k$ in $\{1,\ldots,N\}$:
\[Q_{0,T_k}^{\hmu} g(x) = \la \hmu_{T_k}, g \ra.\]

\textbf{Multi-index (2).}  Let $\mc{M}$ be defined by \eqref{Eq:defM}. Let $\beta \in \mc{M}$. We define $|\beta| = l$ if  
$\beta= (\beta_1,\ldots,\beta_l)$, $|\beta|_0 := \text{card}\{i:\beta_i =0\}$ and $\|\beta\| := |\beta|+|\beta|_0$. Naturally $|\emptyset|=|\emptyset|_0 = \|\emptyset\|=0$. For every $\beta\neq \emptyset$, we set $-\beta := (\beta_2,\ldots,\beta_l)$ and $\beta-:= (\beta_1,\ldots,\beta_{l-1})$. We set $\beta^+$ the multi-index obtained by deleting the zero components of $\beta$.\\

We will frequently refer to the set of multi-indices of degree at most $l$ denoted by $\mc{A}_l:=\{\beta \in \mc{M}: \| \beta\|\leq l\}$. We define as well its \emph{frontier set}  \(\p\mc{A}:=\{\beta \in \mc{M}\setminus\mc{A}: -\beta\in \mc{A}\}.\) We can easily check that $\p\mc{A}_l \subset \mc{A}_{l+2}\setminus \mc{A}_l$.\\

\textbf{Directional derivatives.}
For notational convenience, let us define the second order operator 
\[\mc{V}_{(0)} := \partial_t + \mc{L},\]
and for $j=\{1,\ldots,d\}$ the operator 
\[\mc{V}_{(j)} := V_j\frac{\partial}{\partial x_j}.\]
where, as announced in the notation section, we do not mark explicitly the time, space and law dependence. For every $\|\beta\|\leq l$ let us define recursively
\begin{equation}
\mc{V}_\beta g:= \begin{cases} g & \text{ if } |\beta| = 0 \\ V_{\beta_1}  \mc{V}_{-\beta}g & \text{ if } |\beta| > 0, \end{cases}
\end{equation}
provided that $g:[0,T]\times \R^d \mapsto \R$ is smooth enough. Hence, for $n\in \NN$ we denote by $\mc{D}^n_b$, the space of such functions $g$ for which  $\mc{V}_{\beta} g$ exists and is bounded for every $\beta \in \mc{A}_n$. 
For any function $g$ in $\mc{D}_b^n$, we set for all $\beta\in \mc{A}_n$,  \[D_\beta g := \frac{\p}{\p y_{\beta_1}}\cdots\frac{\p}{\p y_{\beta_{|\beta|}}}g,\]
where $[\p/\p y_0]$ must be understood as $[\p/\p t]$. \\

\textbf{Iterated integrals.}
For any multi-index $\beta$ and adapted process $g$ we define for all $t<s \in [0,T]$ the \emph{multiple It{\^o} integral} $I_{\beta}^{t,s}[g]$ recursively by
\[I^{t,s}_{\beta}(g) = \begin{cases} g(\tau) & \text{ if } |\beta| = 0 \\ \int_t^s I^{\rho,r}_{\beta-}(g)dr & \text{ if } |\beta|>0 \text{ and } \beta_l = 0 \\ \int_t^s I_{\beta-}^{t,r}(g)d{B}^{\beta_l}_r & \text{ if } |\beta|>0 \text{ and } \beta_l > 0 \end{cases}\]
We will write $I^{t,s}_{\beta}:=I^{t,s}_{\beta}(1)$.\\

The previous notation is very convenient to introduce an It{\^o}-Taylor expansion, that is an analogue of Taylor formula when dealing with It{\^o} processes. The proof follows simply by repeated iteration of It{\^o}'s lemma, and may be found (without the law dependence) in \cite{kloeden_numerical_1992}.
\begin{lemme}
\label{Theo:ItoTaylorExpansion}
Let $t<s \in [0,T]$ and $y\in \R^d$. Let $n\in \NN^*$ and let $g$ in $D_b^n$ Then, for each family of probability measures $\eta$ on $\R^d$, we have have an \emph{Itô-Taylor expansion of order $n$}, that is
\[g(s,X_{s}^{t,y,\eta}) = g(t,y) + \sum_{\beta \in \mc{A}_n} \mc{V}_\beta g(t,y) I_{\beta}^{t,s} + \sum_{\beta \in \p\mc{A}_n} I_{\beta}^{t,s}[\mc{V}_{\beta}g(.,X_{.}^{t,y,\eta})]\]
where $(X_{s}^{t,y,\eta},\ t\leq s\leq T)$ is the solution of \eqref{Eq:ArtifDyna2}.
\end{lemme}

The following lemma is a particular case of a result in \cite{kloeden_numerical_1992}. It follows from integration by parts formula and expectation properties.
\begin{lemme}\label{lemmathatgivestheorderoftheiterateditointegral}
Let $\beta \in \mc{M}$, and let $t_1 <t_2\in [0,T]$. Then for any bounded and measurable functions $g_1$ and $g_2$ in $[t_1,t_2]$ there exists a constant depending only on $\beta$ and $i$ such that
\begin{align*}
\cEX{I^{t_1,t_2}_{\beta}(g_1)I^{t_1,t_2}_{(i)}(g_2)}{\F_{t_1}} \leq  \one{(\beta)^+=i} C(\beta,i) (t_2-t_1)^{(||\beta||+1)/2} \sup_{t_1\leq s \leq t_2}|g_1(s)| \sup_{t_1\leq s \leq t_2}|g_2(s)|,\\
\cEX{I^{t_1,t_2}_{\beta}(g_1)I^{t_1,t_2}_{(0,i)}(g_2)}{\F_{t_1}} \leq  \one{(\beta)^+=i} C'(\beta,i) (t_2-t_1)^{(||\beta||+3)/2} \sup_{t_1\leq s \leq t_2}|g_1(s)| \sup_{t_1\leq s \leq t_2}|g_2(s)|.
\end{align*}

\end{lemme}

\section{Proof of Theorem \ref{MR} under \textbf{(SB)}}\label{Sec:ProofMR}
\subsection{Rate of convergence of the forward approximation: proof of \eqref{Eq:ErrorforSB}}
Here we prove the approximation order of the forward component. Let $\eta$ be a given family of probability measures on $\R^d$. We have the following decomposition of the error:
\begin{eqnarray*}
(P_{T_0,T_N} - Q_{T_0,T_N}^{\hmu}) \phi(x) &=& (P_{T_{0},T_N} - P^{\eta}_{T_{0},T_N})\phi(x) + Q_{T_{0},T_{N-1}}^{\hmu} (P^{\eta}_{T_{N-1},T_N} - Q_{T_{N-1},T_N}^{\hmu}) \phi(x)\\
&& + Q_{T_0,T_{N-2}}^{\hmu} P^{\eta}_{T_{N-2},T_N}\phi(x)-Q_{T_{0},T_{N-1}}^{\hmu} P^{\eta}_{T_{N-1},T_N}\phi(x)\\
&& + Q_{T_0,T_{N-3}}^{\hmu} P^{\eta}_{T_{N-3},T_N}\phi(x)-Q_{T_0,T_{N-2}}^{\hmu} P^{\eta}_{T_{N-2},T_N}\phi(x)\\
&& \vdots\\
&& + P^{\eta}_{T_{0},T_N}\phi(x)- Q_{T_{0},T_1}^{\hmu} P^{\eta}_{T_1,T_{N}}\phi(x) \\
&=&  (P_{T_{0},T_N} - P^{\eta}_{T_{0},T_N})\phi(x) + (P^{\eta}_{T_{0},T_1}- Q_{T_{0},T_1}^{\hmu}) P^{\eta}_{T_1,T_{N}}\phi(x) \\
&& + \sum_{j=1}^{N-1} Q_{T_{0},T_{j}}^{\hmu} \left[(P^{\eta}_{T_{j},T_{j+1}} - Q_{T_{j},T_{j+1}}^{\hmu} )P^{\eta}_{T_{j+1},T_{N}}\phi(x)\right],
\end{eqnarray*}
so that:

\begin{eqnarray}
(P_{T_0,T_N} - Q_{T_0,T_N}^{\hmu}) \phi(x) &=&   (P_{T_{0},T_N} - P^{\eta}_{T_{0},T_N})\phi(x) + (P^{\eta}_{T_{0},T_1}- Q_{T_{0},T_1}^{\hmu}) P^{\eta}_{T_1,T_{N}}\phi(x)\notag \\
&& + \sum_{j=1}^{N-2} Q_{T_{0},T_{j}}^{\hmu} \left[(P^{\eta}_{T_{j},T_{j+1}} - Q_{T_{j},T_{j+1}}^{\hmu} )P^{\eta}_{T_{j+1},T_{N}}\phi(x)\right]\notag\\
&& + Q_{T_{0},T_{N-1}}^{\hmu} \left[(P^{\eta}_{T_{N-1},T_{N}} - Q_{T_{N-1},T_{N}}^{\hmu} )\phi(x)\right] \notag
\end{eqnarray}
That is, the global error is decomposed as a sum of local errors (i.e., as the sum of errors on each interval). These local errors can be also split. Let us define the function

\begin{equation}\label{Eq:DefPsi}
\psi(T_k,x) := P^{\eta}_{T_k,T_{N}}\phi(x). 
\end{equation}
We emphasize that for all $t$ in $[0,T)$, $y\mapsto \psi(t,y)$ is $C_b^{\infty}$. Indeed, this function can be seen as the solution of the PDE
\begin{equation}\label{PDEfirst}
\left\lbrace
\begin{array}{ll}
\p_t \psi(t,y) + \mL^\eta \psi(t,y)=0,\text{ on } [0,T]\times \R^d\\
\psi(T,y)=\phi(y)
\end{array}
\right.
\end{equation}
taken at time $T_k$, where $\mL^{\eta}$ is defined by \eqref{Eq:CondimL}. The claim follows from Lemma \ref{PropPDE}. On each interval $\Delta_{T_k},\ k=1,\cdots,N-1$, the local error $(P^{\eta}_{T_{k},T_{k+1}} - Q^{\hmu}_{T_{k},T_{k+1}})\psi(T_{k+1},x)$ is expressed as

\begin{eqnarray}
(P^{\eta}_{T_{k},T_{k+1}} - Q^{\hmu}_{T_{k},T_{k+1}})\psi(T_{k+1},x) &=& (P^{\eta}_{T_{k},T_{k+1}} - \tilde{P}^{\hmu}_{T_{k},T_{k+1}})\psi(T_{k+1},x)\label{reg1}\\
&&\quad  +  (\tilde{P}^{\hmu}_{T_{k},T_{k+1}} - Q_{T_{k},T_{k+1}}^{\hmu})\psi(T_{k+1},x)\label{reg3}.
\end{eqnarray} 
Error \eqref{reg1} can be identified as a frozen (in time) error (and so, a sort of weak Euler error) plus an approximation error, in the sense that in step $k$,  the measure $\mu_{T_k}$ is approximated by the discrete law $\hmu_{T_k}$. Then, \eqref{reg3} is a (purely) cubature error on one step, and we have:

\begin{eqnarray*}
&&(P_{T_0,T_N} - Q_{T_0,T_N}^{\hmu}) \phi(x) \\
&& = (P_{T_{0},T_N} - P^{\eta}_{T_{0},T_N})\phi(x)\\
&& \quad + (P^{\eta}_{T_{0},T_{1}} - \tilde{P}^{\hmu}_{T_{0},T_{1}})\psi(T_{1},x) + \sum_{j=1}^{N-2} Q_{T_{0},T_{j}}^{\hmu} \left[(P^{\eta}_{T_{j},T_{j+1}} - \tilde{P}^{\hmu}_{T_{j},T_{j+1}})\psi(T_{j+1},x)\right]\\
&& \quad + (\tilde{P}^{\hmu}_{T_{0},T_{1}} - Q_{T_{0},T_{1}}^{\hmu})\psi(T_{1},x) + \sum_{j=1}^{N-2} Q_{T_{0},T_{j}}^{\hmu} \left[ (\tilde{P}^{\hmu}_{T_{j},T_{j+1}} - Q_{T_{j},T_{j+1}}^{\hmu})\psi(T_{j+1},x)\right]\\
&& \quad + Q_{T_{0},T_{N-1}}^{\hmu} \left[(P^{\eta}_{T_{N-1},T_{N}} - Q_{T_{N-1},T_{N}}^{\hmu} )\phi(x)\right]
\end{eqnarray*}
$\newline$

We have the two following Claims:
\begin{claim}\label{c1}
There exists a positive constant $C(T,V_{0:d})$ depending on the regularity of the $V_{0:d}$ and on $T$ such that,  for all $y$ in $\R^d$, for all family of probability measures $\eta$, for all $k \in \{1,\cdots,N-1\}$, one has:
\begin{eqnarray*}
&&\left|(P^{\eta}_{T_{k},T_{k+1}} - \tilde{P}^{\hmu}_{T_{k},T_{k+1}})\psi(T_{k+1},y)\right| \\
&&  \leq C(T,V_{0:d}) ||\psi(T_{k+1},.)||_{2,\infty} \sum_{i=0}^{d} \int_{T_k}^{T_{k+1}} \left|\la\eta_t,\varphi_i\ra-\sum_{p=0}^{q-1} [(t-T_k)^p/p!]\la \eta_{T_k} ,(\mL^{\eta})^p\varphi_i\ra\right| dt \\
&& \quad + C(T,V_{0:d}) ||\psi(T_{k+1},.)||_{2,\infty} \sum_{i=0}^{d}\sum_{p=0}^{q-1} [\Delta_{T_{k+1}}^p/p!]\left|\la \eta_{T_k}, (\mL^\eta)^p\varphi_i\ra - \la
\hmu_{T_k} ,(\mL^{\hmu})^p\varphi_i\ra \right|.
\end{eqnarray*}
\end{claim}
\begin{proof}
We deduce the claim by applying Lemma \ref{OnestepEuler} to $y\mapsto \psi(T_{k},y)$ for each $k \in \{0,\cdots,N-2\}$.
\end{proof}

\begin{claim}\label{c2}
There exists a positive constant $C(T,V_{0:d},d,m)$ depending on the regularity of the $V_{0:d}$ on the dimension $d$ and the cubature order $m$ such that: for all $y$ in $\R^d$, for all family of probability measure $\eta$, for all $k \in \{0,\cdots,N-1\}$, one has:
\begin{eqnarray*}
 \left|(\tilde{P}^{\hmu}_{T_{k},T_{k+1}} - Q_{T_{k},T_{k+1}}^{\hmu})\psi(T_{k+1},y) \right|\leq C(T,V_{0:d},d,m) \sum_{l=m+1}^{m+2} ||\psi(T_{k+1},\cdot)||_{l,\infty}\Delta_{T_{k+1}}^{\frac{l}{2}}.
\end{eqnarray*}
\end{claim}
\begin{proof}
The claim follows, by applying Lemma \ref{OnestepCub} with $\hmu$ to the function $y\in \R^d \mapsto \psi(T_{k},y)$ for each $k\in \{1,\cdots,N-1\}$.
\end{proof}

$\newline$
Then, by plugging estimates of Claims \ref{c1} and \ref{c2} in the error expansion we deduce that:
\begin{eqnarray}  
&&\left|(P_{T_0,T_N} - Q_{T_0,T_N}^{\hmu}) \phi(x)\right|\label{errorrra}\\
&& \leq C(T,V_{0:d}) \sum_{j=0}^{N-1}||\psi(T_{j+1},\cdot)||_{2,\infty} \int_{T_{j}}^{T_{j+1}}\sum_{i=0}^{d} \left|\la\eta_t,\varphi_i\ra-\sum_{p=0}^{q-1} [(t-T_k)^p/p!]\la \eta_{T_{j}}, (\mL^\eta)^p\varphi_i\ra\right| dt \notag\\
&& \quad +  C(T,V_{0:d})  \sum_{j=0}^{N-1}||\psi(T_{j+1},\cdot)||_{2,\infty} \sum_{p=0}^{q-1} \Delta_{T_{j+1}}^{p+1} \sum_{i=0}^{d} \left|\la\eta_{T_{j}},  (\mL^{\eta})^p\varphi_i\ra - \la \hmu_{T_{j}}, (\mL^{\hmu})^p\varphi_i\ra \right|\notag\\
&& \quad + C(T,V_{0:d},d,m)\sum_{j=0}^{N-1}\sum_{l=m+1}^{m+2}||\psi(T_{j+1},\cdot)||_{l,\infty}\Delta_{T_{j+1}}^{\frac{l}{2}}+\left|(P_{T_{0},T_N} - P^{\eta}_{T_{0},T_N})\phi(x)\right|\notag
\end{eqnarray}

$\newline$

Up to now, the analysis holds for any family of probability measures $\eta$. \emph{The key point in the proof is to note that we can actually choose $\eta=\mu$}, that is, the law of the solution of the forward component in \eqref{NLFBSDE}. In that case for all measurable function $g$:
$$\la \eta_\cdot, g\ra = \la \mu_\cdot, g\ra = \E[g(X^x_\cdot)].$$
Then, \eqref{errorrra} becomes:
\begin{eqnarray}
&&\left|(P_{T_0,T_N} - Q_{T_0,T_N}^{\hmu}) \phi(x)\right|  \label{forwass}\\
&& \leq C(T,V_{0:d}) \sum_{j=0}^{N-1}||\psi(T_{j+1},\cdot)||_{2,\infty} \int_{T_{j}}^{T_{j+1}} \sum_{i=0}^{d} \left|\E\left[\varphi_i(X_{t}^{x,\mu})\right]-\sum_{p=0}^{q-1} [(t-T_k)^p/p!]\E\left[ (\mL^\mu)^p\varphi_i(X_{T_{j}}^{x,\mu})\right]\right|dt \notag\\
&& \quad + C(T,V_{0:d})  \sum_{j=0}^{N-1}||\psi(T_{j+1},\cdot)||_{2,\infty}\sum_{p=0}^{q-1} \Delta_{T_{j+1}}^{p+1} \sum_{i=0}^{d} \left|\E\left[(\mL^{\mu})^p\varphi_i(X_{T_j}^{x,\mu})\right]-\E \left[(\mL^{\hmu})^p\varphi_i(\hX_{T_{j}}^{x,\hmu})\right] \right|\notag\\
&& \quad +  C(T,V_{0:d},d,m)\sum_{j=0}^{N-1}\sum_{l=m+1}^{m+2}||\psi(T_{j+1},\cdot)||_{l,\infty}\Delta_{T_{j+1}}^{\frac{l}{2}}\notag
\end{eqnarray}
since $P_{s,t}  = P^{\mu}_{s,t}$ for all $s<t \in [0,T]$, by definition. Now we have:
\begin{claim}\label{c4}
For any $k \in \{0,\cdots,N-1\}$, and for all $t$ in $[T_k;T_{k+1})$ there exists a positive constant $C(d,V_{0:d})$ such that:
\begin{eqnarray*}
\int_{T_{k}}^{T_{k+1}} \sum_{i=0}^{d}\left|\E\left[\varphi_i(X_{t}^{x,\mu})\right]-\sum_{p=0}^{q-1} [(t-T_k)^p/p!]\E \left[(\mL^{\mu})^p\varphi_i(X_{T_{k}}^{x,\mu})\right]\right|dt \leq C(d,V_{0:d})\|\varphi\|_{2q,\infty}\Delta_{T_{k+1}}^{q+1}
\end{eqnarray*}
\end{claim}
\begin{proof}
This follows by Itô-Taylor expansion of order $q$ of $\varphi_i(X^{X_{T_k},\mu})$ for each $i=0,\cdots,d$ and for any $k$ in $\{0,\cdots,N-1\}$.
\end{proof}
Therefore,

\begin{eqnarray}
&&\left|(P_{T_0,T_N} - Q_{T_0,T_N}^{\hmu}) \phi(x)\right|  \label{Eq:AlmostFinalErrorA} \\
&& \leq C(T,V_{0:d})\sum_{j=0}^{N-1}||\psi(T_{j+1},\cdot)||_{2,\infty}\sum_{p=0}^{q-1} \Delta_{T_{j+1}}^{p+1} \sum_{i=0}^{d} \left|(P_{T_0,T_N} - Q_{T_0,T_N}^{\hmu}) \mL^p\varphi_i(x)\right| \notag\\
&& \quad + C(T,V_{0:d},d) \|\varphi\|_{2q,\infty}\sum_{j=0}^{N-1}||\psi(T_{j+1},\cdot)||_{2,\infty}\Delta_{T_{j+1}}^{q+1}  \notag\\
&& \quad +  C(T,V_{0:d},d,m)  \sum_{j=0}^{N-1}\sum_{l=m+1}^{m+2}||\psi(T_{j+1},\cdot)||_{l,\infty}\Delta_{T_{j+1}}^{\frac{l}{2}}.\notag
\end{eqnarray}
Thanks to estimate \eqref{Eq:Bound_derivatives_linear_u} in Lemma \ref{PropPDE}, for all $n$ in $\mathbb{N}$, we have the following bound on the supremum norm of the derivatives of $\psi$ up to order $n$:
\begin{eqnarray}\label{Eq:GradBoundsmooth}
&&||\nabla_y^n \psi(t,\cdot) ||_{\infty} \leq C(T,V_{0:d}) ||\phi||_{n,\infty}.
\end{eqnarray}
By plugging this bound in \eqref{Eq:AlmostFinalErrorA} we get
\begin{eqnarray*}
&&\left|(P_{T_0,T_N} - Q_{T_0,T_N}^{\hmu}) \phi(x) \right|\\
&& \leq   C(T,V_{0:d})||\phi||_{2,\infty} \sum_{j=0}^{N-1} \sum_{p=0}^{q-1}\Delta_{T_{j+1}}^{p+1}\sum_{i=0}^{d} \left|(P_{T_0,T_j}-Q^{\hmu}_{T_0,T_j})\mL^p\varphi_i(x)\right|\notag\\
&& \quad + C(T,V_{0:d},d,m)( ||\phi||_{2m+2,\infty}+||\varphi||_{2q,\infty} )\left(\frac{1}{N}\right)^{q\wedge (m-1)/2}.\notag
\end{eqnarray*}
It should be remarked that the term on the right hand side is controlled in terms of the approximation error itself, acting on the functions $\mL^p\varphi_i,\ i=1,\cdots,d,\ p=0,\cdots,q-1$, from step $0$ to $j$ for any $j$ in $\{1,\cdots,N-1\}$. To proceed, the argument is the following one: since these bounds hold for all (at least) $\phi$ smooth enough, one can let $\phi=\varphi_0$, and use the discrete Gronwall Lemma to get the following bound on $\varphi_0$:
\begin{eqnarray*}
\left|(P_{T_0,T_N} - Q_{T_0,T_N}^{\hmu}) \varphi_0(x) \right| &\leq & C(T,V_{0:d},m,||\phi||_{m+2,\infty},||\varphi||_{2q,\infty},||\varphi_0||_{2q +m,\infty})\left(\frac{1}{N}\right)^{q\wedge [(m-1)/2]}\\
&& \quad \times \bigg\{\sum_{j=0}^{N-1} \sum_{p=0}^{q-1}\Delta_{T_{j+1}}^{p+1}\sum_{i=1}^{d} \left|(P_{T_0,T_j}-Q^{\hmu}_{T_0,T_j})\mL^p\varphi_i(x)\right|\\
&& \qquad + \sum_{j=0}^{N-1} \sum_{p=1}^{q-1}\Delta_{T_{j+1}}^{p+1} \left|(P_{T_0,T_j}-Q^{\hmu}_{T_0,T_j})\mL^p\varphi_0(x)\right|\bigg\}
\end{eqnarray*}
It is clear that, by iterating this argument (i.e., by letting $\phi=\varphi_1$ and then $\phi=\varphi_2$,$\ldots$, $\phi = \mL \varphi_0 $, etc...) we obtain:

\begin{eqnarray*}
\left|(P_{T_0,T_N} - Q_{T_0,T_N}^{\hmu}) \phi(x) \right|& \leq & C(T,V,d,q,m,||\phi||_{m+2,\infty},||\varphi||_{2q+m,\infty}) \left(\frac{1}{N}\right)^{q\wedge[ (m-1)/2]}. 
\end{eqnarray*}
This concludes the proof \eqref{Eq:ErrorforSB} at time $T$.  From these arguments, we easily deduce that the estimate holds for any $T_k$, $k=1,\ldots,N$.\qed

\subsection{Rate of convergence for the backward approximation: proof of \eqref{Eq:rate1} and \eqref{Eq:rate3}}\label{subsec:backwardSB}
Here we prove the approximation order of the backward component. Before presenting the proof, we introduce some notations. Let us define the Brownian counterparts of $\hTheta_{k+1,k}, \hTheta_{k}$ and $\hzeta_k$ given in step \ref{Eq:DefHThetaK1} in Algorithm \ref{alg:FirstOrder} and steps \ref{Eq:DefHThetaK2}, \ref{Eq:DefHThetaK}  and \ref{Eq:DefHZeta} in Algorithm \ref{alg:SecondOrder}. For all family of probability measures $\eta$ we set
\begin{eqnarray}
&& \Theta_{k}^\eta(y):= \left(T_{k},y,u(T_{k},y),v(T_{k},y),\la \eta_{T_k},\varphi_f [\cdot,u(T_k,\cdot)] \ra\right), \label{Eq:DefThetaContinue}\\
&&\bTheta_{k+1,k}^{\eta^1,\eta^2}(y):= \left(T_{k+1},X_{T_{k+1}}^{T_k,y,\eta^1},u(T_{k+1},X_{T_{k+1}}^{T_k,y,\eta^1}),v(T_{k},y),\la \eta^2_{T_k}, \varphi_f [\cdot,u(T_k,\cdot)]\ra\right),\notag
\end{eqnarray}
and
\begin{equation*}
\zeta_k = 4 \frac{B_{T_{k+1}}-B_{T_k}}{\Delta_{T_{k+1}}} - 6 \frac{\int_{T_k}^{T_{k+1}} (s-T_k)dB_{s}}{\Delta^2_{T_{k+1}}}.
\end{equation*}

The proof uses extensively the regularity of the function $u$. From Lemma \ref{RegNonLinearPDE}, for all $t\in [0,T)$, the function $y\in \R^d \mapsto u(t,y)$ is $C_b^{\infty}$ with uniform bounds in time. In the elliptic case the same situation holds, although the bounds depend on time and blow up in the boundary. Hence, we keep track of the explicit dependence of each error term on $u$ and its derivatives in such a way that the proof is simplified for the elliptic case.

Moreover, we will expand and bound terms of the form $y\mapsto f(\cdot,y,u(\cdot,y),\mc{V}u(\cdot,y),\cdot)$. When differentiating such a term, the bounds involve the product of the derivatives of $u$ with respect to the space variable. Namely, the $r^{{\rm th}}$ differentiation of $f$ involves a product of at most $r+1$ derivatives of $u$. To keep track of the order of the derivatives that appear in the bound, we introduce the set of positive integers for which their sum is less than or equal to $r$: $\mc{I}(l,r)= \{I=(I_1,\ldots,I_l) \in \{1,\ldots,r\}^l: \sum_j I_j \leq r  \}$, and define the quantity:
\begin{equation}
\label{Eq:DefinitionM}
M_u(r,s):=  \sum_{ l=1 }^{r} \sum_{I\in \mc{I}(l,r)}\prod_{j=1}^{l} ||u(s,.)||_{I_j,\infty}.
\end{equation}

\emph{(1) Proof of the order of convergence for the first order algorithm (Algorithm \ref{alg:FirstOrder}).}

Let $k \in \{1,\ldots,N-1\}$. We first break the error between $u$ and $\hu^1$ as follows:
\begin{eqnarray} 
u(T_k,\hX^{\pi}_{T_k})  - \hu^1(T_k,\hX^{\pi}_{T_k}) 
&=&  u(T_k,\hX^{\pi}_{T_k}) - \EX{u(T_{k+1},X^{T_k,\hX^{\pi}_{T_k},\mu}_{T_{k+1}}) + \Delta_{T_{k+1}} f(\bTheta^{\mu,\mu}_{k+1,k}(\hX^{\pi}_{T_k}))}\label{Eq:Erroru11} \\
&&+ \E \left[u(T_{k+1},X^{T_k,\hX^{\pi}_{T_k},\mu}_{T_{k+1}}) + \Delta_{T_{k+1}} f(\bTheta^{\mu,\mu}_{k+1,k}(\hX^{\pi}_{T_k}))\right] \label{Eq:Erroru12}\\
&& \quad - \E_{\Q_{T_k,T_{k+1}}}\left[u(T_{k+1},X^{T_k,\hX^{\pi}_{T_k},\hmu}_{T_{k+1}})+ \Delta_{T_{k+1}} f(\bTheta^{\hmu,\mu}_{k+1,k}(\hX^{\pi}_{T_k}))\right]\notag\\
&&  +\E_{\Q_{T_k,T_{k+1}}} \bigg[u(T_{k+1},X^{T_k,\hX^{\pi}_{T_k},\hmu}_{T_{k+1}}) - \hu^1(T_{k+1},X^{T_k,\hX^{\pi}_{T_k},\hmu}_{T_{k+1}})\label{Eq:Erroru13}\\
&&\quad  + \Delta_{ T_{k+1}} \left( f(\bTheta^{\hmu,\mu}_{k+1,k}(\hX^{\pi}_{T_k})) -  f(\ApproxVarT{\Theta}_{k+1,k})\right)\bigg].\notag
\end{eqnarray}
Similarly, we can expand the error between $v$ and $\hv^1$ as:
\begin{eqnarray}
&&\Delta_{T_{k+1}} \left[v(T_k,\hX^{\pi}_{T_k})  - \hv^1 (T_k,\hX^{\pi}_{T_k}) \right]\notag\\ 
&&=\Delta_{T_{k+1}}  v(T_k,\hX^{\pi}_{T_k}) - \EX{ u\left(T_{k+1},X_{T_{k+1}}^{T_k,\hX^{\pi}_{T_k},\mu}\right)\Delta B_{T_{k+1}} } \label{Eq:Errorv11}\\
&& \quad  + \EX{ u\left(T_{k+1},X_{T_{k+1}}^{T_k,\hX^{\pi}_{T_k},\mu}\right) \Delta B_{T_{k+1}} }-\ApproxEx[T_k,T_{k+1}]{ u\left(T_{k+1},X_{T_{k+1}}^{T_k,\hX^{\pi}_{T_k},\hmu}\right) \Delta B_{T_{k+1}} }  \label{Eq:Errorv12}\\
&& \quad  + \E_{\Q_{T_k,T_{k+1}}} \left( \left[ u\left(T_{k+1},X_{T_{k+1}}^{T_k,\hX^{\pi}_{T_k},\hmu}\right) - \hu^1\left(T_{k+1},X_{T_{k+1}}^{T_k,\hX^{\pi}_{T_k},\hmu}\right) \right] \Delta B_{T_{k+1}}\right).\label{Eq:Errorv13}
\end{eqnarray}

Then, at each step, the approximation error on the backward variables can be expanded as: a first term \eqref{Eq:Erroru11} and \eqref{Eq:Errorv11}, corresponding to scheme errors; a second term, \eqref{Eq:Erroru12} and \eqref{Eq:Errorv12}, corresponding to generalized cubature errors and can be viewed as one step versions of the forward error \eqref{Eq:ErrorforSB} in Theorem \ref{MR} ; and a third term, \eqref{Eq:Erroru13} and \eqref{Eq:Errorv13}, which are propagation errors.

Let us explain how the proof works. We will bound separately each error: the scheme, cubature and propagation errors. Each bound is summarized in a Claim (respectively Claims \ref{Cl:back1-1}, \ref{Cl:back1-2} and \ref{Cl:back1-2} below). Then, we will deduce the dynamics of the error at step $k$, $\mc{E}_u^1(k)$ defined as \eqref{Eq:def_epsilon} and conclude with a Gronwall argument.

The first claim below gives the bounds on the scheme errors.
\begin{claim}\label{Cl:back1-1}
There exists a constant $C$ depending on the regularity of $V_{0:d}$ and $f$ (and not on $k$) such that the scheme errors \eqref{Eq:Erroru11} and \eqref{Eq:Errorv11} are bounded by:
\begin{eqnarray*}
&&\left|u(T_k,\hX^{\pi}_{T_k}) - \EX{u(T_{k+1},X^{T_k,\hX^{\pi}_{T_k},\mu}_{T_{k+1}}) + \Delta_{T_{k+1}} f(\bTheta^{\mu,\mu}_{k+1,k}(\hX^{\pi}_{T_k}))}\right| \leq C \sup_{s\in [T_k,T_{k+1}]}||u(s,\cdot)||_{4,\infty} \Delta^2_{T_{k+1}}\\
&&\left|\Delta_{T_{k+1}}  v(T_k,\hX^{\pi}_{T_k}) - \EX{ u\left(T_{k+1},X_{T_{k+1}}^{T_k,\hX^{\pi}_{T_k},\mu}\right)\Delta B_{T_{k+1}} }\right| \leq C \sup_{s\in [T_k,T_{k+1}]}||u(s,\cdot)||_{3,\infty} \Delta^2_{T_{k+1}}
\end{eqnarray*}
\end{claim}
\begin{proof}
The proof of the first estimate follows from a second order Itô-Taylor expansion. Applying Lemma \ref{Theo:ItoTaylorExpansion} with $n=2$ to $u$ and taking the expectation leads to:
\begin{equation}\label{eq:cl1}
\EX{u(T_{k+1},  X^{T_k,y,\mu}_{T_{k+1}})} = u(T_{k},y) + \Delta_{T_k} \mc{V}_{(0)}u(T_{k},  y)  +   \sum_{\beta \in \p\mc{A}_2} \E\left(I_{\beta}^{T_k,T_{k+1}}[\mc{V}_{\beta} u(\cdot,X^{T_k,y,\mu}_{\cdot})] \right)
\end{equation}
and applying again Lemma \ref{Theo:ItoTaylorExpansion} with $n=1$ to $ \mc{V}_{(0)}u$ and taking the expectation gives:
\begin{equation}\label{eq:cl2}
\EX{\mc{V}_{(0)}u(T_{k+1},  X^{T_k,y,\mu}_{T_{k+1}})} = \mc{V}_{(0)}u(T_{k},y)  +   \sum_{\beta \in \p\mc{A}_1} \E\left(I_{\beta}^{T_k,T_{k+1}}[\mc{V}_{\beta} \mc{V}_{(0)}u(\cdot,X^{T_k,y,\mu}_{\cdot})] \right).
\end{equation}
Now, note that since $u$ is the solution of PDE \eqref{Eq:NonLinearPDE} we have $f = \mc{V}_{(0)}u $. So that, by combining \eqref{eq:cl1}, \eqref{eq:cl2} and estimate of Lemma \ref{lemmathatgivestheorderoftheiterateditointegral}, we obtain
\begin{eqnarray*}
&&\left| u(T_k,\hX^{\pi}_{T_k}) - \EX{u(T_{k+1},X^{T_k,\hX^{\pi}_{T_k},\mu}_{T_{k+1}}) + \Delta_{T_{k+1}} f(\bTheta^{\mu,\mu}_{k+1,k}(\hX^{\pi}_{T_k}))} \right|\\
 && \quad \leq 2\sum_{\beta \in \p\mc{A}_2} \E\left(I_{\beta}^{T_k,T_{k+1}}[\mc{V}_{\beta} u(\cdot,X^{T_k,y,\mu}_{\cdot})] \right)\\
&& \quad \leq  C(T,V_{0:d},f) \sup_{s\in [T_k,T_{k+1}]}||u(s,\cdot)||_{4,\infty} \Delta_{T_{k+1}}^2.
\end{eqnarray*}
This concludes the proof of the first estimate. The proof of the second estimate is similar. We first apply an Itô-Taylor expansion of Lemma \ref{Theo:ItoTaylorExpansion} with $n=1$ on $u$. Then, by noticing that $\Delta B_{T_{k+1}} = (I^{T_k,T_{k+1}}_{(1)},\ldots,I^{T_k,T_{k+1}}_{(d)})^T$ and by multiplying by $I^{T_k,T_{k+1}}_{(j)}$ the previous expansion of $u$, and taking the expectation gives, thanks to It\^{o}'s Formula
\begin{equation*}
\EX{u(T_{k+1},  X^{T_k,y,\mu}_{T_{k+1}})I_{(j)}^{T_k,T_{k+1}}} =  \mc{V}_{(j)}  u(T_{k},y) \Delta _{T_{k+1}}  + \sum_{\beta \in \p\mc{A}_1} \E \left(I_{\beta}^{T_k,T_{k+1}}[\mc{V}_{\beta} u(\cdot,X^{T_k,y,\mu}_{\cdot})]I_{(j)}^{T_k,T_{k+1}}\right),
\end{equation*}
 for $j=1,\ldots,d$ and where the first term in the right hand side is the bracket between the stochastic integrals. The last term is controlled by using Lemma \ref{lemmathatgivestheorderoftheiterateditointegral}. Recalling that the $j$-th component of the function $v$ is given by $\mc{V}_{(j)}u$ and reordering the terms we obtain the second inequality.
\end{proof}

We now turn to bound the cubature like error terms \eqref{Eq:Erroru12} and  \eqref{Eq:Errorv12}. This is summarized by:
\begin{claim}\label{Cl:back1-2}
There exist two constants $C$, depending on $d$, $q$, $T$, $m$, and the regularity of $V_{0:d}$ and $\varphi_{0:d}$ (and not on $k$), and $C'$, depending in addition on the regularity of $f$, such that:
\begin{eqnarray*}
&&\bigg| \E \left[u(T_{k+1},X^{T_k,\hX^{\pi}_{T_k},\mu}_{T_{k+1}}) + \Delta_{T_{k+1}} f(\bTheta^{\mu,\mu}_{k+1,k}(\hX^{\pi}_{T_k}))\right] - \E_{\Q_{T_k,T_{k+1}}}\left[u(T_{k+1},X^{T_k,\hX^{\pi}_{T_k},\hmu}_{T_{k+1}})+ \Delta_{T_{k+1}} f(\bTheta^{\hmu,\mu}_{k+1,k}(\hX^{\pi}_{T_k}))\right]\bigg|\\
&&\quad \leq C \bigg( ||u(T_{k+1},\cdot)||_{2,\infty} \left[ \Delta_{T_{k+1}}^{q+1}+ \Delta_{T_{k+1}} N^{- [ (m-1) \wedge 2q]/2} \right] + ||u(T_{k+1},\cdot)||_{m+1,\infty} \Delta_{T_{k+1}}^{(m+1)/2}\\
&&\qquad + ||u(T_{k+1},\cdot)||_{m+2,\infty} \Delta_{T_{k+1}}^{(m+2)/2}\bigg) + C' \bigg(M_u(2,T_{k+1}) \left[ \Delta_{T_{k+1}}^{q+1}+\Delta_{T_{k+1}} N^{- [ (m-1) \wedge 2q]/2} \right]\\
&&\qquad +  M_u(m+1,T_{k+1}) \Delta_{T_{k+1}}^{(m+1)/2} +M_u(m+2,T_{k+1}) \Delta_{T_{k+1}}^{(m+2)/2}\bigg)\\
&&\left| \EX{ u\left(T_{k+1},X_{T_{k+1}}^{T_k,\hX^{\pi}_{T_k},\mu}\right) \Delta B_{T_{k+1}} }-\ApproxEx[T_k,T_{k+1}]{ u\left(T_{k+1},X_{T_{k+1}}^{T_k,\hX^{\pi}_{T_k},\hmu}\right) \Delta B_{T_{k+1}} }\right|\\
&& \quad \leq C \bigg( ||u(T_{k+1},\cdot)||_{3,\infty} \left[ \Delta_{T_{k+1}}^{q+1}+ \Delta_{T_{k+1}}^2 N^{- [ (m-1) \wedge 2q]/2} \right] + ||u(T_{k+1},\cdot)||_{m,\infty} \Delta_{T_{k+1}}^{(m+1)/2} \\
&& \qquad + ||u(T_{k+1},\cdot)||_{m+1,\infty} \Delta_{T_{k+1}}^{(m+2)/2}\bigg)
\end{eqnarray*}
\end{claim}
\begin{proof}
Note that the $r^{{\rm th}}$ derivative of the function $y\mapsto f(\cdot,y,u(\cdot,y),\cdot,\cdot)$ is bounded by $C'M_u(r,\cdot)$ defined by \eqref{Eq:DefinitionM}. Then, the proof of the first assertion follows from  \eqref{Eq:Lem:OneStepAll1} in Lemma \ref{Lem:OneStepAll} applied to $u$ and $f$ and the second assertion from  \eqref{Eq:Lem:OneStepAll2} in Lemma \ref{Lem:OneStepAll} applied to $u$.
\end{proof}

Finally, an estimate on the propagation error \eqref{Eq:Erroru13}  is given by:
\begin{claim}\label{Cl:back1-3}
There exists a constant $C$ depending on $d$, $q$, $T$, $m$, and the regularity of $V_{0:d}$ and $\varphi_{0:d}$ (and not on $k$) such that:
\begin{eqnarray*}
&&\left|\E_{\Q_{T_k,T_{k+1}}} \bigg[u(T_{k+1},X^{T_k,\hX^{\pi}_{T_k},\hmu}_{T_{k+1}}) - \hu^1(T_{k+1},X^{T_k,\hX^{\pi}_{T_k},\hmu}_{T_{k+1}})  + \Delta_{ T_{k+1}} \left( f(\bTheta^{\hmu,\mu}_{k+1,k}(\hX^{\pi}_{T_k})) -  f(\ApproxVarT{\Theta}_{k+1,k})\right)\bigg]\right|\\
&&\quad \leq (1+C\Delta_{T_{k+1}} )\mE^1_{u}(k+1) + C\bigg(||u(T_{k+1},\cdot)||_{m+2,\infty}\Delta_{T_{k+1}} N^{-[(m-1)/2]\wedge q} \\
&&\qquad\quad  +   ||u(T_{k+1},\cdot)||_{3,\infty} \left[\Delta_{T_{k+1}}^2 + \Delta_{T_{k+1}}^{q+1} +  \Delta_{T_{k+1}}^2 N^{- [ (m-1) \wedge 2q]/2}\right] \\
&&\qquad \quad + ||u(T_{k+1},\cdot)||_{m+1,\infty} \Delta_{T_{k+1}}^{(m+1)/2}+ ||u(T_{k+1},\cdot)||_{m+2,\infty} \Delta_{T_{k+1}}^{(m+2)/2}\bigg)\notag
\end{eqnarray*}
\end{claim}
\begin{proof}

Let us start by expanding the $f$ term. We get from the mean value theorem that there exist three random variable $\Psi_1,\Psi_2,\Psi_3$, respectively bounded by $\| \partial_{y'}  f \|_{\infty}$, $\| \partial_{z}  f \|_{\infty}$  and $\| \partial_{w}  f \|_{\infty} $ almost surely, depending on each argument of $\ApproxVarT{\Theta}_{k+1,k}$ and $\bTheta^{\hmu,\mu}_{k+1,k}$ such that:
\begin{eqnarray*}
&&\Delta_{ T_{k+1}} \left( f(\bTheta^{\hmu,\mu}_{k+1,k}(\hX^{\pi}_{T_k})) -  f(\ApproxVarT{\Theta}_{k+1,k})\right)\\
&& =   \Delta_{T_{k+1}} \Psi_1 \left(u(T_{k+1},X^{T_k,\hX^{\pi}_{T_k},\hmu}_{T_{k+1}}) - \hu^1(T_{k+1},X^{T_k,\hX^{\pi}_{T_k},\hmu}_{T_{k+1}})\right) + \Delta_{T_{k+1}} \Psi_2 \left(v(T_k,\hX^{\pi}_{T_k})  - \hv^1 (T_k,\hX^{\pi}_{T_k}) \right)\notag\\
&& \qquad + \Delta_{T_{k+1}} \Psi_3 \left(\la \mu_{T_{k+1}}, \varphi_f[\cdot,u(T_{k+1},\cdot)]\ra - \la \hmu_{T_{k+1}}, \varphi_f[\cdot,\hu^1(T_{k+1},\cdot)]\ra\right).
\end{eqnarray*}
Now,  we can use the error expansion \eqref{Eq:Errorv11}, \eqref{Eq:Errorv12} and \eqref{Eq:Errorv13} of $\Delta_{T_{k+1}} \left(v(T_k,\hX^{\pi}_{T_k})  - \hv^1 (T_k,\hX^{\pi}_{T_k}) \right)$ together with the second assertion of Claims \ref{Cl:back1-1} and \ref{Cl:back1-2} to get
\begin{eqnarray}
&&\bigg|\E_{\Q_{T_k,T_{k+1}}} \bigg[u(T_{k+1},X^{T_k,\hX^{\pi}_{T_k},\hmu}_{T_{k+1}}) - \hu^1(T_{k+1},X^{T_k,\hX^{\pi}_{T_k},\hmu}_{T_{k+1}})  + \Delta_{ T_{k+1}} \left( f(\bTheta^{\hmu,\mu}_{k+1,k}(\hX^{\pi}_{T_k})) -  f(\ApproxVarT{\Theta}_{k+1,k})\right)\bigg]\bigg|\notag\\
&&\leq \left| \E_{\Q_{T_k,T_{k+1}}} \left[ \left(u(T_{k+1},X^{T_k,\hX^{\pi}_{T_k},\hmu}_{T_{k+1}}) - \hu^1(T_{k+1},X^{T_k,\hX^{\pi}_{T_k},\hmu}_{T_{k+1}})\right) (1+\Psi_1 \Delta_{T_{k+1}}+\Psi_2\Delta B_{T_{k+1} })\right]\right|\notag\\
&&\quad +  C  ||u(T_{k+1},\cdot)||_{3,\infty} \left[\Delta_{T_{k+1}}^2 + \Delta_{T_{k+1}}^{q+1}+ \Delta_{T_{k+1}}^2 N^{- [ (m-1) \wedge 2q]/2} \right] + C ||u(T_{k+1},\cdot)||_{m+1,\infty} \Delta_{T_{k+1}}^{(m+1)/2}\notag\\
&&\quad  + C||u(T_{k+1},\cdot)||_{m+2,\infty} \Delta_{T_{k+1}}^{(m+2)/2}  + \Delta_{T_{k+1}}\| \partial_{w} f \|_{\infty} \left|\la \mu-\hmu_{T_{k+1}}, \varphi_f[\cdot,u(T_{k+1},\cdot)]\ra \right|\notag\\
&& \quad  + \Delta_{T_{k+1}}\| \partial_{w} f \|_{\infty} \left| \la \hmu_{T_{k+1}}, \varphi_f[\cdot,u(T_{k+1},\cdot)] - \varphi_f[\cdot,\hu^1(T_{k+1},\cdot)]\ra\right|.\label{Eq:CS}
\end{eqnarray}
Note that from the forward result \eqref{Eq:ErrorforSB} in Theorem \ref{MR}, we have
\begin{equation}\label{Eq:ity1}
\Delta_{T_{k+1}}\| \partial_{w} f \|_{\infty} \left|\la \mu-\hmu_{T_{k+1}}, \varphi_f[\cdot,u(T_{k+1},\cdot)]\ra \right| \leq C(||u(T_{k+1},\cdot)||_{m+2,\infty}) \Delta_{T_{k+1}} N^{-((m-1)\wedge 2q)/2},
\end{equation} 
while using the regularity of $\varphi_f$ and the definition of $\hmu$ gives,
\begin{equation}\label{Eq:ity2}
\Delta_{T_{k+1}}\| \partial_{w} f \|_{\infty}  \left| \la \hmu_{T_{k+1}}, \varphi_f[\cdot,u(T_{k+1},\cdot)] - \varphi_f[\cdot,\hu^1(T_{k+1},\cdot)]\ra\right| \leq C' \Delta_{T_{k+1}}\mE^1_u (k+1) .
\end{equation}
The Claim follows by applying the Cauchy-Schwartz inequality on the first term in the right hand side of \eqref{Eq:CS} and plugging \eqref{Eq:ity1} and \eqref{Eq:ity2} in \eqref{Eq:CS}.
\end{proof}

We can now analyze the local error at step $k$. By plugging the estimates from Claims \ref{Cl:back1-1}, \ref{Cl:back1-2} and \ref{Cl:back1-3} in the expansion \eqref{Eq:Erroru11}, \eqref{Eq:Erroru12} and  \eqref{Eq:Erroru13} of $u-\hu$ we obtain that:
\begin{equation}\label{Eq:propa1}
\mE^1_u(k)  \leq  \left(1+C\Delta_{T_{k+1}}\right)\mE^1_u(k+1) + \bar{\epsilon}(k+1),
\end{equation}
with
\begin{eqnarray}
\bar{\epsilon}(k+1) &= &  C\bigg( \sup_{s\in [T_k;T_{k+1}]}||u(s,\cdot)||_{2,\infty} \Delta_{T_{k+1}}^{4}+ ||u(T_{k+1},\cdot)||_{2,\infty}\Delta_{T_{k+1}} N^{- [ (m-1) \wedge 2q]/2} \notag\\
&&\quad + ||u(T_{k+1},\cdot)||_{m+1,\infty} \Delta_{T_{k+1}}^{(m+1)/2} + ||u(T_{k+1},\cdot)||_{m+2,\infty} \Delta_{T_{k+1}}^{(m+2)/2}\notag\\
&&\quad +  ||u(T_{k+1},\cdot)||_{3,\infty} \left[\Delta_{T_{k+1}}^2 + \Delta_{T_{k+1}}^{q+1}+ \Delta_{T_{k+1}}^2 N^{- [ (m-1) \wedge 2q]/2} \right]\bigg)\label{Eq:barepsilon}\\
&& + C'\bigg( M_u(m+1,T_{k+1}) \Delta_{T_{k+1}}^{(m+1)/2} +M_u(m+2,T_{k+1}) \Delta_{T_{k+1}}^{(m+2)/2}\notag\\
&&\quad + M_u(2,T_{k+1}) \left[ \Delta_{T_{k+1}}^{q+1}+\Delta_{T_{k+1}} N^{- [ (m-1) \wedge 2q]/2} \right]\bigg).\notag
\end{eqnarray}

Under \SB, we have from Lemma \ref{RegNonLinearPDE} that for all $n \in \mathbb{N}^*$ there exists a constant $K$, depending on the regularity of $V_{0:d}$ and $\phi$, such that for all $k \in \{0,\cdots,N-1\}$, 
\begin{equation*}
M_u(n,T_{k+1}) + \sup_{s\in [T_k,T_{k+1}]}||u(s,\cdot)||_{n,\infty}\leq K.
\end{equation*}
Therefore, Gronwall's Lemma applied to \eqref{Eq:propa1}  and the definition of $\Delta_{T_k}$ implies
\begin{equation}
\mE_u^1 \leq CN^{-1}. \label{Ineq:FirstOrderU}
\end{equation}

Moreover, Claims \ref{Cl:back1-1} and \ref{Cl:back1-2} and expansion \eqref{Eq:Errorv11}, \eqref{Eq:Errorv12} and \eqref{Eq:Errorv13} show that:
 \begin{eqnarray*}
&&\Delta_{T_{k+1}} \mE_v^1(k+1) \leq  C \bigg( \sup_{s\in [T_k,T_{k+1}]}||u(s,\cdot)||\Delta_{T_{k+1}}^2 + ||u(T_{k+1},\cdot)||_{3,\infty} \left[ \Delta_{T_{k+1}}^{q+1}+  \Delta_{T_{k+1}}^2 N^{- [ (m-1) \wedge 2q]/2} \right]\\
&&\hphantom{ \Delta_{T_{k+1}} \mE_v^1(k+1) \leq}  + ||u(T_{k+1},\cdot)||_{m+1,\infty} \Delta_{T_{k+1}}^{(m+1)/2} + C||u(T_{k+1},\cdot)||_{m+2,\infty} \Delta_{T_{k+1}}^{(m+2)/2}\bigg) + \Delta_{T_{k+1}}^{1/2}\mE_u^1(k+1),
\end{eqnarray*}
which together with \eqref{Ineq:FirstOrderU} imply 
\[\Delta_{T_k}^{1/2}\mE_v^1 \leq CN^{-1},\]
and the result holds.\\

\emph{(2) Proof of the order of convergence for the second order algorithm (Algorithm \ref{alg:SecondOrder}).}\\

Let $k\in \{0, \ldots, N-2\}$. We can expand as before the errors on the $u$ and $v$ approximations as:

\begin{align}
u(T_k ,\hX^{\pi}_{T_k})  - \hu^{2}(T_k,\hX^{\pi}_{T_k}) & = \ApproxEx[T_k,T_{k+1}]{u(T_{k+1},X^{T_k,\hX^{\pi}_{T_k},\hmu}_{T_{k+1}}) - \hu^2(T_{k+1},X^{T_k,\hX^{\pi}_{T_k},\hmu}_{T_{k+1}})} \notag\\
& \qquad  + \frac{1}{2}\Delta_{T_{k+1}} \ApproxEx[T_k,T_{k+1}]{ f(\Theta_{k+1}^{\mu}(X_{T_{k+1}}^{T_k,\hX^{\pi}_{T_k},\mu}))-  f(\ApproxTwoVarT{\Theta}_{k+1})} \label{Eq:Erroru23} \\
& \quad + \frac{1}{2}\Delta_{ T_{k+1}} [ f(\Theta_{k}^{\mu}(\hX^{\pi}_{T_k})) -  f(\PredictorVarT{\Theta}^{\pi}_{k})] \label{Eq:Erroru24}\\
& \quad + \epsilon_{\hu^2,k}^{s}(\pi)+ \epsilon_{\hu^2,k}^{c}(\pi),\notag
\end{align}
where
\begin{align}
\epsilon_{\hu^2,k}^{s}(\pi)&=  u(T_k,\hX^{\pi}_{T_k}) - \EX{u(T_{k+1},X^{T_k,\hX^{\pi}_{T_k},\mu}_{T_{k+1}}) + \frac{\Delta_{ T_{k+1}}}{2} \left( f(\Theta_{k}^{\mu}(\hX^{\pi}_{T_k}))+ f(\Theta_{k+1}^{\mu}(X_{T_{k+1}}^{T_k,\hX^{\pi}_{T_k},\mu}))\right)} \label{Eq:Erroru21}\\
\epsilon_{\hu^2,k}^{c}(\pi)&= \EX { u(T_{k+1},X^{T_k,\hX^{\pi}_{T_k},\mu}_{T_{k+1}}) } - \ApproxEx[T_k,T_{k+1}] { u(T_{k+1},X^{T_k,\hX^{\pi}_{T_k},\hmu}_{T_{k+1}}) } \notag\\
&\quad + \frac{\Delta_{ T_{k+1}}}{2} \left(\EX {   f(\Theta_{k+1}^{\mu}(X_{T_{k+1}}^{T_k,\hX^{\pi}_{T_k},\mu}))} - \ApproxEx[T_k,T_{k+1}] { f(\Theta_{k+1}^{\mu}(X_{T_{k+1}}^{T_k,\hX^{\pi}_{T_k},\mu}))}  \right). \label{Eq:Erroru22} 
\end{align}

Similarly, we have
\begin{align}
v(T_k,\hX^{\pi}_{T_k})  - \hv^2 (T_k,\hX^{\pi}_{T_k})  & = \E_{\Q_{T_k,T_{k+1}}} \left( \left[ u\left(T_{k+1},X_{T_{k+1}}^{T_k,\hX^{\pi}_{T_k},\hmu}\right) - \hu^2\left(T_{k+1},X_{T_{k+1}}^{T_k,\hX^{\pi}_{T_k},\hmu}\right) \right]  \zeta_{k+1}\right)\notag\\
& \quad\qquad  + \ApproxEx[T_k,T_{k+1}]{\left( f[\Theta_{k+1}^{\mu}(X_{T_{k+1}}^{T_k,\hX^{\pi}_{T_k},\mu})]  - f(\ApproxTwoVarT{\Theta}_{k+1}) \right) \Delta_{T_{k+1}}\zeta_{k+1} }\label{Eq:Errorv23}\\
& \quad + \epsilon_{\hv^2,k}^{s}(\pi)+ \epsilon_{\hv^2,k}^{c}(\pi),\notag
\end{align}
with
\begin{align}
\epsilon_{\hv^2,k}^{s}(\pi)&= v(T_k,\hX^{\pi}_{T_k}) - \E \left( \left[u(T_{k+1},X_{T_{k+1}}^{T_k,\hX^{\pi}_{T_k},\mu})+\Delta_{T_{k+1}}f[\Theta_{k+1}^{\mu}(X_{T_{k+1}}^{T_k,\hX^{\pi}_{T_k},\mu})] \right] \zeta_{k+1} \right) \label{Eq:Errorv21} \\
\epsilon_{\hv^2,k}^{c}(\pi)&=   \EX{u(T_{k+1},X_{T_{k+1}}^{T_k,\hX^{\pi}_{T_k},\mu})\zeta_{k+1}} -\ApproxEx[T_k,T_{k+1}]{u(T_{k+1},X_{T_{k+1}}^{T_k,\hX^{\pi}_{T_k},\hmu})\zeta_{k+1}}  \notag\\
& \quad\qquad  + \EX{f[\Theta_{k+1}^{\mu}(X_{T_{k+1}}^{T_k,\hX^{\pi}_{T_k},\mu})]\Delta_{T_{k+1}}\zeta_{k+1}} -\ApproxEx[T_k,T_{k+1}]{f[\Theta_{k+1}^{\mu}(X_{T_{k+1}}^{T_k,\hX^{\pi}_{T_k},\mu})]\Delta_{T_{k+1}}\zeta_{k+1}}.\label{Eq:Errorv22}
\end{align}

We identify, as for the first order expansion, some error terms corresponding to the scheme error \eqref{Eq:Erroru21} and \eqref{Eq:Errorv21}, generalized cubature errors \eqref{Eq:Erroru22} and \eqref{Eq:Errorv22} and propagation errors \eqref{Eq:Erroru23}, \eqref{Eq:Errorv23}. Some important changes are clear from the expansion: we have in addition a prediction error term \eqref{Eq:Erroru24} reflecting the fact that we perform a new intermediate step, and we have some $f$ term in \eqref{Eq:Errorv23}, adding to the propagation error.

The proof for the second order approximation is then similar to its first order equivalent, but we will have to consider the mentioned additional terms. In particular, the fact that the second order approximation $\hv^2(\hX^\pi_{T_k})$ includes the term $f(\ApproxTwoVarT{\Theta}_{T_{k+1}})$, adds an additional coupling effect. With this in mind, and in order to simplify the analysis, we introduce the following quantity
\[\mE_{f}(k) = \max_{\pi \in \mc{S}(k)} | f(\Theta_k^\mu(\hX^{\pi}_{T_k})) - f(\ApproxTwoVarT{\Theta}_{k})  |\]
and we will analyze the dynamics of the sum of errors at step $k$, $\mc{E}_u^2(k) + \Delta_{T_{k}}\mE_{f}(k) $.

To this aim, we will bound separately the scheme and cubature errors. Each bound is summarized in a Claim (respectively Claims \ref{Cl:back2-1}, \ref{Cl:back2-2} below). Then, we will conclude with a Gronwall argument. 

The scheme error terms \eqref{Eq:Erroru21} and \eqref{Eq:Errorv21} and the generalized cubature errors \eqref{Eq:Erroru22} and \eqref{Eq:Errorv22} are treated similarly as in the first order scheme. We show this in Claims \ref{Cl:back2-1} and \ref{Cl:back2-2}.

\begin{claim}\label{Cl:back2-1}
There exists a constant $C$, depending on the regularity of $V_{0:d}$ and $f$ (and not on $k$) such that:
\begin{eqnarray*}
\left| \epsilon_{\hu^2,k}^{s}(\pi) \right| & \leq C \sup_{s \in [T_k,T_{k+1}]}||u(s,\cdot)||_{6,\infty} \Delta^3_{T_{k+1}}\\
 \left| \epsilon_{\hv^2,k}^{s} (\pi)\right| & \leq C \sup_{s \in [T_k,T_{k+1}]}||u(s,\cdot)||_{5,\infty} \Delta^2_{T_{k+1}}.
\end{eqnarray*}
where $\epsilon_{\hu^2,k}^{s}(\pi), \epsilon_{\hv^2,k}^{s}(\pi)$ are defined in \eqref{Eq:Erroru21} and \eqref{Eq:Errorv21}.
\end{claim}
\begin{proof}
The proof follows in the same way as the one of Claim \ref{Cl:back1-1}, by performing a Taylor expansion to one additional order. The choice of $\zeta_{k+1}$ is the one needed to match the lower order terms (recall that the $i^{\mathrm th}$ component of $\zeta_k$ is expressed as $\zeta_k^i = 4\Delta_{T_k}^{-1} I^{T_k,T_{k+1}}_{(i)} - 6\Delta_{T_k}^{-2} I^{T_k,T_{k+1}}_{(0,i)}$).

Applying Lemma \ref{Theo:ItoTaylorExpansion} with $n=4$ to $u$ and with $n=2$  to $\mc{V}_{(0)}u$ implies, after taking the expectation, that
\begin{equation*}
\EX{u(T_{k+1},  X^{T_k,y,\mu}_{T_{k+1}})}   =   u(T_{k},y) + \Delta_{T_k} \mc{V}_{(0)}u(T_{k}, y) + \frac{1}{2}\mc{V}_{(0,0)}  u(T_{k},y) \Delta _{T_{k+1}}^2 +   \sum_{\beta \in \p\mc{A}_4} \E\left( I_{\beta}^{T_k,T_{k+1}}[\mc{V}_{\beta} u(.,X^{T_k,y,\mu}_{.})] \right)
\end{equation*}
and
\begin{equation*}\label{Eq:V0ofU}
\EX{\mc{V}_{(0)}u(T_{k+1},  X^{T_k,y,\mu}_{T_{k+1}})}  =  \mc{V}_{(0)} u(T_{k},y) + \mc{V}_{(0,0)}  u(T_{k},y)\Delta _{T_{k+1}} + 
\sum_{\beta \in \p\mc{A}_2} \E\left( I_{\beta}^{T_k,T_{k+1}}[ \mc{V}_{(\beta*0)} u(.,X^{T_k,y,\mu}_{.})] \right).
\end{equation*}
Then, the estimate of Lemma \ref{lemmathatgivestheorderoftheiterateditointegral} gives
\begin{align*}
&\EX{u(T_{k+1},  X^{T_k,y,\mu}_{T_{k+1}}) -\frac{\Delta_{T_k}}{2} \left( \mc{V}_{(0)}u(T_k, y) + \mc{V}_{(0)}u(T_{k+1},  X^{T_k,y,\mu}_{T_{k+1}})\right)}  -  u(T_{k},y) \\
&\quad \quad = \sum_{\beta \in \p\mc{A}_4} \E\left( I_{\beta}^{T_k,T_{k+1}}[ \mc{V}_{(\beta)} u(.,X^{T_k,y,\mu}_{.})] \right) - \frac{\Delta_{T_k}}{2}\sum_{\beta \in \p\mc{A}_2} \E\left( I_{\beta}^{T_k,T_{k+1}}[ \mc{V}_{(\beta*0)} u(.,X^{T_k,y,\mu}_{.})] \right)\\
& \quad \quad \leq C(T,V_{0:d},f)  \Delta_{T_{k+1}}^{3} \sup_{ s\in[T_k,T_{k+1}]} || u(s,.) ||_{6,\infty} ,
\end{align*}
from where we deduce the first inequality. 

Similarly, by using Lemma \ref{Theo:ItoTaylorExpansion} with $n=3$ on $u$ and with $n=1$ on $\mc{V}_{(0)}u$ and taking the expectation, using the fact that $\zeta_k^i = 4\Delta_{T_k}^{-1} I^{T_k,T_{k+1}}_{(i)} - 6\Delta_{T_k}^{-2} I^{T_k,T_{k+1}}_{(0,i)}$, it follows
\begin{equation*}
\begin{split}
& \EX{\left[  u(T_{k+1},  X^{T_k,y,\mu}_{T_{k+1}})   - \Delta_{ T_{k+1}} \mc{V}_{(0)} u(T_{k+1},  X^{T_k,y,\mu}_{T_{k+1}}) \right] \Delta_{T_{k+1}}\zeta_k^j }\\
& =\EE \Bigg[\big[ \mc{V}_{(j)} u(T_k,y) I_{(j)}^{T_k,T_{k+1}} + \mc{V}_{(0,j)} u(T_k,y) I_{(0,j)}^{T_k,T_{k+1}} + \mc{V}_{(j,0)} u(T_k,y)I_{(j,0)}^{T_k,T_{k+1}} \\
& \qquad\quad  - \Delta_{T_{k+1}}\mc{V}_{(j,0)} u(T_k,y) I_{(j)}^{T_k,T_{k+1}} \big]\left(4 I_{(j)}^{T_k,T_{k+1}}-6 \frac{I_{(0,j)}^{T_k,T_{k+1}}}{\Delta_{T_{k+1}}}\right)\Bigg] + \mc{R}(k,j)\\
& = \Delta_{T_{k+1}} \mc{V}_{(j)}u(T_{k},y)  + \mc{R}(k,j)\\
\end{split}
\end{equation*}
where
\begin{equation*}
\begin{split}
\mc{R}(k,j) & =  \Delta_{T_{k+1}}  \sum_{\beta \in \p\mc{A}_3} \E\left( I_{\beta}^{T_k,T_{k+1}}[ \mc{V}_{(\beta)} u(.,X^{T_k,y,\mu}_{.})]\zeta^j_k \right) \\
&\quad - \Delta_{T_{k+1}}^2 \sum_{\beta \in \p\mc{A}_1} \E\left( I_{\beta}^{T_k,T_{k+1}}[ \mc{V}_{(\beta*0)} u(.,X^{T_k,y,\mu}_{.})]\zeta^j_k \right).
\end{split}
\end{equation*}
Using Lemma \ref{lemmathatgivestheorderoftheiterateditointegral} we bound the residual term $\mc{R}(k,j)$ and obtain
\begin{align*}
& \Delta_{T_{k+1}}  \sum_{\beta \in \p\mc{A}_3} \E\left( I_{\beta}^{T_k,T_{k+1}}[ \mc{V}_{(\beta)} u(.,X^{T_k,y,\mu}_{.})]\zeta^j_k \right) - \Delta_{T_{k+1}}^2 \sum_{\beta \in \p\mc{A}_1} \E\left( I_{\beta}^{T_k,T_{k+1}}[ \mc{V}_{(\beta*0)} u(.,X^{T_k,y,\mu}_{.})]\zeta^j_k \right) \\
& \quad \leq   \Delta_{T_k}^{3}\sup_{s\in[T_k,T_{k+1}]} || u(s,.) ||_{5,\infty} + \Delta_{T_k}^{7/2}\sup_{ s\in[T_k,T_{k+1}]} ||u(s,.) ||_{5,\infty},
\end{align*}
recalling that $ v(T_{k},y)= \mc{V}_{(j)}u(T_{k},y) $  we deduce the second inequality.
\end{proof}

\begin{claim}\label{Cl:back2-2}
There exist two constants $C$, depending on $d$, $q$, $T$, $m$, the regularity of $V_{0:d}$ and $\varphi_{0:d}$ (and not on $k$), and $C'$, depending in addition on the regularity of $f$, such that:
\begin{align*}
 \left| \epsilon_{\hu^2,k}^{c}(\pi) \right| & \leq C \bigg( ||u(T_{k+1},\cdot)||_{2,\infty} \left[ \Delta_{T_{k+1}}^{q+1}+ \Delta_{T_{k+1}} N^{- [ (m-1) \wedge 2q]/2} \right] + ||u(T_{k+1},\cdot)||_{m+1,\infty} \Delta_{T_{k+1}}^{(m+1)/2}\\
&\qquad \quad + ||u(T_{k+1},\cdot)||_{m+2,\infty} \Delta_{T_{k+1}}^{(m+2)/2}\bigg)\\
&\quad  + C' \bigg( M_u(3,T_{k+1})\left[ \Delta_{T_{k+1}}^{q+2}+ \Delta_{T_{k+1}}^2 N^{- [ (m-1) \wedge 2q]/2} \right] + M_u(m+2,T_{k+1}) \Delta_{T_{k+1}}^{(m+3)/2} \\
&\qquad \quad + M_u(m+3,T_{k+1}) \Delta_{T_{k+1}}^{(m+4)/2}\bigg), \\
\left|\epsilon_{\hv^2,k}^{c} (\pi)\right| & \leq C \bigg( ||u(T_{k+1},\cdot)||_{3,\infty} \left[ \Delta_{T_{k+1}}^{q}+ \Delta_{T_{k+1}} N^{- [ (m-1) \wedge 2q]/2} \right] + \sum_{i=m-2}^{m+1} ||u(T_{k+1},\cdot)||_{i,\infty} \Delta_{T_{k+1}}^{(i-1)/2}\bigg)\\
& \qquad + C' \bigg(  M_u(4,T_{k+1}) \left[ \Delta_{T_{k+1}}^{q+1}+ \Delta_{T_{k+1}}^2 N^{- [ (m-1) \wedge 2q]/2} \right] + \sum_{i=m-1}^{m+2} M_u(i,T_{k+1}) \Delta_{T_{k+1}}^{i/2}\bigg).
\end{align*}
where $\epsilon_{\hu^2,k}^{c}(\pi), \epsilon_{\hv^2,k}^{c}(\pi)$ are defined in \eqref{Eq:Erroru22} and \eqref{Eq:Errorv22}.
\end{claim}

\begin{remarque}
Although the rates of convergence have a leading term of order $\Delta_{T_{k+1}}^{(m-1)}$ that is worst than the one in the first order scheme result, (Claim \ref{Cl:back1-2} ), here we assume that $m$ is bigger, and thus they are suitable for a second order scheme.
\end{remarque}

\begin{proof}
Note first that the $r^{{\rm th}}$ derivative of the function $y\mapsto f(\cdot,y,u(\cdot,y),\mc{V}u(\cdot,y),\cdot)$ is bounded by $M_u(r+1,\cdot)$ defined by \eqref{Eq:DefinitionM}. This estimate goes up to $r+1$ and not just $r$ as in Claim \ref{Cl:back1-2}, because the differentiation of $y \in \R^d \mapsto f(\Theta_{k+1}^{\mu}(y))$ involves the additional dependence on $\mc{V}u$.

Then, the first assertion follows from applying \eqref{Eq:Lem:OneStepAll1} in Lemma \ref{Lem:OneStepAll} to $u$ and $f(\Theta_{k+1}^{\mu})$. For the second assertion, recall that the $i^{\mathrm th}$ component of $\zeta_k$ is expressed as $\zeta_k^i = 4\Delta_{T_k}^{-1} I^{T_k,T_{k+1}}_{(i)} - 6\Delta_{T_k}^{-2} I^{T_k,T_{k+1}}_{(0,i)} $. Then, applying \eqref{Eq:Lem:OneStepAll2} and \eqref{Eq:Lem:OneStepAll3} with $n=m$ in Lemma \ref{Lem:OneStepAll} to $u$ and $f(\Theta_{k+1}^{\mu})$, we conclude on the second assertion. 
\end{proof}

It will be handy to have an expansion on the prediction error, by recalling that $\tilde{u}$ is essentially an application of the first order scheme, it follows that
\begin{align} 
 (u- \tu)(T_{k},\hX^{\pi}_{T_k}) \ =\ &  \E_{\Q_{T_k,T_{k+1}}} \bigg[ (u - \hu^2) (T_{k+1},X^{T_k,\hX^{\pi}_{T_k},\hmu}_{T_{k+1}}) + \Delta_{T_{k+1}} \left( f(\Theta^{\hmu}_{k+1}(X^{T_k,\hX^{\pi}_{T_k},\hmu}_{T_{k+1}}) -  f(\ApproxTwoVarT{\Theta}_{k+1})\right) \bigg] \label{Eq:utilde_pro_error}\\
&  \quad + \epsilon^{s}_{\tilde{u}, k}(\pi)+\epsilon^{c}_{\tilde{u}, k}(\pi)  \notag
\end{align}
where
\begin{align}
&  \epsilon^{c}_{\tilde{u}, k}(\pi)  := \E \left[ u(T_{k+1},X^{T_k,\hX^{\pi}_{T_k},\mu}_{T_{k+1}}) + \Delta_{T_{k+1}} f(\Theta^{\mu}_{k+1}(X^{T_k,\hX^{\pi}_{T_k},\hmu}_{T_{k+1}}) \right] \notag\\
&  \quad - \E_{\Q_{T_k,T_{k+1}}}\left[u(T_{k+1},X^{T_k,\hX^{\pi}_{T_k},\hmu}_{T_{k+1}})+ \Delta_{T_{k+1}} f(\Theta^{\hmu}_{k+1}(X^{T_k,\hX^{\pi}_{T_k},\hmu}_{T_{k+1}}) \right]\label{Eq:utilde_cub_error}\\
&  \epsilon^{s}_{\tilde{u}, k}(\pi) := u(T_k,\hX^{\pi}_{T_k}) - \EX{u(T_{k+1},X^{T_k,\hX^{\pi}_{T_k},\mu}_{T_{k+1}}) + \Delta_{T_{k+1}} f[\Theta^{\mu}_{k+1}(X^{T_k,\hX^{\pi}_{T_k},\hmu}_{T_{k+1}})]}\label{Eq:utilde_scheme_error}.
\end{align}
Note that $\epsilon^{s}_{\tilde{u}, k}(\pi) $ and $\epsilon^{c}_{\tilde{u}, k}(\pi) $ may be bounded respectively as in Claims \eqref{Cl:back1-1} and \eqref{Cl:back1-2} . The fact that we are using here $\Theta^{\mu}_{k+1}$ instead of $\bar{\Theta}^{\mu,\mu}_{k+1,k}$ in those claims, is not really problematic. For the scheme error in Claim \eqref{Cl:back1-1}, the difference can be controlled by an additional application of Ito's theorem, but we skip the details. For the cubature error in Claim \eqref{Cl:back1-2} this difference plays no role at all.

Let us now focus on the errors when approaching the driver. Using the mean value theorem, we know that there exist $\Psi_1, \Psi_2, \Psi_3$ respectively bounded by $\| \partial_{y'}  f \|_{\infty}$, $\| \partial_{z}  f \|_{\infty}$  and $\| \partial_{w}  f \|_{\infty} $, such that
\begin{align} 
 f(\Theta_{k}^{\mu}(\hX^{\pi}_{T_k})) -   f(\ApproxTwoVarT{\Theta}_{k})\ =\ &  \Psi_{1} \big[ u(T_k,\hX_{T_k}^\pi)- \hu^2(T_{k},\hX^{\pi}_{T_k}) \big] +  \Psi_{2} \big[ v(T_k,\hX_{T_k}^\pi) - \hv^2(T_{k},\hX^{\pi}_{T_k}) \big] \label{Eq:ExpF}\\
 & + \Psi_3 \left(\la \hmu_{T_{k}}, \varphi_f[\cdot,u(T_{k},\cdot)]\ra - \la \hmu_{T_{k}}, \varphi_f[\cdot,\hu^2(T_{k},\cdot),\cdot)]\ra\right)\label{Eq:ExpF2} \\
 & + \Psi_3 \left(\la \mu_{T_k}, \varphi_f[\cdot,u(T_{k},\cdot)]\ra - \la \hmu_{T_{k}}, \varphi_f[\cdot,u(T_{k},\cdot),\cdot)]\ra\right),\label{Eq:ExpF3}
\end{align}
Similarly, there exist random variables  and $\Psi_1^{\prime}, \Psi_2^{\prime}, \Psi_3^{\prime}$ respectively bounded by $\| \partial_{y'}  f \|_{\infty}$, $\| \partial_{z}  f \|_{\infty}$  and $\| \partial_{w}  f \|_{\infty} $, such that
\begin{align} 
 f(\Theta_{k}^{\mu}(\hX^{\pi}_{T_k})) -   f(\PredictorVarT{\Theta}^{\pi}_{k})\  =\ &  \Psi^{\prime}_{1} \big[ u(T_k,\hX_{T_k}^\pi)- \tu(T_{k},\hX^{\pi}_{T_k}) \big] +  \Psi^{\prime}_{2} \big[ v(T_k,\hX_{T_k}^\pi) - \hv^2(T_{k},\hX^{\pi}_{T_k}) \big] \label{Eq:ExpFpred}\\
 & + \Psi_3^{\prime} \left(\la \hmu_{T_{k}}, \varphi_f[\cdot,u(T_{k},\cdot)]\ra - \la \hmu_{T_{k}}, \varphi_f[\cdot,\tu(T_{k},\cdot),\cdot)]\ra\right)\label{Eq:ExpFpred2} \\
 & + \Psi_3^{\prime} \left(\la \mu_{T_k}, \varphi_f[\cdot,u(T_{k},\cdot)]\ra - \la \hmu_{T_{k}}, \varphi_f[\cdot,u(T_{k},\cdot),\cdot)]\ra\right) .\label{Eq:ExpFpred3}.
\end{align}

Then, using the error development for $(u- \hu^2)(T_{k},\hX^{\pi}_{T_k})$ given in \eqref{Eq:Erroru21},\eqref{Eq:Erroru22}, \eqref{Eq:Erroru23},\eqref{Eq:Erroru24} the error and the ones for $f(\Theta_{k}^{\mu}(\hX^{\pi}_{T_k})) -   f(\ApproxTwoVarT{\Theta}_{k})$ in \eqref{Eq:ExpF},\eqref{Eq:ExpF2} and \eqref{Eq:ExpF3} and $ f(\Theta_{k}^{\mu}(\hX^{\pi}_{T_k})) -   f(\PredictorVarT{\Theta}^{\pi}_{k})$ in \eqref{Eq:ExpFpred}, \eqref{Eq:ExpFpred2} and \eqref{Eq:ExpFpred3} ,  it follows

\begin{align*}
(u- \hu^{2})(T_k,\hX^{\pi}_{T_k}) & + \frac{1}{2} \Delta_{T_{k}} \left( f(\Theta_{k}^{\mu}(\hX^{\pi}_{T_k})) -   f(\ApproxTwoVarT{\Theta}_{k})  \right) \notag \\
& =  (1+ \Delta_{T_{k}}\Psi_{1}) \E_{\Q_{T_k,T_{k+1}}} \bigg[ u(T_{k+1},X^{T_k,\hX^{\pi}_{T_k},\hmu}_{T_{k+1}}) - \hu^2(T_{k+1},X^{T_k,\hX^{\pi}_{T_k},\hmu}_{T_{k+1}}) \notag\\
& \qquad\qquad\qquad\qquad\qquad  + \frac{1}{2}\Delta_{T_{k+1}} \left( f(\Theta_{k+1}^{\mu}(X_{T_{k+1}}^{T_k,\hX^{\pi}_{T_k},\mu}))-  f(\ApproxTwoVarT{\Theta}_{k+1}) \right) \bigg] \\
& \quad + (1+ \Delta_{T_{k}}\Psi_{1}) (\epsilon_{\hu^2,k}^s(\pi) + \epsilon_{\hu^2,k}^c)(\pi)\\
& \quad +   \frac{1}{2}[\Psi_{2}\Delta_{T_{k}} +\Psi_{2}^{\prime}\Delta_{T_{k+1}}] \big[ v(T_k,\hX_{T_k}^\pi) - \hv^2(T_{k},\hX^{\pi}_{T_k}) \big] \notag\\
& \quad + \Delta_{T_{k+1}}\frac{\Psi^{\prime}_{1}}{2} \big[ u(T_k,\hX_{T_k}^\pi)- \tu(T_{k},\hX^{\pi}_{T_k}) \big] \notag\\
& \quad + \frac{1}{2}[\Psi_{3}\Delta_{T_{k}} +\Psi_{3}^{\prime}\Delta_{T_{k+1}}] \left(\la \mu_{T_k}, \varphi_f[\cdot,u(T_{k},\cdot)]\ra - \la \hmu_{T_{k}}, \varphi_f[\cdot,u(T_{k},\cdot),\cdot)]\ra\right),\notag\\ 
 & \quad  + \Delta_{T_{k}} \frac{\Psi_{3}}{2} \left(\la \hmu_{T_{k}}, \varphi_f[\cdot,u(T_{k},\cdot)]\ra - \la \hmu_{T_{k}}, \varphi_f[\cdot,\hu^2(T_{k},\cdot),\cdot)]\ra\right)\notag\\
 & \quad + \Delta_{T_{k+1}}\frac{\Psi^{\prime}_{3}}{2} \left(\la \hmu_{T_{k}}, \varphi_f[\cdot,u(T_{k},\cdot)]\ra - \la \hmu_{T_{k}}, \varphi_f[\cdot,\tu(T_{k},\cdot),\cdot)]\ra\right) \notag
\end{align*}
Then, replacing the expansion for the error of $\hv^2$ in terms of \eqref{Eq:Errorv21},\eqref{Eq:Errorv22}, \eqref{Eq:Errorv23} and  the one for $(u-\tilde{u})$ in terms of \eqref{Eq:utilde_pro_error}, \eqref{Eq:utilde_cub_error}, \eqref{Eq:utilde_scheme_error},  and up to rescaling some of the $\Psi$ random variables, one obtains

\begin{align}
(u- \hu^{2})& (T_k,\hX^{\pi}_{T_k})  + \frac{1}{2} \Delta_{T_{k}} \left( f(\Theta_{k}^{\mu}(\hX^{\pi}_{T_k})) -   f(\ApproxTwoVarT{\Theta}_{k})  \right) \label{Eq:FinalExpUhat2} \\
& =   \E_{\Q_{T_k,T_{k+1}}} \Bigg((1+ \Delta_{T_{k+1}}\Psi_{1}^{\prime\prime} + \Delta_{T_{k}} \zeta_{k+1} \Psi_{2}^{\prime\prime}) \bigg[ (u - \hu^2)(T_{k+1},X^{T_k,\hX^{\pi}_{T_k},\hmu}_{T_{k+1}}) \notag\\
& \qquad\qquad\qquad\qquad\qquad  + \frac{1}{2}\Delta_{T_{k+1}} \left( f(\Theta_{k+1}^{\mu}(X_{T_{k+1}}^{T_k,\hX^{\pi}_{T_k},\mu}))-  f(\ApproxTwoVarT{\Theta}_{k+1}) \right) \bigg] \Bigg) \notag\\
& \quad -  \E_{\Q_{T_k,T_{k+1}}}\Bigg( \bigg[ (u - \hu^2)(T_{k+1},X^{T_k,\hX^{\pi}_{T_k},\hmu}_{T_{k+1}}) \bigg](\Delta_{T_{k+1}}\Psi_{1}^{\prime\prime} + \Delta_{T_{k+1}}\zeta_{k+1}\Psi_{2}^{\prime\prime}) \Bigg)\notag\\
& \quad + (1+ \Delta_{T_{k}}\Psi_{1}) (\epsilon_{\hu^2,k}^s(\pi) + \epsilon_{\hu^2,k}^c(\pi)) + \Delta_{T_{k+1}}\Psi^{\prime\prime}_{3} (\epsilon^{s}_{\tilde{u}, k}(\pi)+\epsilon^{c}_{\tilde{u}, k}(\pi)) \notag\\
& \quad +  \Delta_{T_{k+1}}\Psi^{\prime\prime}_{4} (\epsilon^{s}_{\hv^2, k}(\pi)+\epsilon^{c}_{\hv^2, k}(\pi))\notag\\
& \quad + \frac{1}{2}[\Psi_{3}\Delta_{T_{k}} +\Psi_{3}^{\prime}\Delta_{T_{k+1}}] \left(\la \mu_{T_k}, \varphi_f[\cdot,u(T_{k},\cdot)]\ra - \la \hmu_{T_{k}}, \varphi_f[\cdot,u(T_{k},\cdot),\cdot)]\ra\right),\notag\\ 
 & \quad  + \Delta_{T_{k}} \frac{\Psi_{3}}{2} \left(\la \hmu_{T_{k}}, \varphi_f[\cdot,u(T_{k},\cdot)]\ra - \la \hmu_{T_{k}}, \varphi_f[\cdot,\hu^2(T_{k},\cdot),\cdot)]\ra\right)\notag\\
 & \quad + \Delta_{T_{k+1}}\frac{\Psi^{\prime}_{3}}{2} \left(\la \hmu_{T_{k}}, \varphi_f[\cdot,u(T_{k},\cdot)]\ra - \la \hmu_{T_{k}}, \varphi_f[\cdot,\tu(T_{k},\cdot),\cdot)]\ra\right). \notag
\end{align}

We are now in position to give the dynamics of the sum of the maximal errors $\mc{E}_u^2(k) + \Delta_{T_{k}}\mE_{f}(k)$. We use Cauchy-Schwartz inequality on the first two terms of \eqref{Eq:FinalExpUhat2}, bound the last two terms in \eqref{Eq:FinalExpUhat2} using the Lipschitz property of $\varphi_f$, as in the first order error analysis, we can deduce  that for some constants $C, C'$,
\begin{equation}
\mE_u^2(k) + \Delta_{T_{k+1}} \mE_f(k)  \leq  (1+C\Delta_{T_{k+1}}) \left[\mc{E}^2_u(k+1) + \Delta_{T_k}\mc{E}_f (k+1)\right] +C' ( \bar{\epsilon}_2(k) + \Delta_{T_{k+1}}N^{-[(m-1)\wedge q]/2}) \label{Eq:propa2}
\end{equation}
with 
\begin{equation}
\label{Eq:DefEpsBar2}
\bar{\epsilon}_2(k) = \sup_{\pi\in \mc{S}_k}|\epsilon_{\hu^2,k}^s(\pi) + \Delta_{T_{k+1}}\epsilon_{\hv^2,k}^s(\pi)| + \sup_{\pi \in \mc{S}_k}| \epsilon_{\hu^2,k}^c(\pi)+ \Delta_{T_{k+1}}\epsilon_{\hv^2,k}^c(\pi)| + \Delta_{T_{k+1}}(\sup_{\pi \in \mc{S}_k}|\epsilon_{\tu,k}^s(\pi) + \epsilon_{\tu,k}^c(\pi)| ).
\end{equation}
We can bound the two first terms of $\bar{\epsilon}_2(k)$ using respectively Claims \ref{Cl:back2-1}, \ref{Cl:back2-2}, while the last term may be bounded using Claims \ref{Cl:back1-1} and \ref{Cl:back1-2} as we have discussed before, so that there is a constant $C''$

\begin{align}
\bar{\epsilon}_2(k) \leq & C'' \bigg( \sup_{s\in [T_k,T_{k+1}]}||u(s,\cdot)||_{4,\infty} \Delta^3_{T_{k+1}}  \label{Inq:eps2k}\\
 &\quad +  [||u(T_{k+1},\cdot)||_{3,\infty} + M_u(4,T_{k+1}) ]\left[ \Delta_{T_{k+1}}^{q+1}+ \Delta_{T_{k+1}}^2 N^{- [ (m-1) \wedge 2q]/2} \right] \notag\\
 &\quad + \sum_{i=m-2}^{m+1} [||u(T_{k+1},\cdot)||_{i,\infty} + M_u(i,T_{k+1})  )]\Delta_{T_{k+1}}^{(i+1)/2} + \sum_{i=m+2}^{m+3} M_u(i,T_{k+1}) \Delta_{T_{k+1}}^{(i+1)/2}\bigg).\notag
\end{align}

Given that the initialization step is the first order scheme, Claims \ref{Cl:back1-1}, \ref{Cl:back1-2} and \ref{Cl:back1-3} imply $\mE_u^2(N-1)\leq KN^{-2}$ and $\Delta_{T_{N-1}}\mE_f^2(N-1)\leq KN^{-2}$. 
Moreover, under \SB, we have from Lemma \ref{RegNonLinearPDE} that for all $n \in \mathbb{N}^*$ there exists a constant $K$, depending on the regularity of $V_{0:d}$ and $\phi$, such that for all $k \in \{0,\cdots,N-1\}$, 
\begin{equation*}
M_u(n,T_{k+1}) + \sup_{s\in [T_k,T_{k+1}]}||u(s,\cdot)||_{n,\infty}\leq K.
\end{equation*}
An application of the discrete Gronwall lemma on the sum $ \mE_u^2(k) + \Delta_{T_{k+1}} \mE_f(k)$, gives
\[\sup_{k\leq N-1} \mE_u^2(k) + \Delta_{T_{k}} \mE_f(k) \leq C\left( \mE_u^2(N-1) + \Delta_{T_{N}} \mE_f(N-1) +N^{-[(m-1)\wedge q]/2} +\sum_{i=0}^{N-2} \bar{\epsilon}_2(k) \right). \]
Using  \eqref{Inq:eps2k} we deduce that, if  $m\geq 7$,where the bound is a consequence of the second assertion of Claim \ref{Cl:back2-2},  we have
\[\sup_{k\leq N-1} \mE_u^2(k) + \Delta_{T_{k}} \mE_f(k) \leq CN^{-2}. \]
As for the first order case, the previous result together with the expansion of $v-\hv^2$ given  \eqref{Eq:Errorv21}, \eqref{Eq:Errorv22}, \eqref{Eq:Errorv23} and Cauchy-Schwartz inequality, implies
\[ \Delta_{T_{k}}^{1/2} \mE_v^2(k) \leq CN^{-2}.\]
This concludes the proof of assertions \eqref{Eq:rate1} and \eqref{Eq:rate3} in Theorem \ref{MR}.
\qed

\section{Mathematical tools}\label{Sec:MathTools}

Here we will intensively use the notions defined in section \ref{Sec:Preliminaire}.

\subsection{The conditional linear PDE}\label{ProofPDE}
\begin{lemme}\label{PropPDE}
Let $\eta$ be a given family of probability measures on $\R^d$ and consider the PDE
\begin{equation}\label{PDE}
\left\lbrace
\begin{array}{ll}
\p_t \psi(t,y) + \mL^\eta \psi(t,y)=0,\text{ on } [0,T]\times \R^d\\
\psi(T,y)=\phi(y)
\end{array}
\right.
\end{equation}
where $\mL^{\eta}$ is defined by \eqref{Eq:CondimL}. Suppose that assumption \textbf{(SB)} holds, then, this PDE admits a unique infinitely differentiable solution $\psi$ and for every multi-index $\beta\in \mc{M}$ there exists a positive constant $C$ depending on the regularity of $V_{0:d}$, $\varphi_{0:d}$ and $T$ such that, :
\begin{equation}\label{Eq:Bound_derivatives_linear_u}
||D_{\beta} \psi(t,.)||_\infty  \leq C ||\phi||_{||\beta||,\infty}
\end{equation}
\end{lemme}

\begin{proof}
Since the law that appears in the coefficients of the SDE is fixed, this is an obvious consequence of the regularity of the coefficients and the terminal condition.\\
\end{proof}

\subsection{The conditional semi-linear PDE}

\begin{lemme}\label{RegNonLinearPDE} Under \textbf{(SB)} there exists a function $u$ from $[0,T]\times \R^d$ to $\R$ such that
$$Y_t^{y} = u(t,X_t^y)$$ 
where $Y_t^{y}$ is defined in \eqref{NLFBSDE}. This function is in $C^{1,2}([0,T] \times \R^d,\R)$ and is the unique solution of the semi-linear PDE:
\begin{equation}\label{Eq:AssociatedPDE}
\left\lbrace
\begin{array}{ll}
\p_t u(t,y)  + \mL^\mu u(t,y) =f\left(t,y,u(t,y),(\mc{V}^\mu u(t,y))^T,\la \mu_t,\varphi_f(\cdot,u(t,\cdot)\ra \right),\text{ on } [0,T]\times \R^d\\
 u(T,y)=\phi(y)
\end{array}
\right.
\end{equation}
where $\mathcal{L}^\mu$ is defined as in \eqref{defmL}. 

Moreover, $u$ is infinitely differentiable, and for every multi-index $\beta\in \mc{M}$ there exists a positive constant $C$ depending on the regularity of $V_{0:d},f,\varphi_{0:d}$, $\phi$ and $T$ such that,
\begin{equation}\label{Eq:Bound_derivatives_u}
||D_{\beta} u||_\infty  \leq C
\end{equation}
\end{lemme}

\textbf{Proof of Lemma \ref{RegNonLinearPDE}.}\\

(i) \emph{Existence and PDE solution.} Consider the conditional BSDE:
\begin{equation}\label{cFBSDEinter}
\left\lbrace\begin{array}{ll}
dX_s^{t,y,x}=\sum_{i=0}^dV_i(s,X_s^{t,y,x}, \E\left[  \varphi_i(X_s^x)\right]) dB^i_s\\
d\bY_{s}^{t,y,x}  = -f(s,X_s^{t,y,x},\bY_s^{t,y,x},\bZ_s^{t,y,x},\E\left[ \varphi_f(X_s^x,Y_s^x)\right])ds + \bZ_s^{t,y,x} dB_s^{1:d}\\
X_t^{t,y,x} = y,\quad \bY_{T}^{t,y,x}  = \phi(X^{t,y,x}_T),
\end{array}
\right.
\end{equation}
for $s$ in $[t,T]$. Note that the McKean term in \eqref{cFBSDEinter}, $\E \varphi_f(X_s^x,Y_s^x)$, does not depend on $y$. In fact, if the solution of \eqref{NLFBSDE} is found, one can consider the term $\E \varphi_f(X_s^x,Y_s^x)$ simply as a term depending on time, so that conditionally to knowing the joint law of $(X_t^x,Y_t^x,\ 0\leq t \leq T)$, equation \eqref{cFBSDEinter} is classical and Markov. As pointed out before, the existence of a unique solution to \eqref{NLFBSDE} follows the lines of the results in \cite{buckdahn_mean-field_2009}.\\

It is clear from the previous discussion that we might apply classical results on BSDE to analyze equation \eqref{cFBSDEinter}. In particular, we have from the results of Pardoux and Peng in \cite{pardoux_backward_1992} given the regularity in space of $f$ and $\phi$, that the mapping $(t,y)\mapsto \bu (t,y) = \bY^{t,y,x}_t$ the solution of \eqref{cFBSDEinter} is differentiable, once in time and twice in space with bounded derivatives and satisfies the PDE:
\begin{equation*}
\left\lbrace
\begin{array}{ll}
\p_t \bu(t,y)  + \mL^\mu \bu(t,y) =f\left(t,y,\bu(t,y), (\mc{V}^{\mu} \bu(t,y))^T,\E[\varphi_f(X^x_t,Y^x_t)]\right)\\
 \bu(T,y)=\phi(y)
\end{array}
\right.,
\end{equation*}

Finally, one can show that $\bu$ solves \eqref{Eq:AssociatedPDE}. To see this, notice that due to the uniqueness of the solutions to \eqref{NLFBSDE} and \eqref{cFBSDEinter},
\[\bu(t,X_t^x) = \bu(t,X_t^{0,x,x})= \bY_t^{t,X_t^{0,x},x} = \bY_{t}^{0,x,x} = Y_t^x.\]
This equality implies \[ \E[\varphi_f(X_s^x,\bu(s,X_s^x))] = \E[\varphi_f(X^{x}_s,Y^{x}_s)] \] for all $s$ in $[0,T]$. Therefore we set $u:=\bu$. This concludes the proof of the first assertion. \\

(ii) \emph{Control on the derivatives.} 

To prove the regularity of $u$ and the the bound on its derivatives, we consider first the case involving only space derivatives. In this case the whole  argument of  Pardoux and Peng may be iterated reasoning on the BSDE for the first derivative, as long as the hypotheses remain valid, to obtain a BSDE for higher order derivatives in space. We turn the reader to the paper of Crisan and Delarue \cite{crisan_sharp_2012} where this is done in detail (taking into account the additional law dependence that must be considered in our framework).

It remains to consider the case of general derivatives including time derivatives. As we have said before, iterative applications of the Pardoux and Peng argument lead to PDEs similar to \eqref{Eq:AssociatedPDE}. Then, we can argue that we are able to differentiate once in time for every two derivatives in space. It is also clear that the control on the space derivatives plus the regularity properties of the coefficients imply the control for time derivatives.


\subsection{One-step errors}

Let $\eta$ be a given family of probability measures and $X^{t,y,\eta}$, $\tX^{t,y,\eta}$ defined as in \eqref{Eq:ArtifDyna1} and \eqref{Eq:ArtifDyna2}. We recall that $\hmu$ denotes the discrete probability measure defined by Algorithm \ref{algo:KF} and $\mu$ the law of the forward part of \eqref{NLFBSDE}.

\begin{lemme}\label{OnestepEuler}
Let  $g$ be a $C^{2}_b$ function  from $\R^d$ to $\R$. Then, there exists a constant $C$ depending only on $V_{0:d}$ and $T$ such that for all $k=1,\cdots,N-1$:
\begin{eqnarray}
\left|(P_{T_k,T_{T_{k+1}}}^{\eta}-\tilde{P}_{T_k,T_{T_{k+1}}}^{\hmu}) g(y)\right| &\leq& C||g||_{2, \infty}\sum_{i=0}^{d} \int_{T_k}^{T_{k+1}} \left|\la \eta_t,\varphi_i\ra- \sum_{p=0}^{q-1} [(t-T_k)^p/p!]\la\eta_{T_k},(\mL^{\eta})^p\varphi_i\ra\right| dt \notag\\
&& \quad + C||g||_{2,\infty}\sum_{i=0}^{d}\sum_{p=0}^{q-1} \Delta_{T_{k+1}}^{p+1}\left|\la \eta_{T_k}, (\mL^{\eta})^p\varphi_i\ra - \la \hmu_{T_k}, (\mL^{\hmu})^{p}\varphi_i\ra \right|.\label{OnestepEuler1}
\end{eqnarray}
Moreover, if $g$ is $C^{3}_b$, there exists a constant $C$ depending only on $V_{0:d},d$ such that for all $l=1,\cdots, d$ and $k=1,\cdots,N-1$:
\begin{eqnarray}\label{OnestepEuler2}
&&\left|\E \left( \left[g\left(X_{T_{k+1}}^{T_k,y,\eta}\right)- g\left(\tX_{T_{k+1}}^{T_k,y,\hmu}\right)\right] I^{T_k,T_{k+1}}_{(l)} \right) \right|\\
&& \quad \leq C||g||_{3,\infty}\sum_{i=0}^{d} \int_{T_k}^{T_{k+1}} \int_{T_k}^{t} \left|\la \eta_s,\varphi_i\ra-\sum_{p=0}^{q-1} [(t-T_k)^p/p!]\la \eta_{T_k},(\mL^\eta)^p\varphi_i\ra\right| ds dt \notag \\
&& \quad + C||g||_{3,\infty}\sum_{i=0}^{d}\sum_{p=0}^{q-1} \Delta_{T_{k+1}}^{p+1}\left|\la \eta_{T_k},(\mL^{\eta})^p\varphi_i\ra - \la \hmu_{T_k}, (\mL^{\hmu})^{p}\varphi_i\ra \right|\notag
\end{eqnarray}
and
\begin{eqnarray}\label{OnestepEuler3}
&&\left|\E \left( \left[g\left(X_{T_{k+1}}^{T_k,y,\eta}\right)- g\left(\tX_{T_{k+1}}^{T_k,y,\hmu}\right)\right] I^{T_k,T_{k+1}}_{(0,l)} \right) \right|\\
&& \quad \leq C||g||_{3,\infty}\Delta_{T_{k+1}} \sum_{i=0}^{d} \int_{T_k}^{T_{k+1}} \int_{T_k}^{t} \left|\la\eta_s,\varphi_i\ra-\sum_{p=0}^{q-1} [(t-T_k)^p/p!]\la \eta_{T_k}, (\mL^\eta)^p\varphi_i\ra\right| ds dt \notag\\
&& \quad + C||g||_{3,\infty}\sum_{i=0}^{d}\sum_{p=0}^{q-1} \Delta_{T_{k+1}}^{p+2}\left|\la \eta_{T_k}, (\mL^\eta)^p\varphi_i\ra - \la \hmu_{T_k},(\mL^{\hmu})^{p}\varphi_i\ra\right|.\notag
\end{eqnarray}
\end{lemme}

\begin{lemme}\label{OnestepCub}
Let $ n \geq 1$, $g$ be a $C^{n+2}_b$ function  from $\R^d$ to $\R$ and $\Q$ be a cubature measure of order $n$. Then, there exist constants $C, C'$ depending only on $V_{0:d},d,n$ such that for all $i=1,\ldots, d$, for all $k=1,\cdots,N-1$:
\begin{eqnarray}
&&\left|(\tilde{P}^{\eta}_{T_k,T_{k+1}} -\tilde{Q}^{\eta}_{T_k,T_{k+1}})g(y) \right| = \left|(\E-\E_{\Q} )\left[g(\tilde{X}_{T_{k+1}}^{T_k,y,\eta})\right] \right|\leq C \sum_{l=n+1}^{n+2} ||g||_{l,\infty} \Delta_{T_{k+1}}^{(l)/2} \label{Eq:OnestepCub1}\\
&&\left|(\E-\E_{\Q})\left[g(\tilde{X}_{T_{k+1}}^{T_k,y,\eta}) I_{(i)}^{T_k,T_{k+1}} \right] \right|\leq C \sum_{l=n}^{n+1} ||g||_{l,\infty} \Delta_{T_{k+1}}^{(l+1)/2} \label{Eq:OnestepCub2}\\
&& \left|(\E-\E_{\Q})\left[g(\tilde{X}_{T_{k+1}}^{T_k,y,\eta}) I_{(0,i)}^{T_k,T_{k+1}} \right] \right|\leq C' \sum_{l=n-2}^{n-1} ||g||_{l,\infty} \Delta_{T_{k+1}}^{(l+3)/2}, \label{Eq:OnestepCub3}
\end{eqnarray}
for all $y \in \R^d$.\\
\end{lemme}

\begin{lemme} \label{Lem:OneStepAll}
Let $ n \geq 1$ and $g$ be a $C^{n+2}_b$ function from $\RR^d$ to $\RR$. Let $\Q$ be a cubature measure of order $n$, and $\hat{X}$ be the associated cubature tree. Then, there exists a constant  $C$ depending only on $d$, $q$, $V_{0:d},n,T,||\varphi_{0:d}||_{2q+n+2,\infty}$ such that, for all $k=1,\cdots,N-1$:
\begin{eqnarray}
 \left|\E \left[g\left(X_{T_{k+1}}^{T_k,y,\mu}\right)\right]-  \E_{\Q_{T_k,T_{k+1}}} \left[g\left(X_{T_{k+1}}^{T_k,y,\hmu}\right)\right] \right| &\leq & C ||g||_{2,\infty} \left[ \Delta_{T_{k+1}}^{q+1}+ \Delta_{T_{k+1}} N^{- [ (n-1) \wedge 2q]/2} \right]\notag\\
 && \  + C||g||_{n+1,\infty} \Delta_{T_{k+1}}^{(n+1)/2} + C||g||_{n+2,\infty} \Delta_{T_{k+1}}^{(n+2)/2},\label{Eq:Lem:OneStepAll1}
\end{eqnarray} 
\begin{eqnarray}
\left|\E \left[\left(g\left(X_{T_{k+1}}^{T_k,y,\mu}\right)- g\left(\hX_{T_{k+1}}^{T_k,y,\hmu}\right)\right) I_{(l)}^{T_k,T_{k+1}}\right]\right| &\leq &C ||g||_{3,\infty} \left[ \Delta_{T_{k+1}}^{q+1}+ \Delta_{T_{k+1}}^2 N^{- [ (n-1) \wedge 2q]/2} \right]  \notag \\
&& \quad + C||g||_{n,\infty} \Delta_{T_{k+1}}^{(n+1)/2} + C||g||_{n+1,\infty} \Delta_{T_{k+1}}^{(n+2)/2} \label{Eq:Lem:OneStepAll2}
\end{eqnarray}
and
\begin{eqnarray}
\left|\E \left[\left(g\left(X_{T_{k+1}}^{T_k,y,\mu}\right)-  g\left(\hX_{T_{k+1}}^{T_k,y,\hmu}\right)\right) I_{(0,l)}^{T_k,T_{k+1}}\right]\right| &\leq &C ||g||_{3,\infty} \left[ \Delta_{T_{k+1}}^{q+2}+ \Delta_{T_{k+1}}^3 N^{- [ (n-1) \wedge 2q]/2} \right]\notag\\
&& \quad+ C||g||_{n-2,\infty} \Delta_{T_{k+1}}^{(n+1)/2} + C||g||_{ n-1,\infty} \Delta_{T_{k+1}}^{(n+2)/2}.\label{Eq:Lem:OneStepAll3}
\end{eqnarray}
\end{lemme}


\subsection{Proofs of Lemmas \ref{OnestepEuler}, \ref{OnestepCub} and \ref{Lem:OneStepAll}}

\subsubsection{Proof of Lemma \ref{OnestepEuler}}
Let $k\in\{0,\cdots,N-1\}$, consider the PDE \eqref{PropPDE} with $g$ as boundary condition. It is clear that this PDE admits a unique solution $\tu$ and that there exists a positive constant $C(T ,V_{0:d})$ such that for every multi-index of space derivatives with $\beta\in \mc{A}_3$ 
\begin{eqnarray}\label{boundtu}
&&||\tu||_{\infty} + ||D_\beta \tu ||_{\infty} \leq C(T, V_{0:d}) ||g||_{||\beta||,\infty}.
\end{eqnarray}
Let us write:
\begin{equation*}
\left|(P_{T_k,T_{k+1}}^{\eta}-\tilde{P}_{T_k,T_{k+1}}^{\hmu}) g(y)\right| = \left|\E \left[g\left(X_{T_{k+1}}^{T_k,y,\eta}\right)- g\left(\tX_{T_{k+1}}^{T_k,y,\hmu}\right)\right]  \right|.
\end{equation*}
Now, since $\tilde{u}$ is the solution of \eqref{PropPDE} we have that
 \begin{eqnarray}
&&g\left(X_{T_{k+1}}^{T_k,y,\eta}\right)=\tu(T_k,y)+ \sum_{j=1}^d\int_{T_k}^{T_{k+1}} \mc{V}^{\eta}_{(j)}\tu(s, X_s^{T_k,y,\eta}) dB^j_s \label{repg1},
 \end{eqnarray}
by It\^{o}'s Formula and
\begin{eqnarray}
 g\left(\tX_{T_{k+1}}^{T_k,y,\hmu}\right)&=&\tu(T_k,y)+\int_{T_k}^{T_{k+1}}(\mL^{\eta}-\tilde{\mL}^{\hmu})\tu(s, \tX_s^{T_k,y,\hmu}) ds\\ \notag
 &&+ \sum_{j=1}^d\int_{T_k}^{T_{k+1}} {\mc{V}}^{\hmu}_{(j)}\tu(s, \tX_s^{T_k,y,\hmu}) dB^j_s \label{repg2}.
\end{eqnarray}
Therefore,
\begin{eqnarray}\label{ingrez}
\E  \left[g\left(X_{T_{k+1}}^{T_k,y,\eta}\right)- g\left(\tX_{T_{k+1}}^{T_k,y,\hmu}\right)\right]  = -\E\int_{T_k}^{T_{k+1}} \left(\mL^{\eta} - \tilde{\mL}^{\hmu}\right) \tu(t,\tX_{t}^{T_k,y,\hmu}) dt,
\end{eqnarray}
since $\p_t \tu = -\mL^{\eta}\tu$. As \eqref{boundtu} implies that $\tilde{u}$ and its two first derivatives are bounded, we may control the term above by the difference between the two generators, then by the difference between the frozen (in space of probability measure) $(\eta_t)_{T_K\leq t \leq T_{k+1}}$ and the approximate and frozen (in time) measure $\hmu_{T_k}$. Hence, taking into account the particular dependence on the measure in our framework, we deduce:

\begin{eqnarray*}
&&\left|\E  \left[g\left(X_{T_{k+1}}^{T_k,y,\eta}\right)- g\left(\tX_{T_{k+1}}^{T_k,y,\hmu}\right)\right] \right|\\
&&\quad \leq C||g||_{2,\infty} \int_{T_k}^{T_{k+1}} \bigg[\sum_{i=0}^{d} \bigg|\la \eta_t,\varphi_i\ra - \sum_{p=0}^{q-1} [(t-T_k)^p/p!]\la \eta_{T_k}, (\mL^\eta)^p \varphi_i\ra \bigg|\\
&& \hphantom{\quad \leq C||g||_{2,\infty} \int_{T_k}^{T_{k+1}}} +\sum_{i=0}^{d} \sum_{p=0}^{q-1} \Delta_{T_{k+1}}^p \left|\la \eta_{T_k}, (\mL^\eta)^p\varphi_i\ra - \la \hmu_{T_k}, (\mL^{\hmu})^p\varphi_i\ra \right|\bigg]dt\\
&&\quad \leq C(T,V_{0:d})||g||_{2,\infty}\int_{T_k}^{T_{k+1}} \sum_{i=0}^{d} \bigg|\la \eta_t,\varphi_i\ra-\sum_{p=0}^{q-1} [(t-T_k)^p/p!]\la \eta_{T_k}, (\mL^\eta)^p \varphi_i\ra \bigg|dt\\
&&\qquad +C(T,V_{0:d})||g||_{2,\infty} \sum_{i=0}^{d}\sum_{p=0}^{q-1} \Delta_{T_{k+1}}^p\Delta_{T_{k+1}}\left|\la\eta_{T_k},(\mL^{\eta})^p\varphi_i\ra - \la \hmu_{T_k},(\mL^{\hmu})^p\varphi_i\ra \right|
\end{eqnarray*}
This concludes the proof of the first assertion. Now, we deduce from \eqref{repg1} and \eqref{repg2} and integration by parts, that:

\begin{eqnarray}
&&\E \left( \left[g\left(X_{T_{k+1}}^{T_k,y,\eta}\right)- g\left(\tX_{T_{k+1}}^{T_k,y,\hmu}\right)\right] I^{T_k,T_{k+1}}_{(l)} \right) \label{Eq:second_exp} \\
&&  = \E \int_{T_k}^{T_{k+1}}  \left(\mL^{\eta} - \tilde{\mL}^{\hmu}\right) \tu(t,\tX_{t}^{T_k,y,\hmu}) I^{T_k,t}_{(l)}dt \notag\\
&& \quad + \E \int_{T_k}^{T_{k+1}}  \left[\mc{V}_{(l)}(t,X_{t}^{T_k,y,\eta},\la \eta_t,\varphi_l\ra)-\mc{V}_{(l)}(t,X_{t}^{T_k,y,\eta},\la\hmu_{T_k},\varphi_l\ra )\right]\tu\left(t,X_{t}^{T_k,y,\eta}\right) dt \notag\\
&& \quad  + \frac{1}{2} \E\int_{T_k}^{T_{k+1}} \left[ \mc{V}_{(l)}(t,X_{t}^{T_k,y,\eta},\la\hmu_{T_k},\varphi_l\ra) \tu(t,X_{t}^{T_k,y,\eta})-\mc{V}_{(l)}(t,\tX_{t}^{T_k,y,\hmu},\la\hmu_{T_k},\varphi_l\ra) \tu(t,\tX_{t}^{T_k,y,\hmu})\right] dt. \notag
\end{eqnarray}
First, note that the first two terms in the right hand side above may be controlled by the difference between the coefficients times the supremum norm of the first and second order derivatives of $\tu$ times the order of the integrals, as we did for \eqref{ingrez}. Second, note that the first assertion of Lemma \ref{OnestepEuler} can be applied to the function $\mc{V}_{(l)}(t,.,\la \hmu_{T_k},\varphi_l\ra) \tu(t,.)$ in the last term on the right hand side above. These arguments, together with the bound \eqref{boundtu} lead to:

\begin{eqnarray*}
&&\left|\E \left( \left[g\left(X_{T_{k+1}}^{T_k,y,\eta}\right)- g\left(\tX_{T_{k+1}}^{T_k,y,\hmu}\right)\right] I^{T_k,T_{k+1}}_{(l)} \right)\right| \\
&& \leq C(T,V_{0:d})\int_{T_k}^{T_{k+1}} (||g||_{1,\infty}+||g||_{2,\infty}(t-T_k)^{1/2}) \sum_{i=0}^{d} \bigg|\la\eta_t,\varphi_i\ra-\sum_{p=0}^{q-1} [(t-T_k)^p/p!]\la\eta_{T_k},(\mL^\eta)^p \varphi_i\ra \bigg|dt \notag \\
&&\quad +C(T,V_{0:d})\sum_{i=0}^{d}\sum_{p=0}^{q-1} \Delta_{T_{k+1}}^p\Delta_{T_{k+1}}(||g||_{1,\infty}+||g||_{2,\infty}\Delta_{T_{k+1}}^{1/2})\left|\la \eta_{T_k}, (\mL^\eta)^p\varphi_i\ra - \la \hmu_{T_k}, (\mL^{\hmu})^p\varphi_i\ra \right| \notag \\
&&\quad + C(T,V_{0:d})||g||_{3,\infty} \Bigg\{\int_{T_k}^{T_{k+1}} \int_{T_k}^{t} \sum_{i=0}^{d} \bigg|\la \eta_s,\varphi_i \ra-\sum_{p=0}^{q-1} [(s-T_k)^p/p!]\la \eta_{T_k}, (\mL^\eta)^p \varphi_i\ra \bigg|dsdt \notag \\
&&\quad \hphantom{+C(T,V_{0:d})||g||_{3,\infty}}+\sum_{i=0}^{d}\sum_{p=0}^{q-1} \Delta_{T_{k+1}}^p\Delta_{T_{k+1}}^2\left|\la \eta_{T_k},(\mL^\eta)^p\varphi_i\ra- \la\hmu_{T_k}, (\mL^{\hmu})^p\varphi_i\ra \right|\Bigg\}, \notag
\end{eqnarray*}
and this concludes the proof of the second assertion. Finally, \eqref{repg1} and \eqref{repg2} and integration by parts, give: 
\begin{eqnarray*}
&&\E \left( \left[g\left(X_{T_{k+1}}^{T_k,y,\eta}\right)- g\left(\tX_{T_{k+1}}^{T_k,y,\hmu}\right)\right] I^{T_k,T_{k+1}}_{(0,l)}  \right)\\
&& =  \E \int_{T_k}^{T_{k+1}}  \left(\mL^{\eta} - \tilde{\mL}^{\hmu}\right) \tu(t,\tX_{t}^{T_k,y,\eta}) I^{T_k,t}_{(0,l)}dt  \\
&& \quad + \E \int_{T_k}^{T_{k+1}}  \left[\mc{V}_{(l)}(t,X_{t}^{T_k,y,\eta},\la\eta_t,\varphi_l\ra)-\mc{V}_{(l)}(t,X_{t}^{T_k,y,\eta},\la\hmu_{T_k},\varphi_l\ra) \right] \tu\left(t,X_{t}^{T_k,y,\eta}\right) I^{T_k,t}_{(0)} dt   \\
&& \quad  +\frac{1}{2}\E\int_{T_k}^{T_{k+1}}  \left[\mc{V}_{(l)}(t,X_{t}^{T_k,y,\eta},\la \hmu_{T_k},\varphi_l\ra) \tu(t,X_{t}^{T_k,y,\eta})-\mc{V}_{(l)}(t,\tX_{t}^{T_k,y,\hmu},\la\hmu_{T_k},\varphi_l\ra) \tu(t,\tX_{t}^{T_k,y,\hmu}) \right]  I^{T_k,t}_{(0)} dt .
\end{eqnarray*}
Note the similarity with \eqref{Eq:second_exp}. So that a similar development gives
\begin{eqnarray*}
&&\left|\E \left( \left[g\left(X_{T_{k+1}}^{T_k,y,\eta}\right)- g\left(\tX_{T_{k+1}}^{T_k,y,\hmu}\right)\right] I^{T_k,T_{k+1}}_{(0,l)} \right)\right|\\
&& \leq C(T,V_{0:d})\int_{T_k}^{T_{k+1}} (||g||_{1,\infty}(t-T_k)+||g||_{2,\infty}(t-T_k)^{3/2}) \sum_{i=0}^{d} \bigg|\la\eta_t,\varphi_i\ra-\sum_{p=0}^{q-1} [(t-T_k)^p/p!]\la \eta_{T_k}, (\mL^{\eta})^p \varphi_i\ra \bigg|dt\\
&&\quad +C(T,V_{0:d}) \Delta_{T_{k+1}} \sum_{i=0}^{d}\sum_{p=0}^{q-1} \Delta_{T_{k+1}}^p\Delta_{T_{k+1}}(||g||_{1,\infty}+||g||_{2,\infty}\Delta_{T_{k+1}}^{1/2})\left|\la \eta_{T_k}, (\mL^\eta)^p\varphi_i\ra - \la\hmu_{T_k}, (\mL^{\hmu})^p\varphi_i\ra\right|\\
&&\quad + C(T,V_{0:d}) \Delta_{T_{k+1}} ||g||_{3,\infty} \Bigg\{\int_{T_k}^{T_{k+1}} \int_{T_k}^{t} \sum_{i=0}^{d} \bigg|\la \eta_s,\varphi_i\ra-\sum_{p=0}^{q-1} [(s-T_k)^p/p!]\la \eta_{T_k},(\mL^\eta)^p \varphi_i\ra \bigg|dsdt\\
&&\qquad + \sum_{i=0}^{d}\sum_{p=0}^{q-1} \Delta_{T_{k+1}}^p\Delta_{T_{k+1}}^2\left|\la \eta_{T_k},(\mL^\eta)^p\varphi_i\ra- \la\hmu_{T_k},(\mL^{\hmu})^p\varphi_i\ra\right|\Bigg\}
\end{eqnarray*}
from where the last claim is deduced.
\qed

\subsubsection{Proof of Lemma \ref{OnestepCub}}
Let $k\in\{0,\cdots,N-1\}$, once again, we consider the unique infinitely differentiable solution $\tu$ of PDE \eqref{PropPDE} with $g$ as boundary condition. Recall that for every $\beta\in\mc{M}$ there exists a positive constant $C(T ,V_{0:d})$ such that:

\begin{eqnarray}\label{boundtu2}
&&||\tu||_{\infty} + ||D_\beta \tu ||_{\infty} \leq C(T, V_{0:d}) ||g||_{||\beta||,\infty}.
\end{eqnarray}
The result then follows from Stratonovich-Taylor expansion of $(t,y) \mapsto \tu(t,y)$ around $(T_k,X_{T_k})$ by Theorem  5.6.1 in \cite{kloeden_numerical_1992} and bounding the remainder as in Proposition 2.1 of \cite{lyons_cubature_2004}.
\qed

\subsubsection{Proof of Lemma \ref{Lem:OneStepAll}}
Note that
\begin{eqnarray}\label{decompote}
\left|\E \left[g\left(X_{T_{k+1}}^{T_k,y,\mu}\right)\right]-  \E_{\Q_{T_k,T_{k+1}}} \left[g\left(X_{T_{k+1}}^{T_k,y,\hmu}\right)\right] \right|
 &=&\left|\E \left[g\left(X_{T_{k+1}}^{T_k,y,\mu}\right)\right]-  \E \left[g\left(\tX_{T_{k+1}}^{T_k,y,\hmu}\right)\right] \right|\\
 && \quad + \left|\E \left[g\left(\tX_{T_{k+1}}^{T_k,y,\mu}\right)\right]-  \E_{\Q_{T_k,T_{k+1}}} \left[g\left(X_{T_{k+1}}^{T_k,y,\hmu}\right)\right] \right|.\notag
\end{eqnarray}
Combining estimate \eqref{OnestepEuler1} of Lemma \ref{OnestepEuler} with Claim \ref{c4} and \eqref{Eq:ErrorforSB} in Theorem \ref{MR}, we get that the first term in the right hand side is bounded by:
$$C||g||_{2,\infty} \left[ \Delta_{T_{k+1}}^{q+1}+ \Delta_{T_{k+1}} N^{- [ (n-1) \wedge 2q]/2} \right].$$
The second term in the right hand side of \eqref{decompote} can be estimated by combining this bound with the estimate \eqref{Eq:OnestepCub1} in \ref{OnestepCub} (when choosing $\eta = \mu$).

The other assertion follows from the same procedure, substituting \eqref{OnestepEuler2} (resp. \eqref{OnestepEuler3}) to \eqref{OnestepEuler1} and \eqref{Eq:OnestepCub2} (resp. \eqref{Eq:OnestepCub3}) to \eqref{Eq:OnestepCub1}.

\section{Proofs of Corollary \ref{Lipcase} and \ref{wasscase_backward}}\label{Sec:Lipschitz}

\subsection{Proof of Corollary \ref{Lipcase}}
Many practical applications, particularly in finance, require the algorithm to be able to solve problems in which the boundary condition $\phi$ is less regular, e.g. when $\phi$ is just Lipschitz. In this section, we prove how the results obtained in the regular case extend to the case when assumption \textbf{(LB)} holds as Corollary \ref{Lipcase} state.

\textbf{A preliminary result.} We use in addition an auxiliary result shown in the proof of Theorem 8 in \cite{crisan_convergence_2007}:
\begin{lemme}\label{c5}
There exists a positive constant $C$ such that:
\begin{equation*}
\sum_{j=0}^{N-2} \Delta_{T_{j+1}}^{(m+1)/2} (T-T_j)^{-m/2} \leq C L(\gamma,m),
\end{equation*}
where $L$ is defined in \eqref{Eq:MFunction}.
\end{lemme}

We are now ready to examine the error convergence for the forward and backward components of the algorithm.\\

\textbf{Proof of the forward approximation in Corollary \ref{Lipcase}}

The regularity of the solution of the linear associated linear PDE is essential to our analysis. We start by stating a result in this sense under \textbf{(LB)}. This is summarized by
\begin{claim}\label{PropPDELip}
 Under \textbf{(LB)}, there exists a unique solution $\psi$ to the PDE \eqref{PDE} and for every multi-index $\beta\in \mc{M}$ there exists a constant $C$ such that:
\begin{equation}\label{Eq:Bound_derivatives_linear_Lip_u}
||D_{\beta} \psi(t,\cdot)||_\infty  \leq C(T-t)^{-(||\beta||-1)/2}
\end{equation}
\end{claim}
\begin{proof}
This follows from classical results of parabolic equations with parameter, see Chapter 9, Section 3  of \cite{friedman_partial_2008}.\\
\end{proof}

Thanks to the uniform ellipticity assumption, even if the terminal condition is not differentiable, we know that the solution of the PDE \eqref{PDEfirst} is smooth except at the boundary. Precisely, the gradient bounds \eqref{Eq:GradBoundsmooth} are now given by
\begin{equation}\label{smoothlip}
||\nabla_y^n\psi(t,\cdot)||_{\infty} \leq C(T,V_{0:d})||\phi||_{1,\infty} (T-t)^{(1-n)/2},
\end{equation}
where $\psi$ is defined in \eqref{Eq:DefPsi}. With this in hand, we can follow the proof exactly as the one of the corresponding forward part in Theorem \ref{MR} up to estimate \eqref{Eq:AlmostFinalErrorA} but where we separate the error on the last step, since there is no smoothing effect there. Then, plugging estimate \eqref{smoothlip} in \eqref{Eq:AlmostFinalErrorA} instead of \eqref{Eq:GradBoundsmooth}, we get:
\begin{eqnarray*}
&&\left|(P_{T_0,T_N} - Q_{T_0,T_N}^{\hmu}) \phi(x)\right|  \\
&& \leq C(T,V_{0:d}) ||\phi||_{1,\infty} \sum_{j=0}^{N-2}\sum_{p=0}^{q-1} \Delta_{T_{j+1}}^{p+1}(T-T_{j+1})^{-1/2} \sum_{i=0}^{d} \left|(P_{T_0,T_N} - Q_{T_0,T_N}^{\hmu}) \varphi_i(x)\right| \notag\\
&& \quad + C(T,V_{0:d},d) ||\phi||_{1,\infty} \|\varphi\|_{2q,\infty}\sum_{j=0}^{N-2}\Delta_{T_{j+1}}^{q+1} (T-T_{j+1})^{-1/2} \notag\\
&& \quad +  C(V_{0:d},d,m) ||\phi||_{1,\infty} \sum_{j=0}^{N-2}\sum_{l=m+1}^{m+2}\Delta_{T_{j+1}}^{\frac{l}{2}}(T-T_{j+1})^{(1-l)/2} +\left|(P_{T_{0},T_N} - P^{\eta}_{T_{0},T_N})\phi(x)\right|.\notag
\end{eqnarray*}
We conclude the proof by using Lemma \ref{c5} on the sums and by combining Lipschitz property of $\phi$ and adapted time-step on the last step error.\qed\\

\textbf{Proof of the backward approximation in Corollary \ref{Lipcase}}\\

Just as in the forward case, our analysis relies on the regularization properties of the associated non-linear PDE under \textbf{(LB)}. We have
\begin{claim}\label{EllipticNonLinearPDE}
Under \textbf{(LB)}, there exists a unique solution $u$ of \eqref{Eq:AssociatedPDE}, for all $(t,y)\in [0,T]\times \R^d$, it is given by
$$u(t,y)=Y_t^{t,y,\mu},$$ 
where $Y_t^{t,y,\mu}$ is defined in \eqref{cFBSDE}. Moreover, for $\phi$ Lipschitz and bounded, for every multi-index $\beta\in \mc{M}$ there exists a positive constant $C$ depending on the regularity of $V_{0:d},f,\varphi$ and $T$ such that:
\begin{equation}
|| D_\beta u(t,\cdot) ||_\infty  \leq C ||\phi||^{|\beta|}_{1,\infty} (T-t)^{-(||\beta||-1)/2}
\end{equation}
\end{claim}
\begin{proof}
To prove Claim \ref{EllipticNonLinearPDE}, we follow the same arguments given for Lemma \ref{RegNonLinearPDE}. First, due to the regularity properties of the diffusion under the elliptic case, we have similar properties as those used in the paper of Crisan and Delarue \cite{crisan_sharp_2012}, even in the non-homogeneous case (notably, the integration by parts property as shown in \cite{ma_representation_2002}). Hence, we get the control on derivatives result for space derivatives, and extend it, as before, to time derivatives. 
\end{proof}

Armed with the regularity of the function $u$, we can repeat the proof of the backward approximation in Theorem \eqref{MR}. We recover \eqref{Eq:propa1} for the first order scheme  and \eqref{Eq:propa2} for the second order scheme, i.e.
\begin{align}
&\mE^1_u(k)  \leq  \left(1+C\Delta_{T_{k+1}}\right)\mE^1_u(k+1) + \bar{\epsilon}(k+1),\label{Eq:propa1b}\\
&\mE_u^2(k) + \Delta_{T_{k+1}} \mE_f(k)  \leq  (1+C\Delta_{T_{k+1}}) \left[\mc{E}^2_u(k+1) + \Delta_{T_k}\mc{E}_f (k+1)\right] +C' ( \bar{\epsilon}_2(k) + \Delta_{T_{k+1}}N^{-[(m-1)\wedge q]/2}), \label{Eq:propa2b}
\end{align}
where $\bar{\epsilon}, \bar{\epsilon}_2$ are respectively defined in \eqref{Eq:barepsilon}, \eqref{Eq:DefEpsBar2}. Now, if we show that
\begin{equation}
 \sum_{k=0}^{N-1} \bar{\epsilon}(k) \leq N^{-1} ; \quad \text{and} \quad 
 \sum_{k=0}^{N-2} \bar{\epsilon}_2(k)  \leq N^{-2}; \label{Ineq:lipschitz}
\end{equation}
then, as in the smooth setting, we can apply Gronwall lemma on \eqref{Eq:propa1b} and \eqref{Eq:propa2b} and conclude on the desired rates of convergence for the approximation of $u$. The arguments for the rate of the approximation of $v$ are exactly as in the smooth setting thus completing the proof of the claimed result.

Therefore, we only need to prove \eqref{Ineq:lipschitz}.  But, Claim \ref{EllipticNonLinearPDE} and the definition of $M_u$ given in \eqref{Eq:DefinitionM} imply
$$M_u(n,T_{k+1}) \leq (T-T_{k+1})^{(1-n)/2}||\phi||_{1,\infty}^n,$$
which together with Claim \ref{EllipticNonLinearPDE} show that under the Lipschitz boundary setup,  

\begin{eqnarray}
\bar{\epsilon}(k+1) &\leq &  C\bigg( (T-T_{k})^{-1/2}\Delta_{T_{k+1}} \left[  \Delta_{T_{k+1}}^{3} + \Delta_{T_{k+1}}^q+N^{- [ (m-1) \wedge 2q]/2} \right]\notag\\
&&\quad + (T-T_{k})^{-m/2} \Delta_{T_{k+1}}^{(m+1)/2} + (T-T_{k})^{-(m+1)/2}\Delta_{T_{k+1}}^{(m+2)/2}\notag\\
&&\quad +  (T-T_{k})^{-1}\Delta_{T_{k+1}}^{3/2} \left[\Delta_{T_{k+1}}^{1/2} + \Delta_{T_{k+1}}^{q-1/2}+ \Delta_{T_{k+1}}^{1/2} N^{- [ (m-1) \wedge 2q]/2} \right]\bigg).\notag\\
&\leq&  C\bigg( (T-T_{k})^{-1/2}\Delta_{T_{k+1}} N^{-1} + (T-T_{k})^{-m/2} \Delta_{T_{k+1}}^{(m+1)/2} \label{Ineq:epsilon_t_lips}\\
&& \quad + (T-T_{k})^{-(m+1)/2}\Delta_{T_{k+1}}^{(m+2)/2}+  (T-T_{k})^{-1}\Delta_{T_{k+1}}^{3/2} N^{-1/2}\bigg);\notag
\end{eqnarray}
where we have used the fact that $\Delta_{T_k}\leq CN^{-1}$ even on the decreasing discretization. We can proceed similarly for $\bar{\epsilon}_2$ from inequality \eqref{Inq:eps2k} , to get

\begin{align}
\bar{\epsilon}_2(k) \leq & C'' \bigg( (T-T_{k})^{-3/2} \Delta^3_{T_{k+1}}  +  (T-T_{k})^{-3/2} \left[ \Delta_{T_{k+1}}^{q+1}+ \Delta_{T_{k+1}}^2 N^{- [ (m-1) \wedge 2q]/2} \right] \label{Ineq:epsilon_t_lips2}\\
 &\quad + \sum_{i=m-2}^{m+3} (T-T_{k})^{-(i-1)/2} \Delta_{T_{k+1}}^{(i+1)/2} \bigg).\notag
\end{align}

Then, \eqref{Ineq:lipschitz} follows by applying Lemma \ref{c5} to \eqref{Ineq:epsilon_t_lips} and \eqref{Ineq:epsilon_t_lips2}
\qed
\subsection{Proof of Corollary \ref{wasscase_backward}}
On a first hand, by following the proof of \eqref{Eq:ErrorforSB} in Theorem \ref{MR} we get \eqref{forwass}, where the difference between the integral of the $\varphi_i,\ i=1,\ldots,d$ against the measures in the right hand side are replaced by the distance $d_\mF$. Since for all $T_k<t<T$ $d_\mF(\mu_t,\mu_{T_k})  \leq C(t-T_k)$ (resp. $(t-T_k)^{1/2}$) in the case \eqref{wasserror1} (resp. \eqref{wasserror2}), the result follows from Gronwall's Lemma. This gives the rate of approximation of the law of the forward process. 

On a second hand, the backward errors are then obtained by the same arguments already developed in the proof of \eqref{Eq:rate1} in Theorem \ref{MR}, using the new forward approximation \eqref{wasserror1} (resp. \eqref{wasserror2}) instead of \eqref{Eq:ErrorforSB} (resp. \eqref{Eq:ErrorforLB}) in the proofs of Claims \ref{Cl:back1-2} and \ref{Cl:back1-3}.

\begin{center}
\textbf{ Acknowledgements}
\end{center}

The authors would like to thank the CNRS/Royal Society project `Promising" for partial financial support during the preparation of this work. We thank as well the members of the project for the discussions on the treated topics. Finally, we would like to thank François Delarue for his valuable suggestions and careful reading of the paper.

\bibliographystyle{amsalpha}

\bibliography{Bib_MV-FBSDE}

\end{document}